\documentclass[12pt]{article}

\usepackage[margin=3cm]{geometry}

\usepackage{lmodern}
\usepackage[T1]{fontenc}

\usepackage[english]{babel} 
\usepackage[utf8]{inputenc}

\usepackage{mathrsfs} 
\usepackage{amsfonts}
\usepackage{marvosym} 
\usepackage{amsmath,amssymb,amsthm}

\usepackage[usenames, dvipsnames, svgnames]{xcolor}
\definecolor{vert}{RGB}{15,120,5}
\definecolor{gris}{RGB}{128,128,128}
\definecolor{bleu}{RGB}{0,50,150}
\definecolor{rouge}{RGB}{149,24,24}
\usepackage[colorlinks = true,linkcolor=bleu,citecolor=vert]{hyperref}

\usepackage[nameinlink]{cleveref} 
\crefname{equation}{}{}

\usepackage{epsfig}

\usepackage{tikz-cd} 

\usepackage{stmaryrd}
\usepackage{manfnt}
\usepackage{multicol}
\usepackage{array}
\usepackage{multirow}

\usepackage{fancyhdr} 
\usepackage{fancybox} 

\newcommand{\titre}{}
\newcommand{\auteur}{}

\usepackage{calligra}
\title{\titre}
\author{\auteur}

\usepackage{tikz}
\usetikzlibrary{calc, intersections}

\usepackage{adjustbox}

\numberwithin{equation}{subsubsection}

\theoremstyle{plain}
\newtheorem{thm}{Theorem}[section]
\newtheorem{prop}[thm]{Proposition}

\newtheorem{nota}[thm]{Notation}

\newtheorem{lem}[thm]{Lemma}
\newtheorem{cor}[thm]{Corollary} 
\newtheorem{conj}[thm]{Conjecture}
\theoremstyle{definition}
\newtheorem{defi}[thm]{Definition}
\newtheorem{constr}[thm]{Construction}
\theoremstyle{remark}

\newtheorem{rem}[thm]{Remark}

\numberwithin{equation}{thm}

\newcommand{\C}{\mathbb{C}}

\newcommand{\Q}{\mathbb{Q}}
\newcommand{\Z}{\mathbb{Z}}
\newcommand{\N}{\mathbb{N}}

\DeclareMathOperator{\Spec}{\mathrm{Spec}}
\newcommand{\A}{\mathbb{A}}

\newcommand{\HH}{\mathrm{H}}
\newcommand{\D}{\mathrm{D}}

\newcommand{\oscr}{\mathscr{O}}

\newcommand{\acal}{\mathcal{A}}

\newcommand{\bcal}{\mathcal{B}}

\newcommand{\fcal}{\mathcal{F}}

\newcommand{\ccal}{\mathcal{C}}

\newcommand{\dcal}{\mathcal{D}}
\newcommand{\mcal}{\mathcal{M}}
\newcommand{\rcal}{\mathcal{R}}

\newcommand{\Dd}{\mathrm{D}}
\newcommand{\Ind}{\mathrm{Ind}}

\newcommand{\hscr}{\mathscr{H}}

\newcommand{\colim}{\mathrm{colim}}

\newcommand{\Hom}{\mathrm{Hom}}

\newcommand{\Gm}{{\mathbb{G}_m}}

\newcommand{\Db}{\D^b_c}

\newcommand{\kcal}{\mathcal{K}}

\newcommand{\Mc}{{\mathcal{M}_{\mathrm{ct}}}}
\newcommand{\Mp}{{\mathcal{M}_{\mathrm{perv}}}}

\newcommand{\coker}{\mathrm{coker}\ }
\DeclareMathOperator{\sHom}{\mathscr{H}\text{\kern -3pt {\calligra\large om}}\,}

\newcommand{\HHp}{{\ ^{\mathrm{p}}\HH}}

\newcommand{\dnori}{{\dcal^b_\mcal}}

\newcommand{\Sch}{\mathrm{Sch}}
\newcommand{\op}{\mathrm{op}}

\newcommand{\nscr}{\mathscr{N}}

\newcommand{\catinfty}{\mathrm{Cat}_\infty}

\newcommand{\res}[2]{{#1}_{\mid {#2}}}

\makeatletter
\newcommand{\subjclass}[2][2020]{%
  \let\@oldtitle\@title%
  \gdef\@title{\@oldtitle\footnotetext{#1 \emph{Mathematics subject classification.} #2}}%
}
\newcommand{\keywords}[1]{%
  \let\@@oldtitle\@title%
  \gdef\@title{\@@oldtitle\footnotetext{\emph{Key words and phrases.} #1.}}%
}
\makeatother

\begin{document}
\title{On the Nori and Hodge realisations\\ of Voevodsky motives}
\date{}
\author{Swann Tubach}
\maketitle

\begin{abstract}
	We show that the derived categories of perverse Nori motives and mixed Hodge modules are the derived categories of their constructible
	hearts.
	This enables us to construct $\infty$-categorical lifts of the six operations. As a result, we obtain realisation functors
	from the category of Voevodsky étale motives to the derived categories of perverse Nori motives and mixed Hodge modules that commute
	with the operations.
	We also prove that if a motivic t-structure exists then Voevodsky étale motives and the derived category of perverse Nori motives are
	equivalent.
	Finally, we give a presentation of the indization
	of the derived category of perverse Nori motives as a category of modules in Voevodsky étale motives.
\end{abstract}
\tableofcontents

\section*{Introduction}

Let $k$ be a field of characteristic zero. Following the vision of Beilinson, Deligne, Grothendieck
 and others there should exist an abelian category 
of mixed motives $\mathcal{MM}(k)$, target of the universal cohomology theory
$$M^*:\mathrm{Var}_k^\op\to\mathcal{MM}(k)$$
on $k$-varieties in the sense that any other reasonable cohomology theory on $k$-varieties would factor uniquely through $M^*$. This is out of reach 
of the current technology. However two constructions have almost all the required properties to provide a category of mixed motives. The first one 
is Voevodsky, Levine, Hanamura and others' triangulated category of geometric motives $\mathrm{DM}_\mathrm{gm}(k,\Q)$ (\cite{MR1764197}), a candidate for the bounded derived 
category of mixed motives. The second is the abelian category of Nori motives $\mcal_\mathrm{Nori}(k)$ (\cite{fakhruddinNotesNoriLectures2000}), a candidate for $\mathcal{MM}(k)$. These categories have realisation functors 
factoring the known cohomology theories of varieties. Indeed, if $k$ is a subfield of $\C$ Huber (\cite{MR1775312}), Levine (\cite{MR2181828}) and 
Nori have constructed a 
Hodge realisation functor 
$$\mathrm{DM}_\mathrm{gm}( k,\Q)\to \Dd^b(\mathrm{MHS}^p_\Q)$$
to the derived category 
of polarisable rational mixed Hodge structures 
 that realises the 
Hodge cohomology of varieties and factors through the derived category of the $\Q$-linear abelian category of motives constructed by Nori 
thanks to the work of Harrer (\cite{harrerComparisonCategoriesMotives2016}) and Choudhury--Gallauer (\cite{MR3649230}). When $k$ is arbitrary they also constructed a $\ell$-adic realisation functor with values in the derived category of $\ell$-adic Galois representations.

There exist relative versions of the categories of Voevodsky and Nori. The triangulated category of
 constructible rational étale motives $\mathrm{DM}^{\mathrm{\'et}}_c(-)$ constructed by Ayoub (\cite{MR3205601}) and
Cisinski and Déglise (\cite{MR3477640}) gives a triangulated category of étale motivic sheaves generalising Voevodsky's category, and the abelian category 
of perverse Nori motives $\Mp(-)$ constructed by Ivorra and S. Morel (\cite{ivorraFourOperationsPerverse2022}) gives a 
category of motivic sheaves generalising the abelian category of Nori. 
Both settings have a 6 functors formalism, the tensor product on perverse Nori motives being constructed by Terenzi (\cite{terenziTensorStructurePerverse2024}). Note that another approach based on constructible sheaves and using Nori's ideas has been considered by Arapura in \cite{MR2995668} and \cite{MR4568787} but his construction does not include all 6 operations.
 If $X$ is a 
quasi-projective smooth $k$-variety and $k$ is a subfield of $\C$, Ivorra (\cite{MR3518311}) constructed a realisation functor 
$$\mathrm{DM}^\mathrm{\'et}_c(X)\to \Dd^b(\mathrm{MHM}(X))$$
 to Saito's derived category of mixed Hodge modules (\cite{MR1047415}). It also factors through the derived category of perverse Nori motives and computes 
the relative Hodge homology of $X$-schemes.
 Unfortunately, Ivorra's functor 
has no obvious compatibility 
with the 6 operations that exist on $\mathrm{DM}^\mathrm{\'et}_c(-)$ and $\Dd^b(\mathrm{MHM}(-))$. The reason behind this difficulty is 
that the functor is built
 by constructing an explicit complex $K^\bullet\in\mathrm{Ch}^b(\mathrm{MHM}(X))$ that computes
  the relative homology sheaf $f_!f^!\Q_X\in\Dd^b(\mathrm{MHM}(X))$ of a smooth affine $X$-scheme $f:Y\to X$, and that we do not know
   how to lift the
  6 operations explicitly on complexes: a priori they are only defined over the derived category.
  
  On the other hand, we know since 
  Robalo's thesis \cite{MR3281141} that
  the $\infty$-category $\mathcal{DM}^{\mathrm{\'et}}_c(X)$ (which is a stable $\infty$-category lifting $\mathrm{DM}^\mathrm{\'et}_c(X)$) has an universal 
  property that enables one to construct realisations functors to any stable $\infty$-category that has reasonable properties. Indeed, he shows 
  that $\mathcal{DM}^{\mathrm{\'et}}(X)=\Ind\mathcal{DM}^{\mathrm{\'et}}_c(X)$ is the target of the universal symmetric monoidal functor from $\mathrm{Sm}_X$ to a stable $\Q$-linear 
  presentably symmetric monoidal $\infty$-category, satisfying $\A^1$-invariance, 
  étale hyperdescent and $\mathbb{P}^1$-stability. The work of Drew and Gallauer (\cite{MR4560376}) shows that this 
  universal property works very well in families, so that any contravariant functor $\ccal$ from finite type $k$ schemes to symmetric monoidal $\Q$-linear presentable $\infty$-categories that
  satisfies non-effective étale descent and $\A^1$-invariance, together with $\mathbb{P}^1$-stability and smooth base change, receives a natural transformation from $\mathcal{DM}^\mathrm{\'et}(-)$ compatible with smooth base change functoriality.
  Moreover Cisinski and Déglise (see also Ayoub's \cite[Scholie 1.4.2]{MR2423375}) 
  proved in \cite[Theorem 4.4.25]{MR3971240} that the full subcategory of geometric 
  objects $\ccal_{\mathrm{gm}}$ of any such functor $\ccal$ has the six operations, and that any 
  natural transformation $\mathcal{DM}^{\mathrm{\'et}}\to\ccal$ between such functors would 
  induce a natural transformation between the categories of geometric objects that commutes with the six operations.
  In other words, to construct a family of realisation functors $\mathrm{Hdg}^*:\mathcal{DM}^{\mathrm{\'et}}_c(X)\to\dcal^b(\mathrm{MHM}(X))$ compatible with the 6 operations, 
  it suffices to lift the 6 operations on mixed Hodge modules to the $\infty$-categorical setting. This is what we do. 

  There is another approach to the Hodge realisation of étale motives considered by Brad Drew in \cite{drewMotivicHodgeModules2018}. For each finite type $k$-scheme $X$, he constructs a new category $\mathcal{DH}_c(X)$ of \emph{motivic Hodge modules} that has a tautological functor $\mathcal{DM}^{\mathrm{\'et}}_c(X)\to \mathcal{DH}_c(X)$ commuting with all the operations. If $X = \Spec k$ is the spectrum of a field, then $\mathcal{DH}_c(\Spec k)\hookrightarrow \dcal^b(\mathrm{MHS}^p_k)$ embeds in the derived category of mixed Hodge structures. Moreover one can compute Deligne cohomology as a $\Hom$-group in $\mathcal{DH}_c(X)$. However it was unclear how to relate $\mathcal{DH}_c(X)$ with $\dcal^b(\mathrm{MHM}(X))$ as it is hard to construct an adequate t-structure on $\mathcal{DH}_c(X)$. Using our result on the existence of the Hodge realisation and a description of geometric mixed Hodge modules as modules in $\mathcal{DM}^{\mathrm{\'et}}(X)$ (see \Cref{Norimodules} for the case of Nori motives), it should not be too hard to prove that Drew's category $\mathcal{DH}_c(X)$ embeds fully faithfully in the derived category of mixed Hodge modules. We plan to pursue this in a sequel to this paper.
  
Our main result is the following:
\begin{thm}[\Cref{NoriPB}, \Cref{realNori} and \Cref{weightexact}]

The 6 operations on the categories $\Dd^b(\mathrm{MHM}(-))$ and  $\Dd^b(\Mp(-))$ constructed by Saito, Ivorra, S. Morel and Terenzi admit $\infty$-categorical lifts and are defined over any finite type $k$-scheme. 
For every finite type $k$-scheme $X$
there exists an $\infty$-functor
\[\mathrm{Nor}^*: \mathcal{DM}^{\mathrm{\'et}}_c(X)\to\dcal^b(\Mp(X))\] and if $k$ is a subfield of $\C$ there exist $\infty$-functors
\[R_H:\dcal^b(\Mp(X))\to \dcal^b(\mathrm{MHM}(X))\]
and 
\[\mathrm{Hdg}^*: \mathcal{DM}^{\mathrm{\'et}}_c(X)\to\dcal^b(\mathrm{MHM}(X)).\]
The three functors commute with the $6$ operations and we have $R_H\circ\mathrm{Nor}^*\simeq \mathrm{Hdg}^*$.
Moreover, $\mathrm{Nor}^*$ is compatible with the Betti and the $\ell$-adic realisation functors, $R_H$ is t-exact and all functors are compatible with weights.

\end{thm}

How to give $\infty$-categorical lifts of triangulated functors? In general this is a hard question. One general solution would be to use 
the formalism developed by Liu and Zheng in \cite{liuGluingRestrictedNerves2015}. In our particular case we have a simpler method. Indeed, we use that it is very easy to construct $\infty$-categorical lifts of 
derived functors. Most of the $6$ operations are not t-exact (even on one side) for the perverse t-structure. However, pullback functors are t-exact for 
the constructible t-structure, that is the t-structure whose heart behaves like the abelian category of constructible sheaves. 
In \cite{MR1940678}, Nori proves that the triangulated category of cohomologically constructible sheaves $\Dd^b_c(X(\C),\Q)$ on the $\C$-points of an 
algebraic variety $X$ over a subfield of $\C$ is the derived category of its constructible heart. We show that one can adapt his argument to mixed Hodge modules and perverse Nori motives, and 
obtain the following result:
\begin{thm}[\Cref{cccp}]
  Let $X$ be a quasi-projective $k$-variety.
Denote by $\mathrm{MHM}_c(X)$ (\emph{resp}. by $\Mc(X)$) the heart of the constructible t-structure
 on $\Dd^b(\mathrm{MHM}(X))$ (\emph{resp}. on $\Dd^b(\Mp(X))$). The canonical $\infty$-functors 
\[\mathrm{real}_{\mathrm{MHM}_c(X)}:\dcal^b(\mathrm{MHM}_c(X))\to\dcal^b(\mathrm{MHM}(X))\] 
if $k$ is a subfield of $\C$ and
\[\mathrm{real}_{\Mc(X)}:\dcal^b(\Mc(X))\to\dcal^b(\Mp(X))\] for an arbitrary field $k$
are equivalences of $\infty$-categories.
\end{thm}
This theorem enables us to make the handy change of variables $\dcal^b(\Mc(X))\simeq\dcal^b(\Mp(X))$. As pullback functors $f^*$ are easily seen to be $\infty$-functors this implies they are derived functors of
  their restrictions to the constructible hearts, so that the functoriality of the constructible heart, which can be written by hand, gives immediately the functoriality of the derived $\infty$-category! This consideration was the starting idea of this paper. 

   Let $\mathcal{DN}(X)=\Ind\dcal^b(\Mp(X))$ be the indization of the bounded derived $\infty$-category of $\Mp(X)$. This is a  compactly generated stable presentably symmetric 
  monoidal
   $\infty$-category whose category of compact objects is $\dcal^b(\Mp(X))$. The functor $\mathrm{Nor}^*:\mathcal{DM}^{\mathrm{\'et}}_c(X)\to\dcal^b(\Mp(X))$ extends to a functor 
   \[\mathrm{Nor}^*:\mathcal{DM}^{\mathrm{\'et}}(X)\to\mathcal{DN}(X)\] which preserves colimits, hence has a right adjoint $\mathrm{Nor}_*$. For each finite type $k$-scheme, set 
   $\nscr_X:=\mathrm{Nor}_*\mathrm{Nor}^*\Q_X\in\mathcal{DM}^{\mathrm{\'et}}(X)$. As $\mathrm{Nor}_*$ is automatically lax symmetric monoidal, we have that $\nscr_X$ is 
   an $\mathbb{E}_\infty$-algebra of $\mathcal{DM}^{\mathrm{\'et}}(X)$. Using the proof of the same result for the Betti realisation in \cite[Theorem 1.93]{ayoubAnabelianPresentationMotivic2022}, 
   we prove 
   \begin{prop}[\Cref{Norimodules}]
        For each scheme $X$ of finite type over some field $k$ of characteristic zero,
         denote by $p_X:X\to\Spec\Q$ the unique morphism. Then $p_X^*\nscr_\Q\simeq\nscr_X$ and the natural functor 
         \[\mathrm{Mod}_{\nscr_X}(\mathcal{DM}^{\mathrm{\'et}}(X))\to\mathcal{DN}(X)\] 
         is an equivalence of categories that commutes with pullback functors.
   \end{prop}
   Although we do not go into the details, the same arguments would prove that the $\infty$-category $\mathcal{DH}^\mathrm{geo}(X)\subset\Ind\dcal^b(\mathrm{MHM}(X))$ of objects of geometric origin in the indization 
   of the derived category of mixed Hodge modules is also the category of modules over some algebra 
   $\mathscr{H}_X$ in $\mathcal{DM}^{\mathrm{\'et}}(X)$ (See \cite{SwannMHM} for more details). This shows that Drew's constructions in \cite{drewMotivicHodgeModules2018} indeed provide an module presentation of objects of geometric origin in mixed Hodge modules (and even, by considering motives enriched in mixed Hodge structures, a presentation of a bigger full subcategory of all mixed Hodge modules as modules in a motivic category).

   As the functor $(\acal,\ccal)\mapsto\mathrm{Mod}_\acal(\ccal)$ that sends a symmetric monoidal $\infty$-category pointed at an algebra object $\acal\in\mathrm{CAlg}(\ccal)$ to the $\infty$-category 
   of modules over $\acal$ preserves colimits, we obtain the following corollary:
   \begin{cor}[\Cref{C0ctNori}]
        The $\infty$-functor $X\mapsto\dcal^b(\Mp(X))$ has an unique extension $\mathcal{DN}_c$ to quasi-compact and quasi-separated schemes of characteristic zero such that 
        for every projective system $(X_i)_i$ of such schemes with affine transitions morphisms the natural functor 
        \[ \colim_i\mathcal{DN}_c(X_i)\to\mathcal{DN}_c(\lim_i X_i)\] 
        is an equivalence of $\infty$-categories.
   \end{cor}

   Now that a well behaved comparison functor $\mathrm{Nor}^*:\mathcal{DM}^{\mathrm{\'et}}_c(X)\to\mathcal{DN}_c(X)$ is available, one wants to see to what 
   extend we can compare both categories. We have a result in that direction, 
   conditional to the existence of a motivic t-structure on $\mathrm{DM}_\mathrm{gm}(K)$ for every characteristic zero field $K$, a deep conjecture 
   that we know implies all standard conjectures in characteristic zero (\cite{MR2953406}). Moreover, by the work of Bondarko in \cite[Theorem 3.1.4]{MR3347995}, the existence of motivic $t$-structure for all fields of characteristic zero implies that  
   for all finite type $k$-schemes, we have a perverse and a constructible t-structure on $\mathcal{DM}^{\mathrm{\'et}}_c(X)$.  
\begin{thm}[\Cref{IfConjThenSame}]
	\label{IfConjThenSameIntro}
	Assume that a motivic t-structure exists for all fields of characteristic $0$. Let $X$ be a finite type $k$ scheme for such a field $k$.
	Then the heart of the perverse t-structure of $\mathcal{DM}^{\mathrm{\'et}}_c(X)$ is canonically equivalent to the category $\Mp(X)$ of perverse Nori motives. Moreover, the functor
	$\mathrm{Nor}^* : \mathcal{DM}^{\mathrm{\'et}}_c(X)\to\dcal^b(\Mp(X))$ is an equivalence of stable $\infty$-categories. This implies that $\mathcal{DM}^{\mathrm{\'et}}_c(X)$ is the derived category of both its perverse and its constructible heart.
\end{thm}
In particular, over a field we would have $\mathrm{DM}_\mathrm{gm}(k,\Q)\simeq\Dd^b(\mathcal{MM}(k))$, showing that the ``even more optimistic'' expectation 
of \cite[21.1.8]{MR2115000} is in fact no more optimistic than expecting the motivic t-structure to exist.

\section*{Organisation of the paper.}
After some recollections on motivic constructions in \Cref{SectionRecoll}, in  
\Cref{section1} we review Nori's proof in \cite{MR1940678} that the category of
 cohomologically constructible complexes of sheaves on the $\C$-points of a quasi-projective variety $X$ is equivalent to the derived 
 category of the abelian category of constructible sheaves and show that it gives a proof that the derived categories of the perverse and constructible hearts on the categories of mixed 
 Hodge modules (or of perverse Nori motives) are equivalent. The only new idea here was to replace 
 local systems with dualisable objects and cohomology of the global sections 
 with $\mathrm{Hom}_{\Dd^b(\Mp(X))}(\Q_X,-[q])$. The main point of the proof is showing effaceability of cohomology of the constant object on an affine $n$-space.

We divided \Cref{section2} in two parts. In the first one we prove general results about the functoriality of the functor $\dcal^b(-)$ sending an abelian category to 
its bounded derived $\infty$-category. 
 In the second part we prove that the $6$ operations on  mixed Hodge modules and on perverse Nori motives can be lifted to $\infty$-categorical functors of stable $\infty$-categories.

In \Cref{Sectionrealisation}, we use what has been proven on $\infty$-categorical enhancements to construct the realisation functors. We then prove that Nori motives are modules over an algebra in $\mathcal{DM}^{\mathrm{\'et}}(X)$. The arguments follow more or less Ayoub's proof of the similar statement for the Betti realisation in \cite{ayoubAnabelianPresentationMotivic2022}.
Finally we prove the result related to the  t-structure conjecture. The argument is the repeated use of the universal properties of $\mathcal{DM}^{\mathrm{\'et}}_c(X)$ and $\Mp(X)$ that force the Nori realisation functor to be an equivalence when $\mathcal{DM}^{\mathrm{\'et}}_c(X)$ has a perverse t-structure. 

\subsection*{Forthcoming and future work}
In \cite{SwannMHM}, using the $\infty$-categorical enhancement of mixed Hodge modules we will  extend them to Artin stacks, together with the operations, t-structures and weights.
In joint work with Raphaël Ruimy \cite{integralNori}
we construct an integral version of the derived category of perverse Nori motives. This gives an Abelian category of motivic sheaves with integral coefficients over any finite dimensional scheme of characteristic zero. 
\subsection*{Notations and conventions}

For the whole paper, $k$ is a field of characteristic $0$. When $k$ is a subfield of $\C$, if $X$ is a finite type $k$-scheme,
we denote by $\mathrm{D}_c^b(X,\Q)$ the category of constructible complexes of sheaves on $X(\C)$. By \cite{MR0923133}, it is also the derived category of
the abelian category of perverse sheaves $\mathrm{Perv}(X,\Q)$ defined in \cite{MR0751966}. For a general field $k$,
we denote by $\mathrm{D}^b_c(X,\Q_\ell)$ the category of constructible bounded complexes $\Q_\ell$-adic étale sheaves on $X$ as in \cite{MR4609461}.

We use the language of $\infty$-categories as developed in \cite{lurieKerodon} by Lurie, after Joyal's work on quasi-categories. We fix an universe and call small an $\infty$-category whose mapping spaces and core are in this universe. We denote by $\Delta$ the simplicial category. We denote by $\catinfty$ the $\infty$-category of
 small $\infty$-categories 	and by $\mathrm{Pr}^L$ the $\infty$-category $\infty$-category of presentable $\infty$-categories.
If $\ccal$ is a symmetric monoidal $\infty$-category we denote by $\mathrm{CAlg}(\ccal)$ the commutative algebras in $\ccal$, and if $A\in\ccal$ is a commutative algebra, we denote by $\mathrm{Mod}_A(\ccal)$ the category of modules over $A$ in $\ccal$. For example, the category of $H\Q$-modules in spectra $\mathrm{Mod}_\Q := \mathrm{Mod}_{H\Q}(\mathrm{Sp})$ is the unbounded derived category of the abelian category of $\Q$-vector spaces. 
We try to always put the index $\ccal$ when mentioning $\Hom$, as in $\Hom_\ccal(A,B)$. When $\ccal$ is an ordinary category, this is the $\Hom$-set, and when 
$\ccal$ is an $\infty$-category this is the $\pi_0$ of the mapping space $\mathrm{Map}_\ccal(A,B)$. In stable a $\infty$-category $\ccal$, the mapping spectra will be denoted by $\mathrm{map}_\ccal(-,-)$.

\subsection*{Acknowledgements} This work was written during my PhD supervised by Sophie Morel at the UMPA in Lyon. I would like to thank her
for the faith and constant attention she has for my work
and for the immense amount of time she is spending to accompany me in this thesis. I am grateful to the anonymous referee for their valuable comments and suggestions that helped make this paper more readable. I am much obliged to Joseph Ayoub for finding a mistake during a talk I gave, whose correction simplified the third part of this paper. I thank also Frédéric Déglise for sharing some ideas with me and Robin Carlier for taking the
time to explain to me so many facts about $\infty$-categories. I had valuable email exchanges with Marc Hoyois.
I also had useful conversations with Georg Lehner, Raphaël Ruimy and Luca Terenzi.

\section{Recollections.}
\label{SectionRecoll}
\subsection{Categories of étale motives.}
Ayoub (\cite{MR2602027} and \cite{MR3205601}), Cisinski-Déglise (\cite{MR3971240} and \cite{MR3477640}) and Voevodsky (\cite{MR1883180}) have constructed
triangulated categories of motivic sheaves, that are the relative versions of the triangulated category of geometric motives $\mathrm{DM}_{\mathrm{gm}}(k)$.
In this paper, we are interested in the étale versions. We use the stable $\infty$-categorical construction due to Robalo (\cite[Section 2.4]{MR3281141}).
Also, we only consider $\Q$-linear categories unless specified otherwise.

Let $S$ be a scheme. One starts with the category $\mathrm{Sm}_S$ of smooth $S$-schemes, and considers the $\infty$-category $\mathrm{PSh}(\mathrm{Sm}_S,\mathrm{Sp})$ presheaves of spectra on it. The category of rational effective étale motivic sheaves is the full
subcategory of $\mathrm{PSh}(\mathrm{Sm}_S,\mathrm{Sp})$ whose objects are the presheaves $\fcal$ that are $\Q$-linear (\emph{i.e.} whose image lands in $\dcal(\Q)\simeq\mathrm{Mod}_{H\Q}\subset \mathrm{Sp}$),  $\A^1$-invariant (\emph{i.e.} the natural map $\fcal(X)\to\fcal(\A^1_X)$ is an equivalence for any smooth $S$-scheme $X$), and satisfy étale hyper-descent (\emph{i.e.} for
any étale hyper cover $U_\bullet\to X$ of a smooth $S$-scheme $X$, the natural map $\fcal(X)\to \lim_{\Delta}\fcal(U_n)$ is an equivalence). It is a symmetric monoidal stable $\infty$-category, the tensor product being inherited from the monoidal structure of $\mathrm{Sm}_S$ given by the fiber product.
Inverting the object $\Q(1):=\mathrm{cofib}(S\xrightarrow{1}\mathrm{Gm}_S)$ for the tensor product gives the category that we will denote by $\mathcal{DM}^{\mathrm{\'et}}(S)$ and call the category of Voevodsky (étale) motives. 

In symbols, we have \begin{equation}\label{formuleDM}\mathcal{DM}^{\mathrm{\'et}}(S):= (\mathrm{L}_{\A^1}\mathrm{Shv}_{\acute{e}t}^\wedge(\mathrm{Sm}_S,\dcal(\Q)))[(\Gm_S/1_S)^{-1}].\end{equation}
This is a stable presentably symmetric monoidal $\infty$-category.

We will denote by $\mathcal{DM}^{\mathrm{\'et}}_c(S)$ the full subcategory of $\mathcal{DM}^\mathrm{\'et}(S)$ whose objects are compact objects (\emph{i.e.} the $M$ in $\mathcal{DM}^{\mathrm{\'et}}(S)$ such that $\Hom_{\mathcal{DM}^{\mathrm{\'et}}}(M,-)$ commutes with small filtered colimits).
If $X$ is a smooth $S$-scheme and $j$ is an integer, we denote by $M(X)(j)$ the image of $X$ under the Yoneda functor $\mathrm{Sm}_S\to\mathcal{DM}^\mathrm{\'et}(S)$, tensored by $\Q(1)^{\otimes j}$.
As we work rationally, one can show (\cite[Proposition 8.3]{MR3205601}) that the subcategory of compact objects coincides with the idempotent complete stable subcategory generated
by motives of the form $\mathrm{M}(X)(-n)$ for $X$ a smooth $S$-scheme and $n\in\N$. As $\mathcal{DM}^{\mathrm{\'et}}(S)$ is compactly generated,
we have $\Ind\mathcal{DM}^{\mathrm{\'et}}_c(S)=\mathcal{DM}^{\mathrm{\'et}}(S)$. Moreover, if $S =\Spec k$ is the spectrum of a field, then there is a natural equivalence
$\mathrm{ho}(\mathcal{DM}^{\mathrm{\'et}}_c(\Spec k))\simeq \mathrm{DM}_\mathrm{gm}(k,\Q)$ between the homotopy category of étale motives and the classical triangulated  category of geometric Voevodsky motives over a field.

Let $S$ be a scheme. The categories $\mathcal{DM}^{\mathrm{\'et}}(X)$ for $X$ ranging through $S$-schemes can be put together to give a functor $\mathrm{Sch}_S^\op\to\mathrm{CAlg}(\mathrm{Pr}^L)$ from 
the category of finite type $S$-schemes to the $\infty$-category of presentable symmetric monoidal $\infty$-categories (see \cite[Section 9.1]{robaloThese}). Moreover, 
this functor $\mathcal{DM}^{\mathrm{\'et}}$ is a coefficient system in the sense 
	of Drew and Gallauer (see \Cref{defCoSys}), and the $\infty$-category of constructible objects $\mathcal{DM}^{\mathrm{\'et}}_c(-)$ is part of a six functors 
	formalism (see \cite[Section A.5]{MR3971240} for a detailed definition), which includes the fact that all pullbacks $f^*$ have right adjoints $f_*$, that there exists exceptional functoriality consisting of 
	an adjoint pair $(f_!,f^!)$, such that $f^!\simeq f^*$ for $f$ étale and $p_*\simeq p_!$ for $p$ proper.
	If one replaces the étale topology by the Nisnevich topology in the definition \cref{formuleDM} of $\mathcal{DM}^\mathrm{\'et}(S)$, one obtain $\mathcal{SH}_\Q(S)$ the rationalisation of the stable motivic homotopy category $\mathcal{SH}(S)$, which can be obtained by further replacing in \cref{formuleDM} the $\infty$-category $\dcal(\Q)$ with the $\infty$-category $\mathrm{Sp}$ of spectra and which is the stabilisation of the $\A^1$-homotopy category introduced in \cite{MR1813224}. Except for the computation at $S= \Spec(k)$, the results stated in this paragraph and the previous one also hold for $\mathcal{SH}(S)$ and $\mathcal{SH}_\Q(S)$.

	When $S$ is of finite type over a subfield of the complex numbers, Ayoub constructed in \cite{MR2602027} the Betti realisation of $\mathcal{DM}^{\mathrm{\'et}}(S)$ with values in the
derived category of sheaves on the analytification $S^{an}$ of $S$ 
which restricts to $\mathcal{DM}^{\mathrm{\'et}}_c(S)$ and then lands into  the stable category $\mathcal{D}_c^b(S(\C),\Q)$ of bounded complexes of sheaves with constructible cohomology.
This category is both the derived category of the abelian category perverse sheaves (\cite{MR0923133})
and of the abelian category of constructible sheaves (\cite{MR1940678}).
One can find in \cite{MR3205601} and \cite{MR3477640} the construction
of the $\ell$-adic realisation functors on étale motives with values in the derived category of $\ell$-adic sheaves. An $\infty$-categorical version of this 
functor can be found in \cite[2.1.2]{MR4061978}. This is a functor $\mathcal{DM}^{\mathrm{\'et}}_c(S)\to\dcal_c^b(S,\Q_\ell)$,
the latter category being enhanced using condensed coefficients in \cite{MR4609461} where it is denoted by $\dcal_{\mathrm{cons}}(S,\Q_\ell)$. The $\ell$-adic realisation functor also has a version on the big category $\mathcal{DM}^{\mathrm{\'et}}(S)$ by taking indization. 
All the realisation functors commute with the $6$ operations.

\subsection{Perverse Nori motives.}
\label{norimot}
Let $X$ be a quasi-projective $k$-scheme with $k$ a field of characteristic zero. In \cite{ivorraFourOperationsPerverse2022}, Ivorra and S. Morel 
introduced an abelian category of perverse Nori motives $\Mp(X)$. 
The construction of the category $\Mp(X)$ goes as follows: 

Let $X$ be a quasi-projective $k$-scheme. Pick your favourite prime number $\ell$. We can compose the $\ell$-adic realisation functor
\[\rho_\ell:\mathcal{DM}^{\mathrm{\'et}}_c(X)\to \dcal^b(X_{\acute{e}t},\Q_\ell)\]
with the perverse cohomology functor \[\HHp^0:\dcal^b(X_{\acute{e}t},\Q_\ell)\to\mathrm{Perv}(X,\Q_\ell)\] 
to obtain a homological functor $\HHp^0_\ell\colon \mathcal{DM}_c(X)\to \mathrm{Perv}(X,\Q_\ell).$ 
Let $\rcal(\mathcal{DM}^{\mathrm{\'et}}_c(X))$ be the 
abelian category of coherent modules over $\mathcal{DM}^{\mathrm{\'et}}_c(X)$: it is the full subcategory of additive presheaves 
$\fcal:\mathcal{DM}^{\mathrm{\'et}}_c(X)^\op\to\mathrm{Ab}$ that fit in exact sequences
\[y_M\to y_N\to\fcal\to 0\] where $M\to N$ is a morphism in $\mathcal{DM}^{\mathrm{\'et}}_c(X)$.

The functor $\HH^0_\ell$  factors through $\rcal(\mathcal{DM}^{\mathrm{\'et}}_c(X))$, giving a functor $$\widetilde{\HHp^0_\ell}:\rcal(\mathcal{DM}^{\mathrm{\'et}}_c(X))\to \mathrm{Perv}(X,\Q_\ell)$$.
Denote by $Z_\ell$ the kernel of $\widetilde{\HHp^0_\ell}$, that is the objects mapping to zero.
The category of perverse Nori motives $\Mp(X)$ is the quotient $\rcal(\mathcal{DM}^{\mathrm{\'et}}_c(X))/Z_\ell$. It comes with a universal functor $\HH_\mathrm{univ}\colon \mathcal{DM}^\mathrm{\'et}_c(X)\to\Mp(X)$.

By construction, $\Mp(X)$ has the following universal property: any homological functor $H:\mathcal{DM}^{\mathrm{\'et}}_c(X)\to\acal$ to an abelian category $\acal$ such for all $M\in\mathcal{DM}^{\mathrm{\'et}}_c(X)$ we have $H(M)=0$ whenever $\HHp^0_\ell(M)=0$  factors uniquely 
through the universal functor \[\mathcal{DM}^{\mathrm{\'et}}_c(X)\xrightarrow{\HH_\mathrm{univ}}\Mp(X).\]  A priori this category depends on the chosen $\ell$, but one can show, using a continuity argument, that this is not the case. Indeed for finite type fields (or more generally, for fields embeddable into $\C$), 
the Betti realisation can also be used to define $\Mp(X)$ and the comparison between the $\ell$-adic realisation functor and the Betti realisation functor (see for example the proof of \cite[Proposition 6.10.]{AyoubWeil} that works over any finite type $\C$-scheme) gives that the two constructions of perverse Nori motives agree.

In fact, every realisation functor existing on $\mathcal{DM}^{\mathrm{\'et}}_c(X)$ gives us a realisation functor on $\Mp(X)$: we have the $\ell$-adic realisation $R_\ell$ and the Betti realisation $R_B$. Ivorra and S. Morel have proven that $\Mp(\Spec k)$ coincides with the category of cohomological motives constructed by Nori, and thus also has an Hodge realisation functor.

Ivorra and S. Morel show in \cite[Theorem 5.1]{ivorraFourOperationsPerverse2022} that 
the triangulated derived category of the abelian category perverse motives $\Dd^b(\Mp(-))$ is part of a homotopical $2$-functor in 
the sense of Ayoub \cite[Definition 1.4.1]{MR2423375}. This means that they constructed four of the six operations, namely for $f$ a morphism of quasi-projective varieties we have 
$f^*,f_*,f_!,f^!$, together with Verdier duality $\mathbb{D}$ , verifying $f^!\simeq \mathbb{D}\circ f^*\circ\mathbb{D}$ and $f_!\simeq\mathbb{D}\circ f_*\circ\mathbb{D}$. In his PhD thesis Terenzi \cite{terenziTensorStructurePerverse2024} constructed 
the remaining two operations: tensor product and internal $\sHom$. This proves that $\Dd^b(\Mp(-))$ is part of a full 6 functor formalism. The $6$ operations commute with the realisation functors $R_\ell$ and $R_B$ making them morphisms of symmetric monoidal stable homotopical $2$-functors.
The construction of the operations on perverse Nori motives is very similar to Saito's construction of mixed Hodge modules.

\subsection{Mixed Hodge modules}
Let $X$ be a separated finite type $\C$-scheme. In \cite{MR1047415}, M. Saito constructed an abelian category $\mathrm{MHM}(X)$ of   mixed Hodge modules. He also constructed the six operations on $\Dd^b(\mathrm{MHM}(-))$. 
There is a faithful exact functor $\mathrm{MHM}(X)\to \mathrm{Perv}(X^\mathrm{an},\Q)$ to the abelian category of perverse sheaves on the analytification of $X$. This induces a functor $\Dd^b(\mathrm{MHM}(X))\to\Dd^b(\mathrm{Perv}(X^\mathrm{an},\Q))\to\Dd^b_c(X^\mathrm{an},\Q)$ and gives a natural transformation $\Dd^b(\mathrm{MHM}(-))\to\Dd^b_c(-,\Q)$ that commutes with the 6 operations.
If $X=\Spec \C$, then there is a natural equivalence $\mathrm{MHM}(\Spec \C)\simeq\mathrm{MHS}^p_\Q$ between polarisable mixed Hodge structures and 
  mixed Hodge modules. One of the most interesting properties of the category of mixed Hodge modules is that its \emph{lisse} objects, that is the objects such that the underlying perverse sheaf is a shift of a locally constant sheaf 
are exactly the polarisable variations of mixed Hodge structures. We will not go into details about the construction of Saito's category, and we will mostly need to know that the functor taking a complex of   mixed Hodge modules to 
its underlying complex of perverse sheaves is conservative and commutes with the operations.

\section{Adapting Nori's argument to perverse Nori motives.}
\label{section1}
\subsection*{Hypothesis on the coefficient system.}
Let $k$ be a field of characteristic zero. For this section we will denote by $\mathrm{Var}_k$ either the category of quasi-projective $k$-varieties or of separated reduced $k$-schemes of finite type, and we will call objects of $\mathrm{Var}_k$ \emph{varieties}. Choose
$\Dd:\mathrm{Var}_k^\op\to \mathrm{Triang}$ your favourite $\Q$-linear triangulated category of coefficients  with a $6$ functors formalism (by this we mean a symmetric monoidal stable homotopy $2$-functor as in Ayoub's \cite{MR2423375}) and a
conservative realisation functor $R:\Dd\to\Dd^b_c$ compatible with the operations, where $\Dd^b_c$ is either the derived category of constructible sheaves on the analytification (if $k$ is a subfield of $\C$), of the derived category of $\ell$-adic constructible sheaves. For $X\in\mathrm{Var}_k$ we will call motives the elements of $\Dd(X)$. For example $\mathrm{D}$ could be the derived category of perverse Nori motives (\cite[Theorem 5.1]{ivorraFourOperationsPerverse2022} and \cite[Main theorem]{terenziTensorStructurePerverse2024}) or of mixed Hodge modules (\cite[Theorem 0.1]{MR1047415}). Note that in the mixed Hodge modules situation the realisation is nothing more that the functor forgetting the structure of a mixed Hodge module, and is called `taking the $\Q$-structure'' by Saito. We assume that for each variety $X$ the category $\mathrm{D}(X)$ has a perverse
t-structure for which $R$ is t-exact when $\Dd^b_c(X)$ is endowed with its perverse t-structure of heart the abelian category of perverse sheaves $\mathrm{Perv}(X)$.

Recall the following definition:
\begin{defi}
	The \emph{constructible t-structure} on $\Dd(X)$ is the unique t-structure on $\Dd(X)$ that makes all pullback functors t-exact. Explicitly it is the t-structure given by
	$$\ ^\mathrm{ct}\Dd^{\leqslant 0}(X) := \{M \in \Dd(X)\mid \forall x\in X,\text{ closed, } x^*M\in \Dd^{p\leqslant 0}(\Spec\kappa(x))\}$$ and 
	$$\ ^\mathrm{ct}\Dd^{\geqslant 0}(X) := \{M \in \Dd(X)\mid \forall x\in X,\text{ closed, } x^*M\in \Dd^{p\geqslant 0}(\Spec\kappa(x))\}.$$
	That this indeed gives a t-structure can be proven by gluing, using that over a smooth variety, a lisse object (see \Cref{defiLisse}) is positive (\emph{resp.} negative) for the constructible t-structure if and only it is positive (\emph{resp.} negative) for the perverse t-structure when shifted by the dimension.
\end{defi}

Denote by $\Mc(X)$ the heart of the constructible t-structure on $\Dd(X)$.
As $R$ is t-exact and conservative, all known t-exactness results about
the $6$ functors are true in $\Dd$. We will call elements of $\Mc(X)$ \emph{constructible motives} .
Denote by $\sHom(-,-)$ and $\Q_X$ the internal $\Hom$ and unit object of $\Dd(X)$.

We assume that $\Dd(X)$ has enough structure so that there exists a natural triangulated functor $\Dd^b(\Mc(X))\to\Dd(X)$ extending the inclusion of the constructible heart (for example if $\Dd(X)$ is a
derived category as shown by Beilinson, Bernstein, Deligne and Gabber in \cite[section 3.1]{MR0751966} or more generally if $\Dd(X)$ is the homotopy category of some stable $\infty$-category $\dcal(X)$ as proven in \cite[Remark 7.4.13]{bunkeControlledObjectsLeftexact2019}),
and we assume that the natural functor $ \Dd^b(\Mc(\Spec k))\to\Dd(\Spec k)$ is an equivalence. This is true for perverse Nori motives and mixed Hodge modules.

From now on in this section by the t-structure we mean the the constructible t-structure, unless specified otherwise.

\begin{defi}
	\label{defiLisse}
	Let $X$ be a $k$-variety.
	A motive $L$ is a \emph{lisse} object if its realisation $R(L)$ is a local system, that is a locally constant bounded complex of  finite dimensional $\Q$-vector spaces.
	We denote by $\Dd^{liss}(X)$ the full subcategory of lisse objects in $\Dd(X)$.
\end{defi}
We will denote by $\Q_X\in\Mc(X)$ the unit object for the tensor structure. Its realisation is the constant sheaf.
\begin{rem}
	\label{remDual}
	Note that as $\Dd(X)$ is closed as a tensor category, given an object $M\in\Dd(X)$, we have a natural candidate for the strong dual of $M$ :
	it is $\sHom(M,\Q_X)$. Therefore, an object $M$ is dualisable if and only if the map
	\begin{equation}\label{testdual}\sHom(M,\Q_X)\otimes M \to \sHom(M,M) \end{equation}
	obtained by the $\otimes$-$\sHom$ adjunction from the map $\mathrm{ev}_M\otimes \mathrm{Id}_M : \sHom(M,\Q_X)\otimes M\otimes  M\to \Q\otimes M\simeq M$, is an isomorphism.
	We will denote by $M^\vee = \sHom(M,\Q_X)$ the dual of $M$ when $M$ is dualisable.
\end{rem}

\begin{lem}
	\label{trucliss}
	Let $X$ be a $k$-variety. An object $K\in \Dd(X)$ is lisse if and only if it is dualisable. In particular, for every
	object $N\in \Dd(X)$, there exists a dense open on which $N$ is dualisable.
	If an object $L\in\Mc(X)$ is lisse, then its dual $L^\vee$ is also an object of $\Mc(X)$.
\end{lem}

\begin{proof}
	For the first point, dualisable objects in $\Db(X,\Q)$ are exactly local systems by
	Ayoub's \cite[Lemma 1.24]{ayoubAnabelianPresentationMotivic2022} in the analytic case, and by \cite[Proposition 7.6]{MR4609461} together with \cite[Corollary 7.4.12.]{martini_presentable_2022} in the $\ell$-adic case. As the realisation functor is conservative,
	we can test if the map \Cref{testdual} is an isomorphism after applying the realisation functor. If $L\in\Mc(X)$ is dualisable, the functor $\sHom(L,-) = -\otimes L^\vee$ is t-exact because it is left and right adjoint to the exact functor $-\otimes L$. Therefore, $L^\vee=\sHom(L,\Q_X)\in\Mc(X)$ lies in the heart.

	The second point follows from the fact that after realisation any constructible motives is locally constant on a stratification, hence on some dense open subset.
\end{proof}

\subsection{Cohomology of motives over an affine variety.}

\begin{defi}
	\begin{enumerate}
		\item An additive functor $F:\acal\to \bcal$ between abelian categories is called
		      \emph{effaceable} if for any $M\in \acal$ there is an injection $M\to N$ in $\acal$ such that the induced map $F(M)\to F(N)$ is the zero map.
		\item An object $M\in \Dd(X)$ is called \emph{admissible} if the functor $\Hom_{\Dd(X)}(M,-[q]) : \Mc(X)\to \mathrm{Vect}_\Q$ is effaceable for every $q>0$.
	\end{enumerate}
\end{defi}
\begin{lem}
	\label{Lemme22}
	Let $X$ be a variety and $M$ be a constructible motive on $\A^1_X$. Assume that there is a smooth
	open subset $U\subset \A^1_X$ on which $M$ is lisse and such that the restriction of $\pi:\A^1_X\to X$ to $Z = \A^1_X\setminus U$ is
	finite and surjective. Assume that $\res{M}{Z} = 0$. Then there exists $N\in \Mc(\A^1_X)$ such that $\pi_*N = 0$ and that $M$ injects into $N$.
\end{lem}
\begin{proof}
	Assume that $k$ is a subfield of $\C$. Then the following holds :
	Denote by $p_i:\A^2_X\to \A^1_X$ the projections, and by $\Delta:\A^1_X\to \A^2_X$ the diagonal.
	The map $\alpha : p_1^*M\to \Delta_* M$ in $\Mc(X)$ (both functors are t-exact) obtained by adjunction from
	$\mathrm{id}:\Delta^*p_1^*M\simeq M\to M$ is surjective because it is surjective after realisation. We therefore have an exact sequence in $\Mc(\A^1_X)$
	$$ 0\to \ker \alpha\to p_1^*M\to \Delta_* M\to 0.$$ Let $N:=\HH^1(p_2)_*\ker \alpha$. Then the long exact sequence of cohomology obtained
	after applying $(p_2)_*$ to the above sequence gives a map $M \simeq (p_2)_*\Delta_* M \simeq \HH^0((p_2)_*\Delta_* M) \to N$, which is injective by \cite[Proposition 2.2]{MR1940678}.

	Moreover, Nori proves in \emph{loc cit.} that for all $p\geqslant 0$ we have $\HH^p(\pi_*N)=0$, which is equivalent to $\pi_*N=0$.

	In the case $k$ is too big to be embedded in $\C$, on may use \cite[Proposition 3.2]{MR4325954} and the $\ell$-adic realisation, noting that the construction of $N$ is the same.
\end{proof}

Nori's method gives in our setting a slightly less good result than in the case of constructible sheaves. It will thankfully be sufficient for us.
\begin{thm}[\emph{cf. }\cite{MR1940678} Theorem 1]
	\label{beforethm1}
	For every $n\in\N$, the motive $\Q_{\A^n_k}\in\Mc(\A^n_k)$ is admissible.
\end{thm}
\begin{proof}
	We argue as in Nori's paper.\\
	We proceed by induction on $n$ to prove the theorem. First the case of a point.
	If $n=0$ and $M\in \Mc(\Spec k)$, as $\Dd(\Spec k)= \Dd^b(\Mc(\Spec k))$,
	the functors $\Hom_{\Dd(\Spec k)}(M,-[q])$ are effaceable by \cite[Proposition 3.1.16]{MR0751966}.

	Suppose that $n\geqslant 1$ and that the result is known for $n-1$. Let $M\in \Mc(\A^n_k)$ and $q>0$. Let $f\in k[T_1,\dots,T_n]$ be a polynomial such that $M$ is lisse when restricted to the open subset $U=D(f)\subset \A^n_k$. Up to a change of coordinates, we may assume that the projection 
	$\pi:\A^n_k\to\A^{n-1}_k$ on the $n-1$ first coordinates gives a finite map 
	 $\pi_{|V(f)}:V(f)\to \A^{n-1}_k$ when restricted to $Z=V(f)$.
	Denote by $j$ and $i$ the inclusions of $U$ and $Z$ in $\A^n_k$.\\
	By \Cref{Lemme22}, $j_!\res{M}{U}$ is a submotive of a constructible motive $N'$ such that
	$\pi_*N'=0$. Taking $N$ to be the push-out of the maps $j_!\res{M}{U}\to M$ and $j_!\res{M}{U}\to N'$ we get a
	morphism of exact sequences: \begin{center}\begin{tikzcd}
			0\ar[r] & j_!\res{M}{U} \ar[d,"\delta"]\ar[r] & M \ar[d,"\gamma"] \ar[r] & i_*\res{M}{Z} \ar[d,"\mathrm{id}"] \ar[r] & 0 \\
			0 \ar[r] & N' \ar[r] & N \ar[r] & i_*\res{M}{Z} \ar[r] & 0
		\end{tikzcd}
	\end{center}
	with $\gamma$ also injective.

	Also, as $\pi_*N' = 0$, we see that $\pi_*N\to \pi_*i_*\res{M}{Z}=(\res{\pi}{Z})_*\res{M}{Z}$ is an
	isomorphism because of the distinguished triangle $N'\to N\to i_*M_Z\xrightarrow{+1}$. As $\res{\pi}{Z}$ is finite, its pushforward $(\pi_{\mid Z})_*$ is t-exact, hence $\pi_*N\in \Mc(\A^{n-1}_k)$.

	By the induction hypothesis, we can find an injection $g:\pi_*N\to K$ with $K$ a constructible motive such that $\Hom_{\Dd(\A^{n-1}_k)}(\Q_{\A^{n-1}_k},g[q]) = 0$. Take $L$ the pushout of the counit $\pi^*\pi_*N \to N$
	and the injection $\pi^*g:\pi^*\pi_*N\to \pi^*K$ (note that $\pi^*$ is t-exact). We get a
	morphism of exact sequences :
	\begin{equation}\label{suitesexact2}\begin{tikzcd}
			0\ar[r] & \pi^*\pi_*N \ar[d]\ar[r,"\pi^*g"] & \pi^*K \ar[d,"\iota"] \ar[r] & C \ar[d,"\mathrm{id}_C"] \ar[r] & 0 \\
			0 \ar[r] & N \ar[r,"h"] & L \ar[r] & C \ar[r] & 0
		\end{tikzcd}.
	\end{equation}

	The unit map $\pi_*N \to \pi_*\pi^*\pi_*N$ is an isomorphism by $\A^1$-invariance.
	Therefore in the diagram above the left and the right vertical maps become isomorphisms after applying $\pi_*$, thus
	this is also the case for the middle map $\iota:\pi^*K\to L$.

	Using the adjunction $(\pi^*,\pi_*)$ and the fact that $\pi_*N \to \pi_*\pi^*\pi_*N$ is an isomorphism one gets that  the map 
	$$
		\Hom_{\Dd(\A^n_k)}(\Q_{\A^n_k},\pi^*\pi_*N[q])\simeq\Hom_{\Dd(\A^{n-1}_k)}(\Q_{\A^{n-1}_k},\pi_*\pi^*\pi_*N[q])\to \Hom_{\Dd(\A^n_k)}(\Q_{\A^n_k},N[q])
	$$ is an isomorphism. Moreover the map	$$
		\Hom_{\Dd(\A^n_k)}(\Q_{\A^n_k},\pi^*\pi_*N[q])\xrightarrow{\pi^*g} \Hom_{\Dd(\A^n_k)}(\Q_{\A^n_k},\pi^*K[q])
	$$
	induced by $\pi^*g$ vanishes as it was already the case before applying $\pi^*$.

	Therefore when applying $\Hom_{\Dd(\A^n_k)}(\Q_{\A^n_k},-[q])$ to the left square of \cref{suitesexact2}, we see that the map
	$$
		\Hom_{\Dd(\A^n_k)}(\Q_{\A^n_k},N[q])\to \Hom_{\Dd(\A^n_k)}(\Q_{\A^n_k},L[q])
	$$ induced by $h$ factors through the zero map hence is zero.

	Hence we see that the injection $M\xrightarrow{\gamma} N\xrightarrow{h} L$ induces the zero map after applying the functor $\Hom_{\Dd(\A^n_k)}(\Q_{\A^n_k},-[q])$, finishing the induction.
\end{proof}
\begin{cor}
	\label{thm1}
	If $X$ is affine, the object $\Q_X$ is admissible.
\end{cor}
\begin{proof}
	Let $q>0$.
	Take $i:X\to \A^n_k$ a closed immersion and $M\in\Mc(X)$. By \Cref{beforethm1}, there is an injection
	$i_*M\to K$ with $K$ a constructible motive such that $\Hom_{\Dd(\A^{n})}({\Q}_{\A^{n}_k},i_*M[q])\to \Hom_{\Dd(\A^{n})}({\Q}_{\A^{n}_k},K[q])$ is the zero map.
	We have an inclusion $f:i^*i_*M\simeq M\to N:=i^*K$, and a commutative diagram :
	\begin{center}\begin{tikzcd}
			\Hom_{\Dd(X)}(\Q_X,M[q])\ar[r,"f\circ-"] & \Hom_{\Dd(X)}(\Q_X,i^*K[q])  \\
			\Hom_{\Dd(\A^n_k)}(\Q_{\A^n_k},i_*M[q])\ar[u,"\sim"] \ar[r,"0"] & \Hom_{\Dd(\A^n_k)}(\Q_{\A^n_k},K[q])\ar[u,"\sim"]
		\end{tikzcd}
	\end{center}
	giving that $f\circ-$ is zero, thus finishing the proof.
\end{proof}

\subsection{Admissibility of constant motives on any variety.}

\begin{nota}
	If $j:U\hookrightarrow X$ is an open subset of a variety $X$, and $M\in \Mc(X)$ is a constructible motive,
	we will use the notation $M[U]:=j_!\res{M}{U}$. Note that by localisation we have $M/M[U]\simeq i_*\res{M}{Z}$ with $i:Z=X\setminus U\to X$ the closed complement.
\end{nota}
\begin{lem}[Stability of admissibility]
	\label{stadm}

	Let $$0\to M'\to M\to M''\to 0$$ be an exact sequence of constructible motives on $X$. \label{stadm1}
	\begin{enumerate}
		\item Assume that $M''$ is admissible and at least one of $M', M$ is admissible. Then all three are admissible. \label{stadm2a}
		\item Assume that $M'$ and $M$ are admissible. If the functor $$\coker(\Hom_{\Dd(X)}(M,-)\to\Hom_{\Dd(X)}(M',-))$$
		      is effaceable, then $M''$ is admissible. \label{stadm2b}
	\end{enumerate}

\end{lem}
\begin{proof}
	These are formal properties of abelian categories, proven in \cite[Lemma 3.2]{MR1940678}.
\end{proof}
\begin{cor}[\cite{MR1940678} Lemma 3.5]
	\label{bracket}
	Let $M\in\Mc(X)$. \begin{enumerate}
		\item If $U\subset X$ is open, and if $M$ and $M[U]$ are admissible, then so is $M/M[U]$.
		\item If $V,W$ are open subsets of $X$, and if $M[V],M[W],M[V\cap W]$ are admissible, then the same is true for $M[V\cup W]$
	\end{enumerate}
\end{cor}

\begin{proof}
	We use part 2. of \Cref{stadm}. For 1., it suffices to show that any $P\in \Mc(X)$ is a subsheaf of a
	$Q\in\Mc(X)$ such that the map $f\colon\Hom_{\Dd(X)}(M,Q)\to\Hom_{\Dd(X)}(M[U],Q)$ is surjective. Denote by $j:U\to X$ the open immersion and
	by $i:Z\to X$ its closed complement. Then define $Q = i_*i^*P\oplus j_*j^*P$. Localisation ensures that
	the map $P\to Q$ is injective,$$\Hom_{\Dd(X)}(M,i_*\res{P}{Z})\oplus
		\Hom_{\Dd(U)}(\res{M}{U},\res{P}{U})\simeq\Hom_{\Dd(X)}(M,Q)$$  and $$
		\Hom_{\Dd(U)}(\res{M}{U},\res{P}{U}) \simeq 0\oplus \Hom_{\Dd(X)}(j_!\res{M}{U},j_*\res{P}{U}) \simeq \Hom_{\Dd(X)}(M[U],Q) $$
	thus is $f$ surjective as it is just the projection on the second factor.\\
	For 2. first one has the Mayer-Vietoris exact sequence \begin{equation}0\to M[V\cap W]\to M[V]\oplus M[W] \to M[V\cup W]\to 0.\end{equation}
	We let $U=V\cap W$. For a given $P$ we take the same $Q$ as above for $U$ that gives surjectivity of $\Hom_{\Dd(X)}(M,Q)\to \Hom_{\Dd(X)}(M[U],Q)$.
	Now the map $a:M[U]\to M$ factors through $M[V]$ thus the map \begin{equation}\Hom_{\Dd(X)}(M[V],Q)\xrightarrow{a^*} \Hom_{\Dd(X)}(M[U],Q)\end{equation}
	is surjective. Finally the map induced by $M[U]\to M[V]\oplus M[W]$ is also surjective as one of its direct summands is.
\end{proof}
\begin{prop}[Avoiding the use of injective objects, after Nori \cite{MR1940678} Remark 3.8]
	\label{Injective}
	Let $\acal,\bcal$ and $\ccal$ be abelian categories.
	Let $G:\bcal\to\acal$ be a functor which has an exact left adjoint.
	Let $H:\acal\to \ccal$ be an effaceable functor. Then $HG$ is also effaceable.
\end{prop}
\begin{proof}
	Let $F$ be a left adjoint of $G$. Take $B$ an object of $\bcal$. By assumption on $H$, there is an injection $u:G(B)\to A$ such that
	the induced map $H(u):HG(B)\to H(A)$ vanishes. By exactness of $F$, the morphism $F(u):FG(B)\to F(A)$ is a monomorphism.
	We also have the counit of the adjunction $\varepsilon:FG(B)\to B$. Take $B'$ the pushout of these two maps :
	\begin{center}
		\begin{tikzcd}
			FG(B) & F(A) \\
			B & {B'}
			\arrow["{F(u)}", from=1-1, to=1-2]
			\arrow["\varepsilon"', from=1-1, to=2-1]
			\arrow["v"', hook, from=2-1, to=2-2]
			\arrow["t", from=1-2, to=2-2]
			\arrow["\lrcorner"{anchor=center, pos=0.125, rotate=180}, draw=none, from=2-2, to=1-1]
		\end{tikzcd}
	\end{center}
	The map $v:B\to B'$ is a monomorphism and we have a commutative diagram (with $\eta$ the unit of the adjunction) :
	\begin{center}\begin{tikzcd}
			G(B) & A \\
			GFG(B) & GF(A) \\
			G(B) & {G(B')}
			\arrow["u", hook, from=1-1, to=1-2]
			\arrow["{\eta_A}", from=1-2, to=2-2]
			\arrow["{G(t)}", from=2-2, to=3-2]
			\arrow["{\eta_{G(B)}}"', from=1-1, to=2-1]
			\arrow["G(\varepsilon)"', from=2-1, to=3-1]
			\arrow["{G(v)}", from=3-1, to=3-2]
			\arrow["{GF(u)}", from=2-1, to=2-2]
			\arrow["{\mathrm{Id}}"', bend right=70, from=1-1, to=3-1]
		\end{tikzcd}
	\end{center}
	The composition $G(\varepsilon)\circ\eta_{G(B)}$ is the identity, therefore, the map $G(v)$ factors as $w\circ u$
	with $w = G(t)\circ \eta_A$. Now, $HG(v)=H(w) \circ H (u) = 0$ hence $HG$ is effaceable.
\end{proof}
\begin{cor}
	\label{LAdjAdm}
	Let $X$ and $Y$ be varieties. Let $F:\Dd(X)\to \Dd(Y)$ be a t-exact exact functor which is left adjoint to a t-exact functor. If $M\in\Mc(X)$ is admissible, then so is $F(M)\in\Mc(Y)$.
\end{cor}
\begin{proof}
	Let $q>0$ and let $G$ be a right adjoint of $F$. The functor $H = \Hom_{\Dd(X)}(M,-[q])$ is effaceable by definition. By \Cref{Injective}, $H\circ G$ is therefore effaceable.
	We have a natural isomorphism $H\circ G \simeq \Hom_{\Dd(Y)}(F(M),-[q])$ because $G$ is t-exact, which gives the claim.
\end{proof}
\begin{cor}
	\label{corLAdjAdm}
	Let $X$ be a variety.
	Let $M\in\Mc(X)$ be an admissible motive.
	\begin{enumerate}
		\item Let $j:X\to Y$ be an étale morphism. Then $j_!M$ is admissible.
		\item Let $L\in\Mc(X)$ be a lisse motive. Then $M\otimes L$ is admissible.
	\end{enumerate}
\end{cor}
\begin{proof}
	We use \Cref{LAdjAdm}.
	For the first point, we have the adjunction $(j_!,j^*)$ where $j_!$ and $j^*$ are t-exact.
	Let $L\in\Mc(X)$ be a lisse motive. The functor $-\otimes L$ is left adjoint to $\sHom(L,-)$ and right adjoint to $\sHom(L^\vee,-)$ hence is t-exact. The same can be said about $\sHom(L,-)=-\otimes L^\vee$.
\end{proof}

\begin{prop}[\cite{MR1940678} Proposition 3.6]
	Let $U\subset X$ an open of a variety $X$. Then the constructible motive $\Q_X[U]$ is admissible.
\end{prop}
\begin{proof}
	Let $j:U\to X$ be the open immersion. We prove the theorem by induction on the number $n$ of affines needed to cover $U$. If $U$ is affine,
	\Cref{thm1} ensures that $\Q_U$ is admissible hence by \Cref{corLAdjAdm}, $j_!\Q_U=\Q_X[U]$ is admissible. For the induction step, one can write $U = V\cup W$ with $V$ affine and $W$ covered by $n-1$ affines,
	and separateness of $X$ gives that $V\cap W$  is covered by $n-1$ affines.
	Therefore by induction $\Q_X[V]$, $\Q_X[V\cap W]$ and $\Q_X[W]$ are admissible.
	Then, \Cref{bracket} ensures that $\Q_X[U]$ is admissible.
\end{proof}
\begin{cor}[\cite{MR1940678} Theorem 2]
	\label{thm2}
	Let $X$ be a variety over $k$. Then the constant motive $\Q_X$ is admissible, that is, for every $q>0$,
	every constructible motive $M\in \Mc(X)$ can be embedded in a constructible motive $N\in \Mc(X)$
	such that the map $$\Hom_{\Dd(X)}(\Q_X,M[q])\to \Hom_{\Dd(X)}(\Q_X,N[q])$$ vanishes.
\end{cor}
\begin{proof}
	One takes $U=X$ in the previous proposition.
\end{proof}

\subsection{Lisse motives enter the party.}

\begin{prop}
	\label{lissadm}
	Let $X$ be a variety. Any lisse motive in $\Mc(X)$ is admissible.
\end{prop}
\begin{proof}
	By \Cref{thm2} the unit object $\Q_X$ on $X$ is admissible. By \Cref{corLAdjAdm} $L \simeq \Q_X\otimes L$ is admissible.
\end{proof}

\begin{thm}
	\label{thm3}
	Let $X$ be a variety. Then any $M\in \Mc(X)$ is admissible, that is, for every $q>0$, the functor $\mathrm{Hom}_{\Dd(X)}(M,-[q]):\Mc(X)\to \mathrm{Vect}_\Q$ is effaceable.
\end{thm}
\begin{proof}
	We prove the result by Noetherian induction on the support $S$ of $M$ in $X$, that is the complement of the largest open subset $V$ of $X$ such that $\res{M}{V} = 0$. We follow the proof of \cite[Proposition 3.10]{MR1940678}.
	
	Let $S^d$ be the union of the irreducible components of maximal dimension of $S$. We denoted by $d$ the dimension of $S$. Let $U_\mathrm{liss}$ be the largest open subset of $S^d$ on which $M$ is lisse. Then there is an affine open subset $U_1$ of $X$ such that the intersection $U_1\cap S^d$ is contained in $U_\mathrm{liss}$ and is nonempty. Let $U_2'$ be the open subset of $X$ obtained by removing from $U_1$ the irreducible components of $S$ that are not of maximal dimension. We have that $U_2'\cap S$ is nonempty and contained in $U_\mathrm{liss}$. We choose a nonempty affine open subset $U_2$ of $U_2'$ such that $U_2\cap S$ is nonempty and a Noether normalisation $g_2\colon U_2\cap S\to \A^d_k$. It extends to a map $f_2\colon U_2\to \A^d_k$. Indeed, because $U_2\cap S$ is closed in $U_2$ and the three schemes $U_2$, $U_2\cap S$ and $\A^d_k$ are affine thus we may choose in $\oscr_{U_2}(U_2)$ lifts of the images of the coordinates of $\A^d_k$ in $\oscr_{U_2\cap S}(U_2\cap S)$. Let $W\subset \A^d_k$ be a nonempty open subset over which $g_2$ is smooth. As $g_2$ is of relative dimension $0$ in fact $g_2$ is étale over $W$. We set $U := f_2^{-1}(W)$. 
	
	Let $Z = U\cap S$ and $g\colon Z\to W$ be the restriction of $g_2$ to $Z$. By construction we have $Z = g_2^{-1}(W)$ thus the map $g\colon Z\to W$ is finite and étale. We also let $f:U\to W$ be the restriction restriction of $f_2$ to $U$. The situation may be summarised in the following commutative diagram:
	\[\begin{tikzcd}
		Z & U & W \\
		{U_2\cap S} & {U_2} & {\A^n_k}
		\arrow["i"', from=1-1, to=1-2]
		\arrow["g"', from=1-1, to=1-3,bend left=30]
		\arrow[from=1-1, to=2-1]
		\arrow["\lrcorner"{anchor=center, pos=0.125}, draw=none, from=1-1, to=2-2]
		\arrow["f",swap, from=1-2, to=1-3]
		\arrow[from=1-2, to=2-2]
		\arrow["\lrcorner"{anchor=center, pos=0.125}, draw=none, from=1-2, to=2-3]
		\arrow[from=1-3, to=2-3]
		\arrow[from=2-1, to=2-2]
		\arrow["{g_2}"', from=2-1, to=2-3, bend right=30]
		\arrow["{f_2}"', from=2-2, to=2-3]
	\end{tikzcd}\]
	where we have denoted by $i$ the inclusion of $Z$ inside $U$. 

	Note that because the restriction of $M$ to $Z$ is lisse and $g$ is finite étale,  the object $g_*M_{\mid Z}$ is still lisse, so that $L:=f^*g_*M_{\mid Z}$ is lisse. The restriction $i^*L$ to $Z$ of $L$ is $g^*g_*M_{\mid Z}$ which has $M_{\mid Z}$ as a direct factor because $g$ is étale.

	Because the support of $M_{\mid U}$ is contained in $Z$ we have that $i_*M_{\mid Z}= i_*i^*M_{\mid U}\simeq M_{\mid U}$.
	 By the above, this shows that $M_{\mid U}$ is a direct factor of $i_*L_{\mid Z}$. Let $j$ be the inclusion of the open complement of $Z$ in $U$. By localisation we have that $i_*L_{\mid Z} \simeq L/j_!j^*L$. 
	 By \Cref{lissadm} the objects $L$ and $j^*L$ are admissible, thus as by \Cref{corLAdjAdm} the object $j_!j^*L$ is then admissible, we deduce by \Cref{bracket} that the object $i_*L_{\mid Z}$ is admissible.
	  Thus the restriction $M_{\mid U}$ of $M$ to $U$ is admissible as a direct factor of an admissible object. 
	  
	  If $\iota\colon U\to X$ is the open immersion then $\iota_!M_{\mid U}$ is admissible by \Cref{corLAdjAdm}. We have an exact sequence
	 $$0\to \iota_!M_{\mid U}\to M\to M/\iota_!M_{\mid U}\to 0$$ in which the support of $M/\iota_!M_{\mid U}$ is strictly smaller than the support of $M$. By Noetherian induction $M/\iota_!M_{\mid U}$ is admissible, thus by \Cref{stadm} we finally obtain that $M$ is admissible.
\end{proof}
\begin{cor}
	\label{cccp}
	Let $X$ be a variety. The natural functor $\Dd^b(\Mc(X))\to\Dd(X)$ is an equivalence.
\end{cor}
\begin{proof}
	We apply \cite[Proposition 3.1.16]{MR0751966}. The conditions are verified because of the previous \Cref{thm3}.
\end{proof}

\section{Higher categorical enhancements.}
\label{section2}

\subsection{Categorical preliminaries.}

Recall the following fundamental definition of Lurie:
\begin{defi}
	An $\infty$-category $\ccal$ is \emph{stable} if it has all finite limits and colimits, if initial objects coincide with final objects and if any commutative square in $\ccal$
	is cocartesian if and only if it is cartesian.
	An exact functor between stable $\infty$-categories is an $\infty$-functor preserving finite limits and colimits.
\end{defi}
The homotopy category $\mathrm{ho}(\ccal)$ of a stable $\infty$-category is always a triangulated category (\cite[Theorem 1.1.2.14.]{lurieHigherAlgebra2022}) and exact functors induce triangulated functors. By definition, a t-structure on a stable $\infty$-category is a t-structure on its homotopy category.

\begin{defi}
	The bounded derived $\infty$-category of an abelian category $\acal$ is the $\infty$ category $\dcal^b(\acal)$ obtained by inverting quasi-isomorphisms in the category $\mathrm{Ch}^b(\acal)$ of bounded complexes of objects of $\acal$. It is a stable $\infty$-category with a t-structure, and its homotopy category is the usual derived category $\Dd^b(\acal)$.
\end{defi}
If $F\colon \acal\to \ccal$ is a functor between an abelian category and a stable $\infty$-category, we will say that $F$ is exact if it preserves finite coproducts and it sends short exact sequences to exact triangles. 

We will use the following result of Bunke, Cisinski, Kasprowski, and Winges:
\begin{thm}[{\cite[Corollary 7.4.12]{bunkeControlledObjectsLeftexact2019}}]
	\label{RestrToHeart}
	Let $\ccal$ be a stable $\infty$-category and let $\acal$ be an abelian categories. Then restriction to the heart 
	gives an equivalence of $\infty$-categories $$\mathrm{Fun}^\mathrm{ex}(\dcal^b(\acal),\ccal)\to\mathrm{Fun}^{\mathrm{ex}}(\acal,\ccal)$$
	between $\infty$-categories of exact functors.
\end{thm}

There is an easy generalisation with multiple variables:
\begin{cor}
	\label{CurryDb}
	Let $\ccal$ be a stable $\infty$-category and let $\acal_1,\dots,\acal_n$ be abelian categories. Denote by $\mathrm{Fun}^{\mathrm{nex}}(\prod_i\dcal^b(\acal_i),\ccal)$ the $\infty$-category of $n$-multi-exact functors
	$\prod_i \dcal^b(\acal_i)\to\ccal$ that is functors that are exact in each variables.
	Denote also similarly by $\mathrm{Fun}^{\mathrm{nex}}(\prod_i\acal_i,\ccal)$ the $\infty$-category of $n$-multi-exact functors.
	Then the restriction to the hearts functor \begin{equation}
		\mathrm{Fun}^{\mathrm{nex}}(\prod_i\dcal^b(\acal_i),\ccal)\to \mathrm{Fun}^{\mathrm{nex}}(\prod_i\acal_i,\ccal)
	\end{equation}
	is an equivalence of $\infty$-categories.
\end{cor}
\begin{proof}
	We give the proof for $n=2$, the general case being similar and can be reduced to the case $n=2$ by induction. We have canonical equivalence of $\infty$-categories
	$\mathrm{Fun}(\dcal^b(\acal_1)\times\dcal^b(\acal_2),\ccal)\simeq \mathrm{Fun}(\dcal^b(\acal_1),\mathrm{Fun}(\dcal^b(\acal_2),\ccal))$. This equivalences induces 
	$$\mathrm{Fun}^{\mathrm{2ex}}(\dcal^b(\acal_1)\times\dcal^b(\acal_2),\ccal)\simeq\mathrm{Fun}^\mathrm{ex}(\dcal(\acal_1),\mathrm{Fun}^\mathrm{ex}(\dcal^b(\acal_2),\ccal)).$$
	By \Cref{RestrToHeart} the latter category is exactly $\mathrm{Fun}^{\mathrm{ex}}(\acal_1,\mathrm{Fun}^{\mathrm{ex}}(\acal_2,\ccal))$, which again is equivalent to
	$ \mathrm{Fun}^{\mathrm{2ex}}(\acal_1\times\acal_2,\ccal)$. 
\end{proof}

We will also need a result on the monoidal structure of the bounded derived category. We denote by $\mathrm{SymMono}_1$ the $2$-category of symmetric monoidal $1$-categories and symmetric monoidal functors.

\begin{prop}
	\label{symMonoCatDeri}
	Let $\bcal$ be a symmetric monoidal abelian category such that the tensor product which is exact in each variable. Then there is a canonical symmetric monoidal structure on the stable $\infty$-category $\dcal^b(\bcal)$ such that the inclusion functor 
	$\bcal\to \dcal^b(\bcal)$ is symmetric monoidal. Moreover, for any category $I$ and any functor 
	$\acal: I\to \mathrm{SymMono}_1$ such that for each $i\in I$ the category $\acal(i)$ is a symmetric monoidal abelian category with exact tensor product as above, and such that the transition functors are exact, there is a lift of the functor 
	$\dcal^b\circ \acal\colon I\to \catinfty$
	to a functor $I\to \mathrm{CAlg}(\catinfty).$
\end{prop}

\begin{proof}
	The first case follows from the second with $I=*$.
	Let $\acal:I\to \mathrm{SymMono}_1$ be a diagram of symmetric monoidal abelian categories and symmetric monoidal exact functors, and such that for each $i\in I$ the tensor product on $\acal(i)$ is exact in both variables. We can compose this functor with $\mathrm{Ch}^b(-)$ to obtain a diagram of symmetric monoidal additive categories (the monoidal structure on $\mathrm{Ch}^b(\acal(i))$ is induced by that of $\acal(i)$, using the sign trick). By \cite[Chapter XI, Section 1, Theorem 1]{MR1712872} we know that any symmetric monoidal $1$-category has all higher coherences. Written differently as in Lurie's \cite[Corollary 5.1.1.7]{lurieHigherAlgebra2022}, the forgetful functor induces an equivalence of $2$-categories
	$$\mathrm{CAlg}(\mathrm{Cat}_1)\to \mathrm{SymMono}_1.$$ Thus we can see $\mathrm{Ch}^b\circ\acal$ as a functor $$I\to \mathrm{CAlg}(\mathrm{Cat}_1)\subset\mathrm{CAlg}(\catinfty).$$ 
	Now by \cite[Theorem 2.4.3.18 and Proposition 2.4.2.5]{lurieHigherAlgebra2022} the data of this functor is equivalent to the data of a $I^\amalg$-monoid, which is classified by a cocartesian fibration 
	$p\colon \mathfrak{C}\to I^\amalg$ (we denote by $I^\amalg$ the $\infty$-operad constructed in \cite[Construction 2.4.3.1]{lurieHigherAlgebra2022}). By \cite[Proposition 2.1.4]{MR3460765} we can fiberwise invert quasi-isomorphisms to obtain a cocartesian fibration $q\colon \mathfrak{D}\to I^\amalg$ which classifies (using unstraightening\footnote{One has to verify that the functor $I^\amalg\to\catinfty$ obtained by unstraightening is still a $I^\amalg$-monoid, but this is clear because $\dcal^b(-)$ commutes with finite products: this is the case for $\mathrm{Ch}^b(-)\subset\mathrm{Fun}(\Z,-)$ and then we invert products of maps.} and \cite[Theorem 2.4.3.18, Proposition 2.4.2.5]{lurieHigherAlgebra2022} again) a functor 
	$I\to\mathrm{CAlg}(\catinfty)$ sending an object $i\in I$ to $\dcal^b(\acal(i))$. The fact that the functor $\acal(i)\to\dcal^b(\acal(i))$ is symmetric monoidal comes from the exactness of the tensor product on $\acal(i)$.
\end{proof}

We also have a version of the universal property of $\dcal^b(\acal)$ (see \Cref{RestrToHeart}) that works in families.

\begin{prop}
	\label{SymMonoReal}
	Let $\ccal: I\to \mathrm{CAlg}(\catinfty)$ be a diagram of stably symmetric monoidal $\infty$-categories such that for each $i\in I$ the $\infty$-category $\ccal(i)$ has a t-structure such that the tensor product is t-exact in each variable and every arrow $i\to j$ in $I$ induces a t-exact functor. Then the canonical functors $\dcal^b(\ccal(i)^\heartsuit)\to \ccal(i)$ assemble to give a natural transformation $\dcal^b(\ccal^\heartsuit)\Rightarrow \ccal$ of functors $I\to \mathrm{CAlg}(\catinfty)$.
\end{prop}
\begin{proof}
	For $\bcal$ an abelian category, denote by $\kcal^b(\bcal)$ the $\infty$-category of bounded complexes in $\bcal$, that is the $\infty$-categorical localisation of the additive category of  bounded chain complexes $\mathrm{Ch}^b(\bcal)$ with respect to chain homotopies. By the work of
	Bunke, Cisinski, Kasprowski and Winges, we know (\cite[Theorem 7.4.0]{bunkeControlledObjectsLeftexact2019}) that $\kcal^b(\bcal)$ is the value at $\bcal$ of a functor 
	$$\kcal^b(-)\colon \mathrm{Cat}_\infty^\mathrm{add}\to\mathrm{Cat}_\infty^\mathrm{ex}$$ from the $\infty$-category of additive $\infty$-categories and coproduct preserving functors to the $\infty$-category $\mathrm{Cat}_\infty^\mathrm{ex}$ of small stable $\infty$-categories and exact functors. Moreover, this functor is left adjoint to the forgetful functor $$\mathrm{ff}\colon \mathrm{Cat}_\infty^\mathrm{ex}\to\mathrm{Cat}_\infty^\mathrm{add}.$$
	By Cisinski's \cite[Theorem 6.1.22]{MR3931682} this induces an adjunction 
	\begin{equation}\label{adjKb}\begin{tikzcd}
		\mathrm{Fun}(I^\amalg,\mathrm{Cat}_\infty^\mathrm{add})
			\arrow[r, bend left = 25, "\circ \kcal^b"{name=D}]
			\arrow[r, leftarrow, bend right = 25, swap, "\circ \mathrm{ff}"{name=C}]
			  \arrow[d, from=D, to=C, phantom, "{\bot}"]
		  & \mathrm{Fun}(I^\amalg,\mathrm{Cat}_\infty^\mathrm{ex})
	\end{tikzcd},\end{equation}
 where the category $I^\amalg$ is as introduced in the proof of the above proposition.	Our assumptions on $\ccal$ imply that if we denote by $\acal(i)$ the heart of the t-structure on $\ccal(i)$, we have a natural transformation $\acal\Rightarrow\ccal$ of functors $I\to\mathrm{CAlg}(\mathrm{Cat}_\infty^\mathrm{add}$). Now, as in the proof of \Cref{symMonoCatDeri}, Lurie's \cite[Theorem 2.4.3.18 and Proposition 2.4.2.5]{lurieHigherAlgebra2022} give that the data of this natural transformation is equivalent to the data of a natural transformation $\mathfrak{A}\Rightarrow \mathfrak{C}$ of $I^\amalg$-monoids $I^\amalg\to \mathrm{Cat}_\infty^\mathrm{add}$  (see \cite[Definition 2.4.2.1]{lurieHigherAlgebra2022} for a definition). By \Cref{adjKb}, this gives a natural transformation 
	$$\mathrm{can}\colon\kcal^b\circ \mathfrak{A} \Rightarrow \mathfrak{C}$$ of functors $I^\amalg\to \mathrm{Cat}_\infty^\mathrm{ex}$. We claim that $\kcal^b\circ\mathfrak{A}$ is still a $I^\amalg$-monoid. Indeed, it suffices to prove that $\kcal^b$ preserves finite products, but both in $\catinfty^\mathrm{ex}$ and in $\catinfty^\mathrm{add}$, finite product and coproduct agree (this can be proven exactly as in \cite[Lemma 2.1.38]{bunkeControlledObjectsLeftexact2019}), so that this follows from $\kcal^b$ being a left adjoint which thus preserves colimits. 

	Unravelling the definitions, we have proven that the natural transformation $\kcal^b(\acal)\Rightarrow \ccal$ is well defined and is symmetric monoidal. We will now show that this factors through $\dcal^b(\acal)$, and that everything stays symmetric monoidal. For this, we will see our map of $I^\amalg$-monoidal $\kcal^b\circ\mathfrak{A}\Rightarrow \mathfrak{C}$ through straightening, thus as a commutative triangle
		$$\begin{tikzcd}
		\int_{I^\amalg}\kcal^b\circ\mathfrak{A}\ar[r,"\mathfrak{r}"]\ar[d,"\kappa"] & \int_{I^\amalg}\mathfrak{C}\ar[dl,"\gamma"]\\
		I^\amalg &
	\end{tikzcd}$$ where $\kappa$ and $\gamma$ are the cocartesian fibrations classifying the functors $\kcal^b\circ\mathfrak{A}$ and $\mathfrak{C}$. We consider a marking $W$ on $\int_{I^\amalg}\kcal^b\circ\mathfrak{A}=:\mathscr{K}$ to be the set of arrows
	 that are products of quasi-isomorphisms. That is an arrow $f$ in $\mathscr{K}$ is marked if there exists an integer $n$ and an object $\underline{i}=(i_1,\dots,i_n)\in I^\amalg_{\langle n\rangle}$ such that $f$ is a map in the fiber $\prod_{j=1}^n\kcal^b(\acal(i_j))$ above $\underline{i}$ and is of the form $f_1\times\cdots\times f_n$ with each $f_j$ being a quasi-isomorphism. Then the cocartesian fibration $\kappa$, together with the data of $W$ and the marking of isomorphisms in $I^\amalg$, is a marked cocartesian fibration in the sense of Hinich \cite[Definition 2.1.1]{MR3460765}. Because each $\acal(i)$ is the heart of a t-structure on $\ccal(i)$, all maps in $W$ are sent to isomorphisms in $\int_{I^\amalg}\mathfrak{C}=:\mathscr{C}$. Indeed, over a multi-index $(i_1,\dots,i_n)$ in $I^\amalg$, the functor $\mathfrak{r}$ becomes the functor 
	 \[\prod_{j=1}^n \kcal^b(\acal(i_j))\to \prod_{j=1}^n\ccal(i_j)\] which sends marked maps to isomorphisms by \cite[Proposition 7.4.11]{bunkeControlledObjectsLeftexact2019}. Thus the map 
	$$\mathscr{K}\to\mathscr{C}$$ factors through 
	$$\mathscr{K}\to \mathscr{K}[W^{-1}].$$ By \cite[Proposition 2.1.4]{MR3460765}, the map $\mathscr{K}[W^{-1}]\to I^\amalg$ is a cocartesian fibration, and moreover, the fiber above a $\underline{i}=(i_1,\dots,i_n)$ is the localisation of $\prod_{j= 1}^n\kcal^b(\acal(i_j))$ with respect to products of quasi-isomorphisms, which is $\prod_{j=1}^n\dcal^b(\acal(i_j))$. Because $I^\amalg$ is a $1$-category, to check that the diagram 
	$$\begin{tikzcd}
		\mathscr{K}[W^{-1}]\ar[r]\ar[d,"\kappa"] & \mathfrak{C}\ar[dl,"\gamma"]\\
		I^\amalg &
	\end{tikzcd}$$ commutes it suffices by adjunction to check that $\mathrm{ho}(-)$ of it commutes, which is clear because the horizontal map preserves fibers. Thus after unstraightening we obtain a map of $I^\amalg$-monoids 
	$$\dcal^b\circ\mathfrak{A}\to\mathfrak{C}$$ which express exactly the symmetric monoidality of a natural transformation $\dcal^b\circ\acal\to\ccal$, finishing the proof.

\end{proof}

\begin{rem}
	\label{remKb}
The proof of \Cref{SymMonoReal} also give the same result with $\kcal^b(-)$ instead of $\dcal^b(-)$, and the fact that $\kcal^b(-)\Rightarrow\dcal^b(-)$ is symmetric monoidal on symmetric monoidal abelian categories and symmetric monoidal exact functors.
\end{rem}

We will use the following two theorems :
\begin{thm}[Theorem 3.3.1 of \cite{MR4093970}]
	\label{adjointHo}
Let $\dcal$ be an $\infty$-category which admits finite limits and let $\ccal$ be an $\infty$-category. 
Let $G:\dcal\to\ccal$ be a functor which preserves finite limits. Then $G$ admits a left adjoint 
if and only if $\mathrm{ho}(G):\mathrm{ho}(\dcal)\to\mathrm{ho}(\ccal)$ does.
\end{thm}
\begin{cor}
	\label{adjointHo}
	Let $\ccal$ and $\dcal$ be a stable $\infty$-categories. 
Let $G:\dcal\to\ccal$ be an exact functor. Then $G$ admits a left (\emph{resp}. a right) adjoint 
if and only if $\mathrm{ho}(G):\mathrm{ho}(\dcal)\to\mathrm{ho}(\ccal)$ does.
\end{cor}
\begin{proof}
	The case of the left adjoint is the above theorem, and the case of the right adjoint follows by passing to the opposite $\infty$-categories.
\end{proof}
\begin{thm}[Theorem 7.6.10 of \cite{MR3931682}]
	\label{equivHo}
Let $F:\dcal\to\ccal$ be a functor between $\infty$-categories having finite limits. Assume that $F$ preserves finite limits. Then $F$ is an equivalence of $\infty$-categories if and only if $\mathrm{ho}(F):\mathrm{ho}(\dcal)\to\mathrm{ho}(\ccal)$ is an equivalence of categories.
\end{thm}

Recall that by \cite[\href{https://kerodon.net/tag/01F1}{Tag 01F1}]{lurieKerodon}, we have the following proposition :
\begin{prop}
	\label{replacement}
	Let $F:\ccal\to\dcal$ be a functor of $\infty$-categories. Assume that for a collection $A\subset\ccal$ of objects of $\ccal$ is given the data, for each $a\in A$ of an isomorphism $f_a:F(a)\to d_a$ with $d_a\in\dcal$.
	Then there exist a functor $G:\ccal\to\dcal$ and a natural equivalence $g:F\Rightarrow G$ such that for every $a\in A$,
	we have $g_a = f_a$ (and this is really an equality).
\end{prop}
\begin{proof}
	Indeed, denote by $\iota : A\to\ccal$ the canonical monomorphism of simplicial sets. Then by \cite[\href{https://kerodon.net/tag/01F1}{Tag 01F1}]{lurieKerodon}, the restriction
	\begin{equation}
		\iota^* : \mathrm{Fun}(\ccal,\dcal)\to\mathrm{Fun}(A,\dcal)
	\end{equation}
	is an isofibration of simplicial sets (\emph{cf.} \cite[3.3.15]{MR3931682}). The data of the $f_a$ defines an isomorphism
	$\gamma$ in $\mathrm{Fun}(A,\dcal)$ between $F_{\mid A}$ and the map of simplicial sets $A\to\dcal$ defined by $a\mapsto d_a$.
	By the lifting property of isofibrations, there is a functor $G:\ccal\to\dcal$ and an isomorphism $F\simeq G$, that is a natural equivalence $g:F\Rightarrow G$ such that $\iota^*g = \gamma$, which is exactly the statement of the proposition.
\end{proof}

To prove descent statement for categories of perverse Nori motives and mixed Hodge modules we will use the following proposition and its convenient corollary. Robin Carlier really helped the author in the writing of this proposition.

\begin{prop}
	\label{descabs}
	Let $\chi : (\ccal^\op)^{\lhd}\to \catinfty$ be a functor.
	For any $f:c\to c'$ in $(\ccal^\op)^{\lhd}$, denote $f^*=\chi(f)$,
	and assume that any such $f^*$ has a right adjoint $f_*$.
	For each $c\in\ccal$, denote $f_c:c\to v$ the unique map
	from $c$ to the end point $v$ (we have that $(\ccal^\op)^\lhd \simeq (\ccal^\rhd)^\op$). Let $\overline{\pi}:\overline{\dcal}\to (\ccal^\op)^{\lhd}$ be the
	cocartesian fibration that classifies $\chi$, and let $\pi:\dcal\to \ccal^\op$ be
	its pullback by the inclusion $\ccal^\op\to(\ccal^\op)^{\lhd}$.\\
	Then $\chi$ is a limit diagram
	if and only if the following two conditions are verified :
	\begin{enumerate}\item The family $(f_c^*)_{c\in\ccal}$ is conservative.
		\item For any cocartesian section $X:\ccal^\op\to \dcal $ of $\pi$, the limit $X(v):=\lim_{c\in\ccal}(f_c)_*X(c)$ exists in $\chi(v)$ and the map
		      $f_c^*X(v)\to X(c)$ adjoint to  the canonical map $$\lim_c (f_c)_*X(c)\to (f_c)_*X(c)$$ is an equivalence.
	\end{enumerate}
\end{prop}
\begin{proof}
	By  \cite[\href{https://kerodon.net/tag/02T0}{Construction 02T0}]{lurieKerodon}, there is a diffraction functor
	\begin{equation}
		\mathrm{Df}:\overline{\dcal}_v\to \mathrm{Fun}_{/\mathcal{C}^\op}^\mathrm{Cocart}(\mathcal{C}^\op,\mathcal{D})
	\end{equation}
	that fits in a commutative triangle :
	\begin{equation}
		\label{diffraction}
		\begin{tikzcd}
			{\overline{\dcal}_v} && {\mathrm{Fun}_{/\mathcal{C}^\op}^\mathrm{Cocart}(\mathcal{C}^\op,\mathcal{D})} \\
			& {\mathrm{Fun}_{/(\mathcal{C}^\op)^\lhd}^\mathrm{Cocart}((\mathcal{C}^\op)^\lhd,\overline{\mathcal{D}})}
			\arrow["{\mathrm{Df}}", from=1-1, to=1-3]
			\arrow["{\mathrm{ev}_v}", from=2-2, to=1-1]
			\arrow["{\mathrm{res}}"', from=2-2, to=1-3]
		\end{tikzcd}.
	\end{equation}
	By the diffraction criterion
	\cite[\href{https://kerodon.net/tag/02T8}{Theorem 02T8}]{lurieKerodon},
	$\chi$ is a limit diagram if and only if
	the restriction morphism in \cref{diffraction} is an equivalence. In that case by
	\cite[\href{https://kerodon.net/tag/02TF}{Remark 02TF}]{lurieKerodon}, the functor $\mathrm{Df}$ is
	an equivalence.

	Assume that for any cocartesian section $X:\ccal^\op\to \dcal$
	of $\pi$, the functor $c\mapsto (f_c)_*X(c)$ admits a limit in $\overline{D}_v$. By the preservation of limits of any right adjoint in $(\ccal^\op)^\lhd$ and
	\cite[\href{https://kerodon.net/tag/0311}{Corollary 0311} and \href{https://kerodon.net/tag/02KY}{Corollary 02KY}]{lurieKerodon}, this is equivalent to saying that for any
	cocartesian section $X:\ccal^\op\to\dcal$ of $\pi$, the lifting problem
	\begin{equation}
		\begin{tikzcd}
			\ccal^\op & {\overline{\dcal}} \\
			{(\ccal^\op)^\lhd} & {(\ccal^\op)^\lhd}
			\arrow["X", from=1-1, to=1-2]
			\arrow["{\overline{\pi}}", from=1-2, to=2-2]
			\arrow["i",hookrightarrow, from=1-1, to=2-1]
			\arrow[equal, from=2-1, to=2-2]
			\arrow["{\overline{X}}", dashed, from=2-1, to=1-2]
		\end{tikzcd}
	\end{equation}
	admits a solution $\overline{X}$ which is a $\overline{\pi}$-limit (\cite[\href{https://kerodon.net/tag/02KG}{Definition 02KG}]{lurieKerodon}).
	By \cite[\href{https://kerodon.net/tag/02ZA}{Example 02ZA}]{lurieKerodon}, in that case $\overline{X}$ is the right Kan extension functor along the inclusion $i:\ccal^\op\to(\ccal^\op)^\lhd$. This easily implies
	that the restriction functor $\mathrm{res}:\mathrm{Fun}_{/(\mathcal{C}^\op)^\lhd}^\mathrm{Cocart}((\mathcal{C}^\op)^\lhd,\overline{\mathcal{D}})\to \mathrm{Fun}_{/\mathcal{C}^\op}^\mathrm{Cocart}(\mathcal{C}^\op,\mathcal{D})$ is fully faithful as the left adjoint of the right Kan extension along $i : \ccal^\op\to(\ccal^\op)^\lhd$.

	Therefore, assuming point 2. of the proposition, $\chi$ is a limit diagram if and only if the restriction functor is conservative by
	\cite[\href{https://kerodon.net/tag/03UZ}{Corollary 03UZ}]{lurieKerodon}. But as the family of evaluations $\mathrm{ev}_c : \mathrm{Fun}(\ccal^\op,\dcal)\to\dcal$ is conservative, this is equivalent to the family of functors
	$\mathrm{ev}_c\circ\mathrm{res}=\mathrm{ev}_c\circ\mathrm{Df}\circ\mathrm{ev}_v = f_c^*$ being conservative, which is point 1. of the proposition.

	By \cref{diffraction}, if $\chi$ is a limit diagram, $\mathrm{Df}$ is an equivalence which implies that 1. holds. 
	To finish the proof of the proposition, we therefore only have to show that if $\chi$ is a limit diagram, then 2. holds. 
	Let $\overline{X}$ be a cocartesian section of $\overline{\pi}$. Assume first that $\overline{X}$ is a
	$\overline{\pi}$-limit diagram. By \cite[\href{https://kerodon.net/tag/03ZA}{Example 03ZA}]{lurieKerodon} this
	is equivalent to suppose that $\overline{X}$ is the $\overline{\pi}$-right Kan extension of its restriction $X$ to $\ccal^\op$.
	Then by \cite[\href{https://kerodon.net/tag/0307}{Corollary 0307}]{lurieKerodon}, the limit of
	$c\mapsto (f_c)_*X(c)$ exists and the fact that $\overline{X}$ is
	cocartesian is equivalent to the fact that for every $c\in\ccal$, the canonical map
	\begin{equation}
		f_c^*(\lim_{c'} (f_{c'})_*X(c))\to X(c)
	\end{equation}
	is an equivalence.

	Therefore, to show that 2. holds, we only have to show that any cocartesian section $\overline{X}$ of $\overline{\pi}$ is a
	$\overline{\pi}$-limit. This is exactly what is proved in the second paragraph of the proof of \cite[Theorem 5.17]{lurieDerivedAlgebraicGeometry2011}.

\end{proof}
\begin{cor}[See also {\cite[Corollary 5.2.2.37]{lurieHigherAlgebra2022}}]
	\label{corDescabs}
	Let $\chi,\chi':(\ccal^\op)^{\lhd}\to\catinfty$ be two functors as in \Cref{descabs}, and assume one has a natural transformation
	$B:\chi\to\chi'$ which commutes with the rights adjoints.
	Suppose also that the limits of point 2. in \Cref{descabs}  exist for $\chi$, that $B_v:\chi(v)\to\chi'(v)$ preserves them and that each $B_c$ for $c\in \ccal$ is conservative.
	Then if $\chi'$ is a limit diagram, so is $\chi$.
\end{cor}
\subsection{Enhancement of the $6$ functors formalism.}
\label{enhancementNM}
In this subsection, we will work with perverse Nori motives or mixed Hodge modules identically. Therefore we denote by $\mathrm{Var}_k$ 
the category of quasi-projective $k$-schemes when dealing with perverse Nori motives or the
category of separated and reduced finite type $\C$-schemes when dealing with mixed Hodge structures. We will call elements 
of $\mathrm{Var}_k$ varieties or $k$-varieties. For $X\in\mathrm{Var}_k$, we will denote by $\Mp(X)$ the category of perverse Nori motives constructed in \cite{ivorraFourOperationsPerverse2022}
or the category of Saito's mixed Hodge modules constructed in \cite{MR1047415}. We will have to use realisation functors 
hence we denote by $R: \Dd^b(\Mp(X))\to \Dd^b_c(X)$ either the Betti realisation of perverse Nori motives when $k\subset\C$, the 
underlying $\Q$-structure functor of mixed Hodge modules, or the $\ell$-adic realisation of perverse Nori motives. Here $\Dd^b_c(X)$ is the derived category of bounded cohomologically constructible sheaves on $X^\mathrm{an}$ in the two 
first cases, and the derived category of constructible $\ell$-adic étale sheaves on $X$. The triangulated category $\Dd^b_c(X)$ is endowed with a perverse and a constructible t-structures and the functor $\mathrm{rat}$ is t-exact (for both t-structures) and conservative. 

We begin with a recollection on the $\infty$-categorical enhancement of derived categories of constructible sheaves.

\begin{prop}
	There exists a functor 
	$$\dcal_c^b\colon\mathrm{Var}_k^\op \to \mathrm{CAlg}(\catinfty)$$
	such that for each map $f:Y\to X$ in $\mathrm{Var}_k$, the symmetric monoidal induced by $\dcal_c^b(f)$ on the homotopy categories is the classical pullback of constructible sheaves.
\end{prop}
\begin{proof}
	In both analytic sheaves an étale sheaves, for $X\in\mathrm{Var}_k$, the category $\Dd^b_c(X)$ is embedded in a bigger category $\dcal(X)$ which is of the form $\dcal(\mathrm{Shv}(\mathfrak{X},\oscr))$ with $\mathfrak{X}$ a topos and $\oscr$ a sheaf of rings on $\mathfrak{X}$ and is naturally a symmetric monoidal stable $\infty$-category. Indeed, in the case of analytic sheaves one can take $\mathfrak{X}$ to be the categories of analytic opens of the complex points of $X$ and $\oscr$ is the constant sheaf $\Q_X$. In the case of $\ell$-adic sheaves, following \cite{MR4609461} one can take $\mathfrak{X}$ to be the proétale topos of $X$, and $\oscr$ to be the sheaf of rings induced by the condensed ring $\Q_\ell = \Z_\ell[1/\ell]$. This construction has a canonical lift to the world of symmetric monoidal $\infty$-categories thanks to \cite[Remark 2.1.0.5]{lurieSpectralAlgebraicGeometry}. As the constructions $X\mapsto \mathfrak{X}$ and $\mathfrak{X}\mapsto \dcal(\mathrm{Shv}(\mathfrak{X},\oscr))$ are functorial, we obtain a functor 
	$$\dcal:\mathrm{Var}_k^\op \to\mathrm{CAlg}(\catinfty).$$
	Now as for each $f$ in $\mathrm{Var}_k$, the functor $f^*\colon \dcal(Y)\to\dcal(X)$ preserves constructible objects, this induces a functor
	$$\dcal_c^b:\mathrm{Var}_k^\op \to\mathrm{CAlg}(\catinfty)$$
	which gives the usual pullback on homotopy categories by definition in the case of analytic sheaves, and by \cite[Theorem 7.7]{MR4609461} for $\ell$-adic sheaves.

\end{proof}

\begin{constr}
	\label{constr1}
	We have a symmetric monoidal homotopy $2$-functor 
	$$\Dd^b(\Mp(-)):\mathrm{Var}_k^\op \to \mathrm{CAlg}(\mathrm{Cat}_1).$$
	By restricting to the constructible hearts, this induces a functor $$\Mc(-):\mathrm{Var}_k^\op \to \mathrm{CAlg}(\mathrm{Cat}_1)$$ such that the tensor product is exact in each variable and the pullback functors are exact (this can be checked after realisation). By applying \Cref{symMonoCatDeri} we obtain a functor 
	$$\dcal^b(\Mc(-)):\mathrm{Var}_k^\op \to \mathrm{CAlg}(\catinfty).$$
	Moreover, the equivalence \Cref{cccp} lifts to an equivalence of $\infty$-categories thanks to \Cref{equivHo}. The \Cref{replacement} then gives a functor 
	$$\dcal^b(\Mp(-)):\mathrm{Var}_k^\op \to \mathrm{CAlg}(\catinfty),$$ canonically equivalent to $\dcal^b(\Mc(-))$.
\end{constr}
\begin{prop}
	\label{samepb}
	Let $f:Y\to X$ be a map in $\mathrm{Var}_k$. Then the functor induced by 
	$\dcal^b(\Mc(f))$ on the homotopy categories is the pullback functor constructed by Ivorra and Morel or by Saito, depending on which case we are.
\end{prop}
\begin{proof}
	We claim that if we know that the pullback functors constructed by Ivorra, Morel and Saito admit $\infty$-categorical lifts $f^*$, then they will be canonically equivalent to $f^*_\mathrm{new} := \dcal^b(\Mc(f))$. Indeed in that case the $\infty$-functors $f^*$ and $f^*_\mathrm{new}$ coincide on the constructible heart by definition, hence they are equivalent by \Cref{RestrToHeart}.

	We prove this first for mixed Hodge modules: For any map $f:Y\to X$ of varieties, there is the factorisation 
	$$Y \xrightarrow{i} X\times Y \xrightarrow{p} X$$ where $i$ is the closed immersion given by the inclusion of the graph of $f$ (our varieties are separated) and $p$ is the projection. By definition, we have $f^*\simeq p^*\circ i^*$. The functor $p^*$ is given by the external product with the constant object: $p^*(M)=M\boxtimes \Q_Y$, hence is an $\infty$-functor because the external product is defined as the derived functor on the perverse hearts. The functor $i^*$ is the left adjoint of the functor $i_*$, which is obtained (\cite[(4.2.4) and (4.2.10)]{MR1047415}) by deriving the functor $i_*$ on the perverse heart and is therefore an $\infty$-functor. Thus by \Cref{adjointHo} the functor $i^*$ is an $\infty$-functor. Finally $f^*$ admits an $\infty$-categorical lift.

	Now for perverse Nori motives, any map $f:Y\to X$ being a map of quasi-projective varieties, factors as 
	$$Y\xrightarrow{i} Z\xrightarrow{g} X$$ with $i$ a closed immersion and $g$ a smooth map. Hence $f^*\simeq i^*g^*$. The case of $i^*$ is the same as for mixed Hodge modules, and thus we have an $\infty$-functor. For the smooth map $g$, we can assume that $X$ is connected (as a direct sum of $\infty$-functors is an $\infty$-functor), so that $g$ is smooth of pure relative dimension $d$ for some $d\in\N$. Now by definition $f^* = (f^*[d])[-d]$ where $f^*[d]$ is the derived functor of the exact functor $f^*[d]$ on the perverse heart, hence we have an $\infty$-functor.
\end{proof}
\begin{cor}
Let $f:Y\to X$ be a map in $\mathrm{Var}_k^\op$, then the $\infty$-functor 
$$f^*:\dcal^b(\Mp(X))\to \dcal^b(\Mp(Y))$$ admits a right adjoint $f_*$. If $f$ is moreover smooth, $f^*$ admits a left adjoint $f_\sharp$.
\end{cor}
\begin{proof}
	These adjoints exist on the homotopy categories, thus they have $\infty$-categorical lifts by \Cref{adjointHo}.
\end{proof}
\begin{prop}
	For each $X\in\mathrm{Var}_k$, the monoidal structure obtained in \Cref{constr1} induces on the homotopy categories the same monoidal structure as the one constructed by M. Saito for mixed Hodge modules and by Terenzi (\cite{terenziTensorStructurePerverse2024}) for perverse Nori motives.

	Moreover, the projection formula for a smooth morphism holds in $ \dcal^b(\Mp(X))$.
\end{prop}
\begin{proof}
	Note that in both \cite{terenziTensorStructurePerverse2024} and \cite{MR1047415} the tensor product is built in the following way	: there is
	an external tensor product $\boxtimes : \Mp(X)\times\Mp(X)\to\Mp(X\times X)$ that is exact in both variable hence induces by \Cref{CurryDb} an $\infty$-functor $\dcal^b(\Mp(X))\times \dcal^b(\Mp(X))\to \dcal^b(\Mp(X\times X))$
	which we compose with the pullback $\Delta^*$ under the diagonal $\Delta :X\to X\times X$ to obtain the tensor
	product, thus given by $M\otimes N = \Delta^*(M\boxtimes N)$. The other structure morphism of the monoidal structure are obtained similarly. Denote by $\otimes^T :  \dcal^b(\Mp(X))\times \dcal^b(\Mp(X))\to \dcal^b(\Mp(X))$ Saito's or Terenzi's tensor product and
	denote by $\otimes^\infty :  \dcal^b(\Mp(X))\times \dcal^b(\Mp(X))\to \dcal^b(\Mp(X))$ the one we just constructed.
	We want to check that $\mathrm{ho}(\otimes^\infty)=\mathrm{ho}(\otimes^T)$. For this there are several isomorphisms to check (see \cite[Section XI.1]{MR1712872}) :
	\begin{enumerate}
		\item A functorial isomorphism $\otimes^T\Rightarrow \otimes^\infty$.
		\item Modulo point 1. isomorphism, an identification of the unit isomorphism $\rho:M\otimes 1\to M$ and $\lambda:1\otimes M\to M$.
		\item Modulo point 1. isomorphism, an identification of the associativity isomorphism $c : M\otimes(N\otimes P)\to (M\otimes N)\otimes P$.
		\item Modulo point 1. isomorphism, an identification of the commutativity isomorphism $\gamma : M\otimes N\to N\otimes M$.
	\end{enumerate}
	By definition, $\otimes^\infty$ coincides with $\otimes^T$ when restricted to $\Mc(X)\times\Mc(X)$. By \Cref{CurryDb} this gives 1. We explain how to obtain 3. and the other checks will be left to the reader as they are very similar.
	Denote by $t_1^\infty : \dcal^b(\Mp(X))\times\dcal^b(\Mp(X))\times\dcal^b(\Mp(X))\to\dcal^b(\Mp(X))$ the functor sending $M,N,P$ to $M\otimes^\infty (N\otimes^\infty P)$ and by
	$t_2^\infty : \dcal^b(\Mp(X))\times\dcal^b(\Mp(X))\times\dcal^b(\Mp(X))\to\dcal^b(\Mp(X))$ the functor sending $M,N,P$ to $(M\otimes^\infty N)\otimes^\infty P$. The associativity isomorphism
	is a natural equivalence $c^\infty:t_1^\infty\Rightarrow t_2^\infty$. We give the same definition for $t_1^T$ , $t_2^T$ and $c^T$.
	Again by definition, the image by the restriction functor of the morphisms $c^\infty$ and $c^T$ are the same, hence up to the equivalence of 1., we obtain 3.

	For the projection formula, the arrow exists by a play of adjunctions, and is an equivalence on the homotopy categories.
\end{proof}
\begin{cor}
	Let $X$ be a variety. The tensor structure on $ \dcal^b(\Mp(X))$ is closed.
\end{cor}
\begin{proof}
	Again this is true on the homotopy categories and we can apply \Cref{adjointHo}.
\end{proof}

\begin{prop}
	Let $X$ be a variety. The realisation functor 
	$$R:\Dd^b(\Mp(X))\to \Dd^b_c(X)$$ admits a canonical lift to a map in $\mathrm{CAlg}(\catinfty)$. 

	Moreover, it induces a natural transformation $$\dcal^b(\Mp(-))\Rightarrow \dcal^b_c(-)$$ on functors 
	$$\mathrm{Var}_k^\op \to \mathrm{CAlg}(\catinfty).$$
\end{prop}
\begin{proof}
	In the case of the Betti realisation this is easy as we have an equivalence 
	$$\dcal^b(\mathrm{Cons}(X))\simeq \dcal_c^b(X)$$ thanks to Nori's theorem. This enables us to consider the functor $\Delta^1\times\mathrm{Var}_k^\op\to\mathrm{CAlg}(\mathrm{Cat}_1)$ that send $(0,f)$ to $f^*:\Mc(Y)\to \Mc(X)$, the edge $(1,f)$ to $f^*:\mathrm{Cons}(Y)\to\mathrm{Cons}(X)$ and $(0\to 1,X)$ to the realisation functor. Therefore by \Cref{symMonoCatDeri} we obtain the result.

	For the $\ell$-adic realisation it is not true that the canonical functor 
	$$\mathrm{real}_X:\dcal^b(\mathrm{Cons}(X,\Q_\ell))\to \dcal^b_c(X,\Q_\ell)$$ is an equivalence. What we did in the case of the Betti realisation gives a symmetric monoidal natural transformation $$\dcal^b(\Mc(-))\Rightarrow \dcal^b(\mathrm{Cons}(-,\Q_\ell)).$$ 
	Hence we have to check that the $\mathrm{real}_X$ assemble to give a symmetric monoidal natural transformation $$\dcal^b(\mathrm{Cons}(-,\Q_\ell))\Rightarrow\dcal^b_c(-).$$ This is \Cref{SymMonoReal}.
\end{proof}

\begin{cor}
The realisation functor $$\rcal \colon \dcal^b(\Mp(-))\to\dcal^b_c(-)$$ commutes with all the operations constructed above. This includes: pullbacks $f^*$, pushforwards $f_*$, left adjoints $f_\sharp$ when $f$ is a smooth map, tensor product $-\otimes-$ and internal homomorphism $\sHom(-,-)$.
\end{cor}
\begin{proof}
	Indeed, this is true on the homotopy categories, and the exchange morphisms exist because $\rcal$ commutes with pullbacks and tensor products by definition.
\end{proof}

Verdier duality also lifts:
\begin{prop}
	Let $X$ be a $k$-variety. Then the Verdier duality functor $\mathbb{D}_X$ has a $\infty$-categorical lift $\mathbb{D}_X\colon \dcal^b(\Mp(X))\to\dcal^b(\Mp(X))^\op$ which commutes with the Betti realisation. Moreover, there is a canonical isomorphism $\mathbb{D}_X\circ\mathbb{D}_X\simeq \mathrm{Id}_{\dcal^b(\Mp(X))}$.
\end{prop}
\begin{proof}
	Indeed the functor $\mathbb{D}_X$ both in Ivorra-Morel and in Saito, is constructed as the trivial derived functor of the exact functor $\mathbb{D}_X\colon \Mp(X)\to\Mp(X)^\op$ thus it is an $\infty$-functor. Thus the functor also commutes with the Betti realisation as an $\infty$-functor (as  in fact we apply the functor $\dcal^b(-)$ to the commutative square involving Verdier duality of $\Mp(X)$, $\mathrm{Perv}(X)$ and $R$). The fact that $\mathbb{D}_X\circ\mathbb{D}_X\simeq \mathrm{Id}_{\dcal^b_\mcal(X)}$ holds can be checked on the homotopy category as well.
\end{proof}

For free, we also obtain the exceptional functoriality:
\begin{cor}
	\label{excfuncto}
	The exceptional functors $f^!$ and $f_!$ constructed by Ivorra, Morel and Saito have $\infty$-categorical lifts that assemble to give functors 
	$$\Sch_k^{(\op)}\to\catinfty.$$ Moreover Verdier duality is a natural isomorphism between the star functoriality ($f_*$ or $f^*$) and the shriek functoriality of the opposite category ($f_!$ or $f^!$).
\end{cor}
\begin{proof}
	The last statement of this corollary gives a construction. Indeed the formula $\mathbb{D}\circ f_* \simeq f_!\circ\mathbb{D}$ tells us that if we apply \Cref{replacement} to the isomorphisms $\mathbb{D}_X : \dcal^b(\Mp(X))\to\dcal^b(\Mp(X))^\op$, the obtained functor $\Sch^{(\op)}\to\catinfty$ encode exactly the exceptional functoriality.
\end{proof}

\begin{thm}
	\label{etNori}
	The functor $X\mapsto \dcal^b(\Mp(X))$, with respect to the $(-)^*$-functoriality, is an hypersheaf with respect to $h$-topology.
\end{thm}
\begin{proof}
	Denote by $\dnori(X):=\dcal^b(\Mp(X))$.
	By \cite[Corollary 3.3.38]{MR3971240} (or see the proof of \cite[Proposition 5.11]{MR4278670}), it suffices to prove that this functor is an étale hypersheaf because of localisation and proper
	base change hold for $\dnori$. We will use \Cref{corDescabs} with a realisation functor (either the $\ell$-adic realisation for Nori motives, or the underlying $\Q$-structure functor for mixed Hodge modules) that we will denote by $$B\colon \dnori \to\Dd^b_c.$$ Note that $\Dd^b_c$ has étale hyper descent. Now, to apply \Cref{corDescabs} we need to make sure that the limits ensuring descent exist in $\dnori$. This is not obvious. However for any variety $X$ and $m\in\N$, if we denote by $\dcal^{[-m,m]}_\mcal(X)$ the full subcategory of $\dnori(X)$ of objects concentrated, for the constructible t-structure, in degrees $[-m,m]$, we have 
	\[\bigcup_m \dcal^{[-m,m]}_\mcal(X) = \dnori(X).\]
	 Moreover, the functor $\dnori$ is a hyper sheaf if and only if for each $m$, the functor $\dcal^{[-m,m]}_\mcal$ is a hyper sheaf, and the same holds for $\D^b_c$. Indeed, if $\dnori$ is a hyper sheaf, then as being concentrated in degrees $[-m,m]$ can be tested after pulling back by any surjective map of schemes, it turns out that each $\dcal^{[-m,m]}_\mcal$ is also a hyper sheaf. Conversely, if each $\dcal^{[-m,m]}_\mcal$ is a hyper sheaf, then because pullback functors $f^*$ are t-exact functors, in turns out that the conditions to test that $\dnori$ is a hypersheaf only takes place in $\dcal^{[-m,m]}_\mcal$: if $U_\bullet\to X$ is a hyper covering, fully faithfulness of the sheaf condition is computed in $\dcal^{[-m,m]}_\mcal(X)$ for some $m$, and then essential surjectivity holds because given an element $K=(K_n)_n$ of the limit there exists a fixed $m$ such that $K$ has its $K_0$ hence also all its factors $K_n$ in $\dcal^{[-m,m]}_\mcal(U_n)$.

	 Thus because $\Dd^{[-m,m]}_c$ is an étale hyper sheaf, it suffices to check the conditions of \Cref{corDescabs} for the realisation functor 
	 $$B\colon \dcal^{[-m,m]}_\mcal\to\Dd^{[-m,m]}_c.$$ Let $f_\bullet\colon Y_\bullet\to X$ be an étale hyper covering. As the realisation functor is conservative,  it suffices to check that:
	 \begin{enumerate}
	   \item For each complex $M$ in $\dcal^{[-m,m]}_\mcal(X)$, the limit \[\lim_\Delta \tau^{\leqslant m}(f_n)_*f_n^*M\] exists, where $\Delta$ is the simplicial category and $\tau^{\leqslant m}$ is the truncation functor for the t-structure.
	   \item The realisation functor \[\dcal_\mcal^{[-m,m]}(X)\to\Dd_c^{[-m,m]}(X)\] commutes with such a limit.
	 \end{enumerate}
	 In fact as $\dcal^{[-m,m]}_\mcal(X)$ is a $(2m+1)$-category, the above totalisation in (1) is finite thus exists (because the finite limit $\lim (f_n)_*f_n^*M$ exists in $\dcal_\mcal^{\geqslant -m}$ because it exists in $\dcal^b_\mcal$, and then we apply the right adjoint $\tau^{\leqslant m}$), and the realisation functor clearly commutes with such a finite limit because the realisation functor
	 \[\dcal^{\geqslant -m}_\mcal(X)\to\Dd_c^{\geqslant -m}(X)\] commutes with finite limits (the inclusion $\Dd^{\geqslant -m}\to\Dd^b$ commutes with limits) and the image by the realisation functor of the above limit in (1) is the truncation by $\tau^{\leqslant m}$, which is a right adjoint, of a limit in $\Dd^{\geqslant -m}_c(X)$.

\end{proof}

\section{The realisation functor and applications.}
\label{Sectionrealisation}

\subsection{Construction of the realisation functor.}
\label{realisationSection}
In this section, we construct realisation functors from the category of étale motives to the derived categories of mixed Hodge modules and perverse Nori motives that commute with the 6 operations.  Denote by $\Sch_k$ the category of finite type $k$-schemes.
We use Drew and Gallauer' result in \cite{MR4560376} in which they prove that the stable motivic $\A^1$-invariant $\infty$-category $\mathcal{SH}$ has a universal property among systems of $\infty$-categories called \emph{coefficient systems}: these are functors 
$$\Sch_k^\op \to \mathrm{CAlg}(\catinfty)$$ that have the minimal functoriality needed for the 6 operations to exist thanks to Ayoub's thesis (see \Cref{CoeffSysAppe} for a precise definition and some results on coefficient systems). Any coefficient system $C\in\mathrm{CoSys}_k$ with values in idempotent complete $\infty$-categories receives an unique realisation functor $$\mathcal{SH}_c\to C$$ from the compact objects of the motivic homotopy category that commutes with enough operations to ensures that it commutes with the 6 operations (if $C$ has some good properties, see below).

\begin{thm}
	\label{NoriPB}
	Denote by $\dcal_\mcal^b$ the functor $\mathrm{Var}_k^\op \to\mathrm{CAlg}(\mathrm{\catinfty})$ of \Cref{constr1}. The functor $\dcal_\mcal^b$ defines an object of $\mathrm{CoSys}_k$.
\end{thm}
\begin{proof}
	This is the content of \Cref{enhancementNM} and \Cref{extQProj}. See \Cref{defCoSys} for the precise definition of coefficient systems.
\end{proof}

The above theorem contains the statement that perverse Nori motives and mixed Hodge modules extend to every finite type $k$-scheme, together with all the operations and $h$-hyperdescent (the only non trivial statement here is the h-hyperdescent, but this follows from the fact that the Kan extension of an hypersheaf on a basis of a site to the full site is still an hypersheaf, see for example \cite[Theorem A.6]{MR4549105}).

\begin{rem}
	Using that the mapping spectra in $\dcal^b_\mcal$, $\dcal^b(\Mp)$ and $\dcal^b(\Mc)$ are étale sheaves (for the perverse heart, one has to use affine étale morphisms) and the fact that a limit of a diagram of stable $\infty$-categories with $t$-structures and $t$-exact exact transitions functors is endowed with a $t$-structure (see \cite[Lemma 3.2.18]{MR4061978}), one can show that after extension to finite type $k$-schemes, we still have $$\dcal^b_\mcal(X)\simeq \dcal^b(\Mc(X))\simeq \dcal^b(\Mp(X))$$ and in the case of perverse Nori motives, the category $\Mp(X)$ is the universal category as for quasi-projective schemes. The weights also extend to that case.
\end{rem}

\begin{nota}
	For $X$ a finite type $k$-scheme, we will denote by $\dcal^b(\mathrm{MHM}(X))$ (\emph{resp.} by $\dcal^b(\Mp(X))$) the value at $X$ of the functor $\dnori$ in the mixed Hodge modules case (\emph{resp.} in the Nori motives case).
\end{nota}

Before going into the proof of the main theorem, we recall a construction originally due to Fabien Morel, and we will use \cite[Section 16.2]{MR3971240} as a reference. 
Let $X$ be a finite type $k$-scheme. There is a canonical decomposition $$\mathcal{SH}_\Q(X)\simeq \mathcal{SH}_\Q^+(X)\times \mathcal{SH}_\Q^-(X)$$ of the rational part of the motivic homotopy category $\mathcal{SH}(X)$ over $X$. It comes from the decomposition in $\mathcal{SH}_\Q(X)$ of the unit $\Q_X$ as a direct sum 
\[\Q_X\simeq \Q_X^+\oplus \Q_X^-.\]
Moreover the object $\Q_X^+$ has a commutative algebra structure (because it is idempotent) in $\mathcal{SH}(X)_\Q$ such that the projection map $\Q_X\to\Q_X^+$ is a map of algebras. By \cite[Theorem 16.2.13]{MR3971240}, there is a canonical identification 
of $\mathcal{SH}_\Q^+(X)\simeq \mathrm{Mod}_{\Q^+}(\mathcal{SH}_\Q(X))$ 
with étale motives $\mathcal{DM}^{\mathrm{\'et}}(X)$. 
In particular, a symmetric monoidal colimit preserving $\infty$-functor 
\[F\colon \mathcal{SH}_\Q(X)\to \dcal\] to a stable presentably symmetric monoidal $\infty$-category $\dcal$ factors through $\mathcal{DM}^{\mathrm{\'et}}(X)$ if and only if $F(\Q^-_X)=0$. 
By \cite[Corollary 16.2.14]{MR3971240} the object $\Q^-_X\in\mathcal{SH}_\Q(X)$ vanishes whenever $-1$ is a sum of squares in the residual fields of $X$. 
We will need the following lemma:

\begin{lem}
	\label{LemmaBrad}
	The maps $\mathcal{SH}_\Q(X)\to\mathcal{DM}^{\mathrm{\'et}}(X)$ and 
	$\mathcal{SH}(X)\to\mathcal{SH}_\Q(X)$ are part of morphisms of coefficient systems $\mathcal{SH}_\Q\to \mathcal{DM}^{\mathrm{\'et}}$ and 
	$\mathcal{SH}\to\mathcal{SH}_\Q$.
\end{lem}
\begin{proof}
	This is a particular case of a theorem of Brad Drew (although in the second case of rationalisation a simpler proof could have been given see \cite[Lemma A.11]{MR3205601}): \cite[Theorem 8.10]{drewMotivicHodgeModules2018}, where for example in the first case we take (with his notations) $\mathcal{V} = \mathcal{SH}_\Q(\mathrm{Spec}(k))$, $\mathcal{M}^*=\mathcal{SH}_\Q$ and $A = \Q^+_{\mathrm{Spec}(k)}$.
\end{proof}

\begin{thm}
	\label{realNori}
	There exists an (essentially unique) morphism of coefficients systems $\mathcal{DM}^{\mathrm{\'et}}\to \Ind \dnori$ on the category $\Sch_k$. It restricts to
	a functor  \begin{equation}\label{reali}\nu:\mathcal{DM}^{\mathrm{\'et}}_c\to \dnori.\end{equation}
	The functor \cref{reali} commutes with the $6$ operations, that is all the exchange transformations are isomorphisms. 
\end{thm}
In the case of mixed Hodge modules we will denote by $\mathrm{Hdg}^*$ the realisation functor, and in the case of Nori motives we will denote it by $\mathrm{Nor}^*$.
\begin{proof}
	Thanks to \Cref{NoriPB} we know that $\dcal^b_\mcal$ is an object of $\mathrm{CoSys}_k$, hence its indization $\Ind\dcal^b_\mcal$ is an object of $\mathrm{CoSys}_k^\mathrm{c}$. Together with \Cref{univSh}, this gives a morphism of coefficient systems $$\rho_{\mathcal{SH}}\colon\mathcal{SH}\to \Ind\dcal^b_\mcal.$$ 

	Note that the functor $\Ind\dcal^b_\mcal$ naturally takes values in $\mathrm{CAlg}(\mathrm{Pr}^L)_{\backslash \Ind\dcal^b_\mcal(\Spec(k))}$ the co-slice $\infty$-category, because $\Spec(k)$ is the initial object of $\Sch_k^\op$. Using the functor $\rho_{\mathcal{SH}}$ over $\Spec(k)$ we see that $\Ind\dcal^b_\mcal$ takes values in $\mathrm{CAlg}(\mathrm{Pr}^L)_{\backslash \mathcal{SH}(\Spec(k))}$ the $\infty$-category of $\mathcal{SH}(\Spec(k))$-algebras in $\mathrm{Pr}^L$, where we endowed $\mathrm{Pr}^L$ with the Lurie tensor product constructed in \cite[Section 4.8.1]{lurieHigherAlgebra2022}. We claim that $\rho_{\mathcal{SH},k}\colon\mathcal{SH}(\Spec(k))\to\Ind\dnori(\Spec(k))$ factors through $\mathcal{DM}^{\'et}(\Spec(k))$. 

	First, note that because $\Ind\dnori(\Spec(k))$ is $\Q$-linear, the functor $\rho_{\mathcal{SH},k}$ factors through $\mathcal{SH}_\Q(\Spec(k))$ to give a symmetric monoidal functor 
	\[\rho_{\mathcal{SH}_\Q,k}\colon \mathcal{SH}_\Q(\Spec(k))\to\Ind\dnori(\Spec(k)).\] Moreover by the universal property of the localisation, if $K$ is a finite extension of $k$ in which we added a square root of $-1$, the square \[ \begin{tikzcd}
		\mathcal{SH}_\Q(\Spec(k)) \arrow[r,"\rho_{\mathcal{SH}_\Q,k}"] \arrow[d,"g^*"] 
		  & \Ind\dnori(\Spec(k)) \arrow[d,"g^*"] \\
		\mathcal{SH}_\Q(\Spec(K)) \arrow[r,"\rho_{\mathcal{SH}_\Q,K}"]
		  & \Ind\dnori(\Spec(K))
	\end{tikzcd}\] commutes, where $g\colon \Spec(K)\to\Spec(k)$ it the structural morphism. 
	Now because $\Q^-_{\Spec(k)}$ is a direct factor of $\Q_{\Spec(k)}$, its image under $\rho_{\mathcal{SH}_\Q,k}$ lands in $\dnori(\Spec(k))$. 
	By étale descent of $\dnori$ proven in \Cref{etNori} the functor $g^*$ is conservative on $\dnori(\Spec(k))$, and as $\Q^-_{\Spec(K)}=0$, we see that the image of $\Q^-_{\Spec(k)}$ in $\Ind\dnori(\Spec(k))$ vanishes. 
	By the discussion above \Cref{LemmaBrad} this proves that the functor $\rho_{\mathcal{SH}_\Q,k}$ factors through $\mathcal{DM}^\mathrm{\'et}(\Spec(k))$ to give a functor that we will denote by $\widetilde{\nu_k}$. 
	Through $\widetilde{\nu_k}$ we see that our functor $\Ind\dnori$ takes values in $\mathcal{DM}^{\mathrm{\'et}}(\Spec(k))$-algebras in $\mathrm{Pr}^L$. 
		By \cite[Remark 7.3.2.13]{lurieHigherAlgebra2022} applied to the left adjoint functor \[-\otimes_{\mathcal{SH}(\Spec(k))}\mathcal{DM}^{\mathrm{\acute{e}t}}(\Spec(k))\colon \mathrm{Mod}_{\mathcal{SH}(\Spec(k))}(\mathrm{Pr}^L)\to\mathrm{Mod}_{\mathcal{DM}^{\mathrm{\acute{e}t}}(\Spec(k))}(\mathrm{Pr}^L),\] the forgetful functor \[\mathrm{ff}\colon\mathrm{CAlg}(\mathrm{Pr}^L)_{\backslash \mathcal{DM}^\mathrm{\acute{e}t}(\Spec(k))}\to \mathrm{CAlg}(\mathrm{Pr}^L)_{\backslash \mathcal{SH}(\Spec(k))}\] is right adjoint to the functor 
	\[-\otimes_{\mathcal{SH}(\Spec(k))}\mathcal{DM}^\mathrm{\acute{e}t}(\Spec(k)).\]
	Applying this adjunction term-wise we obtain by \cite[Theorem 6.1.22]{MR3931682} an adjunction between functors on $\Sch^\op_k$. In particular, the natural transformation $\rho_{\mathcal{SH}}$ may be seen as a map 
	$\mathcal{SH}\to \mathrm{ff}(\Ind\dnori)$ in the $\infty$-category 
	$$ \mathrm{Fun}(\Sch_k^\op,\mathrm{CAlg}(\mathrm{Pr}^L)_{\backslash \mathcal{SH}(\Spec(k))}).$$ By adjunction, we obtain a natural transformation
	\[\mathcal{SH}\otimes_{\mathcal{SH}(\Spec(k))}\mathcal{DM}^\mathrm{\'et}(\Spec(k))\to \Ind\dnori\] of functors on $\Sch_k^\op$ and with values in $\mathrm{CAlg}(\mathrm{Pr}^L)_{\backslash \mathcal{DM}^\mathrm{\'et}(\Spec(k))}$. The left hand sides sends a scheme $X$ to the $\infty$-category 
	\[\mathcal{SH}(X)\otimes_{\mathcal{SH}(\Spec(k))}\mathcal{DM}^{\mathrm{\'et}}{(\Spec(k))}.\] By using that $$\mathcal{DM}^{\mathrm{\'et}}(\Spec(k))\simeq \mathrm{Mod}_{\Q^+_k}(\mathcal{SH}(\Spec(k)))$$ and Lurie's \cite[Theorem 4.8.4.6]{lurieHigherAlgebra2022} we see that
	$$\mathcal{SH}(X)\otimes_{\mathcal{SH}(\Spec(k))}\mathcal{DM}^{\mathrm{\'et}}(\Spec(k))\simeq \mathrm{Mod}_{\Q^+_X}(\mathcal{SH}(X))\simeq\mathcal{DM}^\mathrm{\'et}(X).$$ This proves that $\rho_\mathcal{SH}$ factors through the functor $\mathcal{DM}^\mathrm{\'et}$, giving us a natural transformation
	 \[\widetilde{\nu}\colon\mathcal{DM}^\mathrm{\'et}\to\Ind\dnori\] of functors on $\Sch^\op_k$ with values in $\mathrm{CAlg}(\mathrm{Pr}^L)$. The factorisation is given by the unit of the adjunction between functor categories described above.

	We claim that we obtained a morphism of coefficient systems. Let $f\colon X\to Y$ be a smooth map. Because $\mathcal{DM}^\mathrm{\'et}(X)$ is generated under colimits by objects of the form $a_\mathrm{\'et}^\Q(M)$ with $M\in\mathcal{SH}(X)$ and $a_\mathrm{\'et}^\Q$ the morphism $\mathcal{SH}\to\mathcal{DM}^\mathrm{\'et}$, it suffices to prove that the natural exchange morphism 
	\[f_\sharp(\widetilde{\nu}(a_\mathrm{\'et}(M)))\to \widetilde{\nu}(f_\sharp^{\mathcal{DM}^\mathrm{\'et}}(a_\mathrm{\'et}(M)))\] is an equivalence. This follows from the following sequence of equivalences:
	\begin{eqnarray*}
		f_\sharp(\widetilde{\nu}(a_\mathrm{\'et}(M))) & \simeq & f_\sharp (\rho(M))\\ 
		& \overset{(1)}{\simeq} & \rho(f_\sharp^{\mathcal{SH}}(M))\\
		& \simeq & \widetilde{\nu}(a_\mathrm{\'et}^\Q(f_\sharp^\mathcal{SH}(M)))\\
		&\overset{(2)}{\simeq} & \widetilde{\nu}(f_\sharp^{\mathcal{DM}^\mathrm{\'et}}(a_\mathrm{\'et}^\Q(M)))
		,\end{eqnarray*} where we used in (1) the fact that $\rho_{\mathcal{SH}}$ is a morphism of coefficient systems and in (2) the fact that $a_\mathrm{\'et}^\Q$ is a morphism of coefficient systems by \Cref{LemmaBrad}.

	Denote by $\mathcal{DM}^{\mathrm{\'et}}_c(X)$ (\emph{resp.} by $\dcal^{\mathrm{gm}}_\mcal(X)$)
	the thick full subcategory
	of $\mathcal{DM}^{\mathrm{\'et}}(X)$ (\emph{resp.} of $\Ind\dnori(X)$) spanned by the $f_\sharp \Q_Y(i)$ for $f:Y\to X$ smooth
	and $i\in \Z$. As $\dnori(X)\subset \Ind\dnori(X)$ is thick and each $f_\sharp \Q_Y(i)$ is in $\dnori(X)$, we have $\dcal^{\mathrm{gm}}_\mcal(X)\subset \dnori(X)$.
	Because $\widetilde{\nu}$ is a morphism of coefficient systems, it induces a morphism of coefficient systems
	\begin{equation}
		\nu^{\mathrm{gm}}:\mathcal{DM}^{\mathrm{\'et}}_c\to \dcal^{\mathrm{gm}}_\mcal.
	\end{equation}
	In particular, at the level of triangulated categories, we obtain a premotivic morphism that verifies all the hypotheses of
	\cite[Theorem 4.4.25]{MR3971240} : $\mathcal{D}^{\mathrm{gm}}_\mcal$ is $\tau$-generated for $\tau$ the Tate twist,
	$\mathcal{DM}^{\mathrm{\'et}}_c$ is  $\tau$-dualisable by absolute purity, $\mathcal{D}^{\mathrm{gm}}_\mcal$ is separated
	because it is a sub pullback formalism of $\dnori$ which satisfy étale descent by \Cref{etNori} (which is equivalent to separateness in the $\Q$-linear case by \cite[Corollairy 3.3.34]{MR3971240}) and
	all Tate twists are invertible in $\mathcal{DM}^{\mathrm{\'et}}_c$ by construction.

	Therefore, the functor $\nu^{\mathrm{gm}}$ commutes with the 6 operations. One has to be a little careful here: a priori, \cite[Theorem 4.4.25]{MR3971240}
	implies only that $\nu^{\mathrm{gm}}$ commutes with all the operations when $\dcal^{\mathrm{gm}}_\mcal$ is endowed with the operations they
	construct in their book, and not necessarily with the constructions of Saito, Ivorra, Morel as in \Cref{excfuncto}. We know that for $f:Y\to X$ a morphism, $f^*$ is the same for both constructions, hence by adjunction this is also the case for $f_*$.
	For $j:U\to X$ an open immersion the $j_\sharp$ are the same, but $j_\sharp = j_!$ hence by Nagata compactification (and the fact that for a proper $p$, we have $p_*\simeq p_!$) and composition
	both $f_!$ are isomorphic for any $f:Y\to X$ separated. By adjunction we also have an isomorphism between the $f^!$. This defines a isomorphism of étale sheaves of $\infty$-categories, hence we also have compatibility of operations for non separated morphisms (when they are defined for $\mathcal{DM}^{\mathrm{\'et}}$).
	The tensor product is the same, hence also the internal $\sHom$ by adjunction.

	We have obtained that the functor
	\begin{equation}
		\nu : \mathcal{DM}^{\mathrm{\'et}}_c\to \dnori
	\end{equation}
	commute with the 6 operations at the triangulated categories level. But all exchange morphisms exist at the level of $\infty$-categories,
	and the functor from an $\infty$-category to its homotopy category is conservative, hence $\nu$ commutes with all the operations.
\end{proof}
\begin{cor}
	The realisation functor $$\mathcal{SH}_c\to \dnori$$ commutes with the 6 operations.
\end{cor}
\begin{proof}
	Indeed it suffices to check that the functor $\mathcal{SH}_c\to\mathcal{DM}^{\mathrm{\'et}}_c$ commutes with the operations. This is the content of the next \Cref{lemmaSHop}. 
\end{proof}
This lemma is probably well known to experts but the author could not find a reference.
\begin{lem}
	\label{lemmaSHop}
	The morphisms $\rho_\Q\colon\mathcal{SH}_c\to\mathcal{SH}_{\Q,c}$ and $a_{\mathrm{\'et}}\colon\mathcal{SH}_{\Q,c}\to\mathcal{DM}^\mathrm{\'et}_c$ commute with the 6 operations.
\end{lem}
\begin{proof}
	The case of $\rho_\Q$ follows from \cite[Théorème A.15]{MR3205601}, thus we focus on $a_{\mathrm{\'et}}$.
	We know that $a_{\mathrm{\'et}}$ is a morphism of coefficient systems generated by objects of the form $g_\sharp\mathbf{1}_Y(n)$ for $g\colon Y\to X$ a smooth map and $n\in\Z$, thus to prove that it commutes with the 6 operations it suffices to check that it commutes with pushforwards (see the proof of \cite[Theorem 4.4.25]{MR3971240}). Let $f\colon Y\to X$ be a morphism of schemes.
	As the functor $a_{\mathrm{\'et}}$ commutes with $f^*$ its right adjoint $\mathrm{ff}^{\mathrm{\'et}}$ commutes with $f_*$, and because $\mathrm{ff}^{\mathrm{\'et}}$ is conservative, it suffices to check that the natural transformation \[\mathrm{ff}^{\mathrm{\'et}}\circ a_{\mathrm{\'et}}\circ f_*\to f_*\circ\mathrm{ff}^{\mathrm{\'et}}\circ a_{\mathrm{\'et}}\] is an equivalence. Moreover by \cite[Corollary 4.2.4.8]{lurieHigherAlgebra2022} there is a canonical identification of functors \[\mathrm{ff}^{\mathrm{\'et}}\circ a_{\mathrm{\'et}} \simeq -\otimes {\Q^+}.\] Thus we have to prove that the natural map
	\begin{equation}\ f_*(-)\otimes \Q^+ \to f_*(-\otimes \Q^+)\end{equation} is an equivalence.
	This is the case because as $\Q^+$ is a direct factor of $\Q$, it is a dualisable object of $\mathcal{SH}_\Q$ so that we can apply \cite[Lemma 2.8]{MR3259031}, using that $\Q^+_Y\simeq f^*\Q^+_X$.
\end{proof}
\begin{cor}
	There exists a Hodge realisation functor $R_H:\dcal^b(\Mp)\to\dcal^b(\mathrm{MHM})$ compatible with the six operations. It factors the Hodge realisation that we have constructed in \Cref{realNori}.
\end{cor}
\begin{proof}
	By \Cref{realNori}, there exists a Hodge realisation from $\mathcal{DM}^{\mathrm{\'et}}_c(-,\Q)$ to the derived category of mixed Hodge modules. Therefore, taking the underlying $\Q$-structure and then perverse $\HHp^0$,
	we obtain a factorisation of the perverse $\HHp^0$ of the Betti realisation through $\mathrm{MHM}(X)$. The universal property of perverse Nori motives then gives a faithful exact functor $\Mp(X)\to \mathrm{MHM}(X)$,
	which induces a functor $\dcal^b(\Mp(X))\to \dcal^b(\mathrm{MHM}(X))$. Now all the constructions of \cite{ivorraFourOperationsPerverse2022} are the constructions of the operations for mixed Hodge modules,
	hence this functor is compatible with the operations. The fact that this functor factors the Hodge realisation we constructed above follows from the universal property of $\mathcal{DM}^{\mathrm{\'et}}_c$ and the fact that $R_H$ can be promoted to a natural transformation of $\infty$-functors on $\Sch_k^\op$, using that it is also the derived functor of its restriction to the constructible heart.
\end{proof}
Recall that by \cite[Corollary 6.18]{ivorraFourOperationsPerverse2022}, \cite[Proposition 2.3.1]{bondarkoWeightsTstructuresGeneral2011} and \cite[Theorem 3.3]{MR2834728}
the categories $\dcal^b(\Mp(X))$, $\dcal^b(\mathrm{MHM}(X))$ and $\mathcal{DM}^{\mathrm{\'et}}_c(X)$ admit  weight structures (for the last one, note that Beilinson motives and étale motives are equivalent by \cite[Theorem 16.1.2]{MR3971240}).
\begin{prop}
	\label{weightexact}
	Let $X$ be a finite type $k$-scheme. Then the functors \[R_H:\dcal^b(\Mp(X))\to\dcal^b(\mathrm{MHM}(X)),\]
	\[\mathrm{Nor}^*:\mathcal{DM}^{\mathrm{\'et}}_c(X)\to\dcal^b(\Mp(X))\]
	and
	\[\mathrm{Hdg}^*:\mathcal{DM}^{\mathrm{\'et}}_c(X)\to\dcal^b(\mathrm{MHM}(X))\]
	are compatible with weights.
\end{prop}
\begin{proof}
	First note that by definition of the weights on $\dcal^b(\Mp(X))$, the $\ell$-adic realisation $R_\ell :\dcal^b(\Mp(X))\to\dcal^b_c(X,\Q_\ell)$ is compatible with weights and reflects weights (\emph{i.e.} if a complex $K\in\dcal^b(\Mp(X))$ satisfies that $R_\ell(K)$ is of weights $\leqslant w$, then $K$ is of weights $\leqslant w$).
	As all weight filtrations and structures considered are bounded, it suffices to show that our functors map weight zero objects to weight zero objects. Let us first consider functors with source $\mathcal{DM}^{\mathrm{\'et}}_c(X)$.
	For those, is suffices to check that the image of the relative motive of a smooth proper $X$-scheme is pure of weight $0$. For the $\ell$-adic realisation this is \cite{MR0601520} and \cite{morelMixedAdicComplexes}. Therefore, the functor $\mathrm{Nor}^* : \mathcal{DM}^{\mathrm{\'et}}_c(X)\to\dcal^b(\Mp(X))$ is also weight exact.
	The fact that $\mathrm{Hdg}^*$ is exact is the fact that the mixed Hodge modules associated to a smooth proper $X$-scheme is pure of weight $0$, see \cite[Theorem 6.7 and Proposition 6.6]{saitoFormalismeMixedSheaves2006}.

	Now is only left $R_H : \dcal^b(\Mp(X))\to\dcal^b(\mathrm{MHM}(X))$. By \cite[Corollary 6.27]{ivorraFourOperationsPerverse2022}, any pure complex of perverse Nori motives is the direct sum
	of shifts of its $\HHp^i$, hence to check that $R_H$ is weight exact it suffices to show that it sends pure objects of $\Mp(X)$ to pure objects of $\mathrm{MHM}(X)$. By construction of $R_H$,
	we have a commutative diagram:\begin{equation}
		\begin{tikzcd}
			{\mathcal{DM}^{\mathrm{\'et}}_c(X)} & {\Mp(X)} & {\mathrm{MHM}(X)}
			\arrow["{\HH_u}", from=1-1, to=1-2]
			\arrow["{R_H}", from=1-2, to=1-3]
			\arrow["{\HHp^0\mathrm{Hdg}^*}"', from=1-1, to=1-3, bend left=30]
		\end{tikzcd}.
	\end{equation}
	With $H_u$ the universal functor defining $\Mp(X)$.
	But by \cite[Corollary 6.27]{ivorraFourOperationsPerverse2022}, $H_u$ sends pure objects to pure objects, and any motive $M\in\Mp(X)$ is a quotient of some $\HH_u(N)$ for $N\in \mathcal{DM}^{\mathrm{\'et}}_c(X)$.
	Assume that $M$ is pure of some weight $w$. Then $M$ is a direct factor of a quotient of the weight filtration of $\HH^0(N)$. As $\mathrm{Hdg}^*$ is weight exact, it follows that $R_H(M)$ is a direct factor of a weight $w$ mixed Hodge module, hence $R_H$ is weight exact.
\end{proof}

\begin{rem}
	Our construction of the realisation functor $\mathcal{DM}^{\mathrm{\'et}}_c(X)$ is very similar to that of Ivorra in \cite{MR3518311} (which works for smooth varieties).
	Indeed, given a smooth variety $X$, in both definitions one first gives the image of the smooth $X$-schemes in $\dcal^b(\Mp(X))$. To be more precise, Ivorra only gives the image of smooth affine $X$-schemes and then extend to smooth $X$-schemes by colimits but
	as the functor $(f:Y\to X)\mapsto f_!f^!\Q_X$ is a Nisnevich cosheaf, its value on smooth $X$-schemes is indeed determined by its value on smooth affine $X$-schemes. Once one
	has the value of the functor on smooth $X$-schemes, the rest of the construction is just applying the universal property of
	$\mathcal{DM}^{\mathrm{\'et}}_c(X)$ :
	one has to check that the functor on smooth $X$-schemes is an $\A^1$-invariant
	étale
	sheaf and
	$\mathbb{P}^1$-stable. Therefore,
	it is easy to see that if both definitions on smooth affine schemes (ours with the $6$-operations and Ivorra's using cellular complexes and Beilinson basic lemma) coincide, then our functor $\mathcal{DM}^{\mathrm{\'et}}_c(X)\to\dcal^b(\Mp(X))$ and Ivorra's are equivalent. In particular, this would prove that Ivorra's construction is compatible with the Betti realisation constructed by
	Ayoub on Voevodsky motives, hence that the category of perverse Nori motives as in \cite{ivorraFourOperationsPerverse2022} is
	equivalent as the diagram category of relative pairs studied in \cite{MR3723805} and \cite[Section 2.10]{ivorraFourOperationsPerverse2022}.

	This can be reduced to check a simple fact, that the author did not manage to do: in \cite[Proposition 4.15]{MR3518311} Ivorra proves that for a given smooth affine map $f:Y\to X$ with $X$ a smooth variety, there is a specific
	cellular stratification $Y^\bullet$ of $Y$ such that the cellular complex induced on $Y$ is an $\HH^0f_!$-acyclic resolution of $f^!\Q_X$, hence that there is an isomorphism in
	$\dcal^b(\Mp(X))$ between Ivorra's functor at $Y$ and the homology of $Y$, that is $\mathrm{ra}^\Mp_X(Y,Y_\bullet)\simeq f_!f^!\Q_X$. This isomorphism induces a canonical and functorial isomorphism on each $\HHp^i$. The question is whether one can construct such an isomorphism in a functorial way,
	as this would give the wanted isomorphism of functors. This does not seem easy, and it is unlikely that one can arrange those specific stratification so that they are compatible with a given morphism of smooth affine $X$-schemes.
\end{rem}

\subsection{Nori motives as modules in étale motives.}

\label{continuitySection}
\begin{nota}
	Recall that to be coherent with the notations of \cite{ayoubAnabelianPresentationMotivic2022}, we will denote by $\mathrm{Nor}^*$ the realisation functor $$\mathcal{DM}^{\mathrm{\'et}}\to\mathcal{DN}$$ where we have denoted by $\mathcal{DN}$ the indization of the bounded derived category of perverse motives. We will denote by $\mathcal{DN}_c$ its full subcategory of compact objects.
\end{nota}

By construction, the functor $\mathrm{Nor}^*$ preserves colimits, hence it has a right adjoint $\mathrm{Nor}_*:\mathcal{DN}\to \mathcal{DM}^{\mathrm{\'et}}$. 
Denote by \begin{equation}\label{defN}\nscr_k = \mathrm{Nor}_*\mathrm{Nor}^*\Q_{\Spec k}\in\mathcal{DM}^{\mathrm{\'et}}(\Spec k,\Q).\end{equation}This is a $\mathbb{E}_\infty$-ring object in $\mathcal{DM}^{\mathrm{\'et}}(k)$. For $X$ a finite type $k$-scheme, denote by $\nscr_{\mid X}\in\mathcal{DM}^{\mathrm{\'et}}(X)$ the restriction of $\nscr_k$ to $X$.
As in Ayoub's \cite[Construction 1.91]{ayoubAnabelianPresentationMotivic2022}, $\mathrm{Nor}^*$ factors
canonically through the functor \begin{equation}\mathrm{Mod}_\nscr(\mathcal{DM}^{\mathrm{\'et}}) : X\mapsto \mathrm{Mod}_{\nscr_{\mid X}}(\mathcal{DM}^{\mathrm{\'et}}(X,\Q))\end{equation}
as
\begin{equation}
	\label{factoModN}
	\mathrm{Nor}^* : \mathcal{DM}^{\mathrm{\'et}}\xrightarrow{\nscr\otimes -} \mathrm{Mod}_\nscr(\mathcal{DM}^{\mathrm{\'et}}) \xrightarrow{\widetilde{\mathrm{Nor}^*}} \mathcal{DN}
.\end{equation}

\begin{lem}
	\label{fullFaithNor}
	Let $k$ be a field of characteristic zero and let $\pi_X:X\to\Spec k$ be a finite type $k$-scheme. The functor 
	\[\widetilde{\mathrm{Nor}^*}:\mathrm{Mod}_{\pi_X^*\nscr_k}(\mathcal{DM}^{\mathrm{\'et}}(X))\to \mathcal{DN}(X)\]
	is fully faithful. Moreover he natural map $\pi_X^*\mathrm{Nor}_*\Q_{\Spec k}\to\mathrm{Nor}_*\Q_X$ 
	is an equivalence.
\end{lem}
\begin{proof}
	The proof adds nothing new to \cite[Remark 1.97]{ayoubAnabelianPresentationMotivic2022}. 
	Denote by $\widetilde{\mathrm{Nor}_*}$ the right adjoint of $\widetilde{\mathrm{Nor}^*}$. Assume first that $X = \Spec k$. We have to show that the unit map $$\mathrm{Id}\to \widetilde{\mathrm{Nor}_*}\widetilde{\mathrm{Nor}^*}$$ is an equivalence. 	
	By \cite[Lemma 2.6]{MR3789830}, the $\infty$-category $\mathrm{Mor}_{\nscr_k}(\mathcal{DM}^{\mathrm{\'et}}(k))$ is compactly generated by the $f_\sharp\Q_Y(i)\otimes\nscr_k$ for $f:Y\to \Spec k$ smooth and $i\in \Z$. As $\widetilde{\mathrm{Nor}^*}$ of those compact generators are compact objects of $\mathcal{DN}(k)$, the right adjoint $\widetilde{\mathrm{Nor}_*}$ preserves colimits. Thus it suffices to check that for any compact object $C\otimes\nscr_k$ of the above form, the map 
	$$C\otimes\nscr_k\to \widetilde{\mathrm{Nor}_*}\widetilde{\mathrm{Nor}^*}(C\otimes\nscr_k)\simeq \widetilde{\mathrm{Nor}_*}\mathrm{Nor}^*C $$ is an equivalence. As the forgetful functor $\mathrm{Mor}_{\nscr_k}(\mathcal{DM}^{\mathrm{\'et}}(k))\to \mathcal{DM}^{\mathrm{\'et}}(k)$ is conservative it is even sufficient to prove that 
	$$C\otimes\nscr_k\to \mathrm{Nor}_*\mathrm{Nor}^*C$$
	is an equivalence in $\mathcal{DM}^{\mathrm{\'et}}(k)$. This is true because over a field, compact objects in $\mathcal{DM}^{\mathrm{\'et}}(k)$ are dualisable and for dualisable objects this map is an equivalence by \cite[Lemma 2.8]{MR3259031}.
	
	Now for a general $X$ we use that the functor $\widetilde{\mathrm{Nor}^*}$ commutes with the $6$ operations (this follows from the fact that $\nscr_X$ is a filtered colimit of dualisable objects thus the projection formulae of the form $f_*(M)\otimes \nscr \simeq f_*(M\otimes \nscr)$ are true by \cite[Lemma 2.8]{MR3259031}, and this implies commutation with all the operations as in the proof of \cite[Theorem 4.4.2.5 ]{MR3971240}). For compact $M,N\in \mathrm{Mod}_{\pi_X^*\nscr_k}(\mathcal{DM}^{\mathrm{\'et}}(X))$ we have:
	\begin{eqnarray*}
		\mathrm{map}_{\mathrm{Mod}_{\pi_X^*\nscr_k}(\mathcal{DM}^{\mathrm{\'et}}(X))}(M,N) & \simeq & \mathrm{map}_{\mathrm{Mod}_{\pi_X^*\nscr_k}(\mathcal{DM}^{\mathrm{\'et}}(X))}(\Q_X,\sHom(M,N)) \\
		& \simeq & \mathrm{map}_{\mathrm{Mod}_{\nscr_k}(\mathcal{DM}^{\mathrm{\'et}}(k))}(\Q_k,(\pi_X)_*\sHom(M,N)) \\
		& \simeq & \mathrm{map}_{\mathcal{DN}(k)}(\Q_k,\widetilde{\mathrm{Nor}^*}(\pi_X)_*\sHom(M,N)) \\
		& \simeq & \mathrm{map}_{\mathcal{DN}(k)}(\Q_k,(\pi_X)_*\sHom(\widetilde{\mathrm{Nor}^*}M,\widetilde{\mathrm{Nor}^*}N)) \\
		& \simeq &\mathrm{map}_{\mathcal{DN}(X)}(\widetilde{\mathrm{Nor}^*}M,\widetilde{\mathrm{Nor}^*}N)
	\end{eqnarray*}
	which finishes the proof of fully faithfulness. Now that we know that the unit 
	$$\mathrm{Id}\to \widetilde{\mathrm{Nor}_*}\widetilde{\mathrm{Nor}^*}$$ is an equivalence in $\mathrm{Mod}_{\pi_X^*\nscr_k}(\mathcal{DM}^{\mathrm{\'et}}(X))$, by plugging in $\pi_X^*\nscr_k$ we obtain 
	$$\pi_X^*\nscr_k\xrightarrow{\sim} \widetilde{\mathrm{Nor}_*}\widetilde{\mathrm{Nor}^*}(\pi_X^*\nscr_k)$$ and we have $$\widetilde{\mathrm{Nor}^*}(\pi_X^*\nscr_k)\simeq \pi_X^*\widetilde{\mathrm{Nor}^*}(\nscr_k)\simeq \Q_X .$$ Therefore the canonical map 
	$$\pi_X^*\nscr_k\to \widetilde{\mathrm{Nor}_*}\Q_X$$ is an equivalence, thus it stays an equivalence after applying the forgetful functor, leaving us with the equivalence $$\pi_X^*\nscr_k\xrightarrow{\sim} {\mathrm{Nor}_*}\Q_X$$ in $\mathcal{DM}^{\mathrm{\'et}}(X)$.

\end{proof}

From the previous lemma there is no more ambiguity for $\nscr_X=\pi_X^*\nscr_k\simeq \mathrm{Nor}_*\Q_X\in\mathcal{DM}^{\mathrm{\'et}}(X)$. 
\begin{prop}
	\label{C0fields}
	Let $k=\colim_i k_i$ be a filtered colimit of fields of characteristic zero. Then the natural functors
	\[ \colim_i\mathcal{DN}_c(\Spec k_i)\to \mathcal{DN}_c(Spec k)\] 
	and 	\[\colim_i\mathcal{DN}(\Spec k_i)\to \mathcal{DN}(Spec k)\] 
	are equivalences in $\catinfty$ and $\mathrm{Pr}^L$ respectively.

	The same is true for $\eta = \lim_i U_i$ with $\eta\in X$ the generic point of an integral finite type $k$-scheme and the $U_i$ are the nonempty open subsets of $X$: 
	\[ \colim_i\mathcal{DN}_{(c)}(U_i)\to \mathcal{DN}_{(c)}(\eta).\] 
\end{prop}

\begin{proof}
	As the functor $\Ind$, from small idempotent complete stable $\infty$-categories to compactly generated presentable stable categories is an equivalence and that the inclusion functor $\mathrm{Pr}^L_\omega\to\mathrm{Pr}^L$ preserves colimits, the second statement follows from the first.
	By \cite[Proposition 6.8]{ivorraFourOperationsPerverse2022} (or \cite[Corollary 6.10]{ivorraFourOperationsPerverse2022} for the generic point case), the diagram $\mcal:I^{\rhd}\to \mathrm{AbCAt}$ from $I$ to the $(2,1)$-category of small abelian categories with exact functors sending $i\in I$ to $\Mp(\Spec k_i)$ and 
	the cone point $v$ to $\Mp(\Spec k)$ is a colimit diagram. By \cite[Theorem 7.4.9]{bunkeControlledObjectsLeftexact2019} the functor sending an Abelian category to the $\infty$-category of bounded complexes up to homotopy is a left adjoint, hence commutes with colimits, so that we obtain a colimit diagram 
	$$\kcal^b\circ\mcal: I^{\rhd}\to\catinfty^\mathrm{st}.$$
	Now as $\dcal^b(\acal)$ is a localisation of $\kcal^b(\acal)$ for any abelian category $\acal$, and as the transitions of $\kcal^b\circ\mcal$ preserve quasi-isomorphisms, $\dcal^b\circ \mcal$ is a colimit diagram.
\end{proof}
\begin{cor}
	Let $f:\Spec k\to\Spec \Q$ be a map of fields. Then the natural map $f^*\nscr_\Q\to\nscr_k$ is an equivalence. 
	Consequently, for any scheme $X$ of finite type over $k$ if we denote by $p_X:X\to\Spec \Q$ the unique map (which has to factor through $f$), then $\nscr_X\simeq p_X^*\nscr_\Q$.
\end{cor}
\begin{proof}
	The second claim follows from the first. Write $k=\colim_i k_i$ with $k_i/\Q$ of finite type. Then the same proof as \cite[Lemma 3.5.7]{MR4466640}, using continuity of both $\mathcal{DM}^{\mathrm{\'et}}$ and $\mathcal{DN}$ for fields (\Cref{C0fields}) and that these categories are compactly generated gives that the natural map 
	$$\colim_i f_i^*\nscr_\Q \to \nscr_k$$ is an equivalence so that this reduces to the case where $k/\Q$ is of finite type. Then $k$ is the generic point of an integral $\Q$-scheme, and again the same proof as \cite[Lemma 3.5.7]{MR4466640}, using continuity for the generic point of perverse Nori motives, reduces to show that $\pi_U^*\nscr_\Q \to \nscr_U$ is an equivalence for $U$ of finite type over $\Q$, which is \Cref{fullFaithNor}.
	
\end{proof}
Therefore now there is even less ambiguity for $\nscr_X = (X\to \Spec \Q)^*\nscr_\Q$.
\begin{prop}
	\label{continu}
	Let $X=\lim_i X_i$ be a limit of $\Q$-schemes with affine transitions. The natural
	morphism \begin{equation}
		\colim_i \mathrm{Mod}_\nscr(\mathcal{DM}^{\mathrm{\'et}}(X_i))\to \mathrm{Mod}_\nscr(\mathcal{DM}^{\mathrm{\'et}}(X))
	\end{equation}
	is an equivalence where the colimit is computed in $\mathrm{Pr}^{L}$.
\end{prop}
\begin{proof}
	By \cite[Lemma 5.1 (ii)]{MR4319065}, the functor
	\begin{equation}
		\colim_i \mathcal{DM}^{\mathrm{\'et}}(X_i,\Q) \to \mathcal{DM}^{\mathrm{\'et}}(X,\Q)
	\end{equation}
	is an equivalence. Therefore, the assignment $i\mapsto (\mathcal{DM}^{\mathrm{\'et}}(X_i,\Q),\nscr_{\mid X_i})$ is a colimit
	diagram in the category $\mathrm{Pr}^{\mathrm{Alg}}$ of \cite[Notation 4.8.5.10]{lurieHigherAlgebra2022}. By
	\cite[Theorem 4.8.5.11]{lurieHigherAlgebra2022}, the functor $(\ccal,A)$ that sends a presentable symmetric monoidal category $\ccal$ together with an algebra object $A\in\ccal$ to the category $\mathrm{Mod}_A(\ccal)$ has a right adjoint, hence preserves colimits.
	This is why the assignment $i\mapsto \mathrm{Mod}_{\nscr_{X_i}}(\mathcal{DM}^{\mathrm{\'et}}(X_i,\Q))$ is a colimit diagram. We found a
	 precise statement giving the symmetric monoidal structure in the paper by Ayoub, Gallauer and Vezzani. The precise statement is \cite[Lemma 3.5.6]{MR4466640}.
\end{proof}

\begin{prop}
	\label{Norimodules}
	The functor \begin{equation}\widetilde{\mathrm{Nor}^*} :\mathrm{Mod}_\nscr(\mathcal{DM}^{\mathrm{\'et}}(X)) \to\mathcal{DN}(X)\end{equation}
	is an equivalence for every finite type $k$-scheme $X$.
\end{prop}
\begin{proof}
	The functor is fully faithful by \Cref{fullFaithNor}.
	We first prove that it is essential surjective for $X=\Spec K$ the spectrum of a field, not necessarily finite over $k$.
	 As $\mathcal{DN}(\Spec K)$ is compactly generated it suffices to show that $\widetilde{\mathrm{Nor}^*}$ reaches all compact objects $M\in\dcal^b(\Mp(K))$. As the functor is already fully faithful, by dévissage it suffices 
	to check that any object of the heart $M\in\Mp(K)$ is in the image of $\widetilde{\mathrm{Nor}^*}$.  Any such object has a weight filtration whose graded pieces are direct factors 
	of $\HH^{n+2j}(\pi_X)_*\Q_X(i)$ with $\pi_X:X\to\Spec K$ a smooth projective morphism by Arapura's \cite[Theorem 10.2.5]{MR3618276}. Note that here we used that if $K$ is a finite type extension of $k$, then the category $\Mp(\Spec K)$ obtained by colimit as motives over the generic point of a variety over $k$ coincides with Nori motives over $K$. This holds with a beautiful proof in \cite[Proposition 1.15]{terenziTensorStructurePerverse2024} or can be proven by using that the  $\ell$-adic realisation is absolute, \emph{i.e.} does not depend on the base field. Therefore 
	by dévissage again it suffices to check that any such $\HH^{n+2j}(\pi_X)_*\Q_X(i)$ is in the image. But as $(\pi_X)_*\Q_X(i)$ is pure of weight $-2i$, it is the direct sum of its 
	$\HH^n[-n]$, thus is suffices to show that each $(\pi_X)_*\Q_X(i)$ is in the image, which is obvious by compatibility with the $6$ operations.

	Now for a finite type $k$-scheme, the proposition is proven by Noetherian induction.
	Assume now that $X$ is such that the statement is true for any proper closed subset of $X$. We take an object $M\in\mathcal{DN}(X)$ that we may assume to be compact. Let $\eta$ be a generic point of $X$. By the case of a field,
	there exists $N\in\mathrm{Mod}_\nscr(\mathcal{DM}^{\mathrm{\'et}}(\eta))$, compact, such that $\widetilde{\mathrm{Nor}^*}(N)=M_\eta$. By \Cref{continu} there exists an open subset $U$ of $X$ and an object $N'\in\mathrm{Mod}_\nscr(\mathcal{DM}^{\mathrm{\'et}}(U))$, compact, such that $(N')_\eta = N$.
	The functor $\colim_{\eta\in U}\mathcal{DN}(U)\to\mathcal{DN}(\eta)$ is fully faithful by \cref{C0fields}. Thus the isomorphism 
	$\widetilde{\mathrm{Nor}^*}(N)\to M_\eta$ lifts to a smaller open subset $V$ of $U$ : we have $\widetilde{\mathrm{Nor}^*}((N')_{\mid V})\simeq M_{\mid V}$. Denote by $Z$ the reduced closed complement 
	of $V$ in $X$. By induction hypothesis, there exists a $N''\in\mathrm{Mod}_\nscr(\mathcal{DM}^{\mathrm{\'et}}(Z))$ such that $\widetilde{\mathrm{Nor}^*}(N'')=M_{\mid Z}$. 
	The cofiber sequence $\widetilde{\mathrm{Nor}^*}(j_!(N')_{\mid V})\to M \to \widetilde{\mathrm{Nor}^*}(i_*N'')$ with $j:U\to X$ and $i:Z\to X$ the immersion then ensures that 
	$M$ is the image of $\widetilde{\mathrm{Nor}^*}$, which finishes the proof.

\end{proof}
\begin{rem}
	The same arguments would
	prove that the $\infty$-category
	$$\mathcal{DH}^{\mathrm{geo}}(X)\subset\Ind\dcal^b(\mathrm{MHM}(X))$$ of geometric origin objects in the
	 indization 
	of the derived category of mixed Hodge modules is 
	also the category of modules over some algebra $\mathscr{H}_X$ in $\mathcal{DM}^{\mathrm{\'et}}(X)$. Note that the lack of continuity of Mixed Hodge modules would prevent things like $p_X^*\mathscr{H}_{\Spec \Q}\simeq \mathscr{H}_{X}$ of happening so one would have to be careful of which algebra is used and an additional argument over the generic point will be needed, similarly to the case of the Betti realisation.
 Similar ideas are developed in \cite{drewMotivicHodgeModules2018} and the fact that one would obtain $(X\to \Spec k)^*\hscr_{\Spec k}\simeq \hscr_X$ identifies Drew's category as the full subcategory $\mathcal{DH}^\mathrm{geo}(X)$ of geometric origin mixed Hodge modules. His considerations with enriched categories of motives also have an application here: his results, together with the same proof as above, would show that his category $\mathbf{DH}$ of motivic enriched Hodge modules is the smallest full subcategory of $\Ind\dcal^b(\mathrm{MHM}(-))$ which is stable under truncation and filtered colimts, and that contains the $f_*g^*H$ for $f:Y\to X$ proper, $g:Y\to\Spec \C$ the structural morphism and $H\in\dcal(\Ind\mathrm{MHS}^p)$. See \cite{SwannMHM} for a detailed discussion.
\end{rem}
\begin{cor}
	\label{C0ctNori}
	The $\infty$-functor $X\mapsto\dcal^b(\Mp(X))$ has an unique extension $\mathcal{DN}_c(X)$ to quasi-compact 
	and quasi-separated schemes of characteristic zero such that 
	for all limit $X=\lim_i X_i$ of such schemes with affines transitions the natural functor 
	\[ \colim_i\mathcal{DN}_c(X_i)\to\mathcal{DN}_c(X)\] 
	is an equivalence of $\infty$-categories, where we compute the colimit in $\mathrm{Pr}^{L}$.
\end{cor}
\begin{proof}
	Let $\ccal$ be category of schemes that are of finite type over a field of characteristic zero. Let $\mathrm{Sch}_\Q$ be
	 the category of quasi-compact and quasi-separated (qcqs) $\Q$-schemes, and $\iota : \ccal\to\mathrm{Sch}_\Q$ the inclusion.
	Nori motives, together with the pullback functoriality, define a functor $\ccal^\op \to\mathrm{CAlg}(\catinfty)$ which sends
	 limits of schemes with affine transitions to colimits of $\infty$-categories by \Cref{continu} together with the fact that the colimit is taken in 
	 $\mathrm{Pr}^L$ hence restricts to compact objects.
	We can denote by $\mathcal{DN}_c$ its left Kan extension along $\iota$.
\end{proof}

\begin{cor}
	\label{generation}
Let $X$ be a finite type $k$-scheme. The $\infty$-category $\mathcal{DN}(X)$ is compactly generated by the $f_\sharp\Q_Y(i)$ for $f:Y\to X$ smooth and $i\in \Z$.
\end{cor}
\begin{proof}
	Indeed, this the case for $\mathcal{DM}^{\mathrm{\'et}}(X)$, thus for $\mathrm{Mod}_{\nscr_X}(\mathcal{DM}^{\mathrm{\'et}}(X))\simeq \mathcal{DN}(X)$ by \cite[Lemma 2.6]{MR3789830}.
\end{proof}
 We finish with the proof of the following natural corollary, whose proof was surprisingly harder than expected for a result that the author took for granted for a long time.

\begin{cor}
	\label{Huniv}
	Let $X$ be a quasi-projective $k$-scheme. 
	The two functors $$\mathcal{DM}^{\mathrm{\'et}}_c(X)\to \Mp(X)$$ given by $\HHp^0\circ\mathrm{Nor}^*$ and the universal functor $\HH_\mathrm{univ}$ are canonically equivalent.
\end{cor}
\begin{proof}
	First, note that the functor $\HHp^0\circ\mathrm{Nor}^*$ gives a factorisation of the $\HHp^0$ of the Betti realisation of étale motives $$\mathcal{DM}^{\mathrm{\'et}}_c(X)\to \mathrm{Perv}(X).$$ By the universal property of perverse Nori motives, this induces a canonical functor $$F:\Mp(X)\to \Mp(X)$$ such that $R_B\circ F \simeq R_B$ and $\HHp^0\circ\mathrm{Nor}^*\simeq F\circ \HH_\mathrm{univ}$. Now the compatibility of $\mathrm{Nor}^*$ with the operations ensures that for any perverse t-exact operation $h$ the functor $F$ commutes with the operation $h$. Hence $F$ commutes with $i_*$ for $i$ a closed immersion and with $f^*[d]$ for $f$ a smooth map of relative dimension $d$. We will still denote by $F$ the functor induced by $F$ on $\dcal^b(\Mp(X))$. As $F$ commutes with $i_*$, there exists an exchange morphism $i^*\circ F\to F \circ i^*$. Applying the Betti (or $\ell$-adic in the case the field $k$ is too big) realisation where this exchange morphism is an equivalence, we see that $F$ commutes with $i^*$. The same argument works to show that $F$ commutes with $f_\sharp$ for smooth $f$: because $F$ commutes with $f^*$ there is a natural transformation $f_\sharp \circ F \to F\circ f_\sharp$ which is an equivalence after realisation, hence $F$ also commutes with the $f_\sharp$ functors. It also commutes with external tensor product because is it perverse t-exact, hence with the internal tensor product which is obtained from the external one by pulling back along the diagonal.
	
	The functor $F$ induces a compact and colimit preserving functor 
	$$F:\Ind\dcal^b(\Mp(X))\to\Ind\dcal^b(\Mp(X))$$ which commutes with all pullbacks, pushforwards and left adjoint to pushforward by smooth maps. Denote by $G$ its right adjoint. We claim that the natural transformation $$\mathrm{Id}\to GF$$ is an equivalence. Indeed, by \Cref{generation} we have that $\Ind\dcal^b(\Mp(X))$ is generated by colimits by the $f_\sharp\Q_Y(i)$ for $f:Y\to X$ a smooth map and $i\in \Z$. This means that it suffices to check that for $M\in\Ind\dcal^b(\Mp(X))$, the map 
	$$\mathrm{map}(f_\sharp\Q_Y(i),M)\to \mathrm{map}(f_\sharp\Q_Y(i),GF(M))$$ is an equivalence. As $G$ and $F$ commute with filtered colimits and $f_\sharp\Q_Y(i)$ is compact, we can assume that $M$ is as well compact. Moreover, compact objects are constructed by extensions, finite limits, finite colimits and direct factors from the same $f_\sharp\Q_Y(i)$, thus by dévissage we can assume that $M = g_\sharp\Q_Z(j)$ for some smooth map $g:Z\to X$ and $j\in\Z$. Using that $G$ is a right adjoint and that $F$ commutes with the operations, the above map is equivalent to the map 
	$$\mathrm{map}(f_\sharp\Q_Y(i),g_\sharp\Q_Z(j))\to \mathrm{map}(f_\sharp F(\Q_Y)(i),g_\sharp F(\Q_Z)(j))$$ and as $F(\Q) = \Q$, this finishes the proof.

	Now that $F$ is fully faithful, it is automatically essentially surjective because we know it reaches the generators $f_\sharp\Q(i)$ of $\Ind\dcal^b(\Mp(X))$. Now, as $F$ commutes with the operations, we have by universality of $\mathcal{DM}^{\mathrm{\'et}}$ that $F\circ \mathrm{Nor}^*\simeq \mathrm{Nor}^*$. Indeed because we proved that $F$ commutes with the tensor product and because it is t-exact for the constructible t-structure we can see $F$ as the derived functor of a symmetric monoidal functor on the constructible heart, apply \Cref{symMonoCatDeri} and deduce that $F$ is a morphism of coefficient systems. This proves that the compositions of the morphisms $\mathrm{Nor}^*$ and $F\circ\mathrm{Nor}^*$ with the morphism $\mathcal{SH}\to\mathcal{DM}^\mathrm{\'et}$ agree by universality of $\mathcal{SH}$, thus also their factorisation through the localisation $\mathcal{DM}^{\'et}$ of $\mathcal{SH}$. In particular over $X$, $F\circ \HHp^0\circ \mathrm{Nor}^* \simeq \HHp^0\circ\mathrm{Nor}^*$ hence $F\circ F\circ  \HH_\mathrm{univ} \simeq  \HHp^0\circ \mathrm{Nor}^*$ and the universal property of $\Mp(X)$ tells us that $F\circ F\simeq F$, so that $F\simeq \mathrm{Id}$.
\end{proof}

\subsection{Relations with the t-structure conjecture.}
\label{tstructCOnj}
Let $k$ be a field of characteristic zero.
Recall the following conjecture (\cite[Section 21.1.7]{MR2115000} and \cite{MR2953406}):
\begin{conj}
	\label{tstrConj}
	There exists a non degenerate t-structure on $\mathcal{DM}^{\mathrm{\'et}}_c(k)$ compatible with the tensor product and
	such that the $\ell$-adic realisation functor is t-exact.
\end{conj}
In \cite{MR2953406} Beilinson proves that if the characteristic of $k$ is zero the conjecture implies that the realisation functors of $\mathcal{DM}^{\mathrm{\'et}}_c(k)$ are conservative. In \cite[Theorem 3.1.4]{MR3347995}, Bondarko proves that if such a t-structure exists for all fields of characteristic $0$, then there exists a perverse t-structure
on $\mathcal{DM}^{\mathrm{\'et}}_c(X)$ for each finite type $k$-scheme $X$. The joint conservativity of the family of $x^*$ for $x\in X$ and the conservativity of the the realisation functor over fields implies that in that case, the $\ell$-adic realisation over $X$ is conservative and t-exact, when its target is endowed with the perverse t-structure.

\begin{thm}
	\label{IfConjThenSame}
	Assume \Cref{tstrConj} for all fields of characteristic $0$. Let $X$ be a finite type $k$-scheme.
	Then the heart of the perverse t-structure of $\mathcal{DM}^{\mathrm{\'et}}_c(X)$ is canonically equivalent to the category $\Mp(X)$ of perverse Nori motives. Moreover, the functor
	$\mathrm{Nor}^* : \mathcal{DM}^{\mathrm{\'et}}_c(X)\to\dcal^b(\Mp(X))$ is an equivalence of stable $\infty$-categories. This implies that $\mathcal{DM}^{\mathrm{\'et}}_c(X)$ is the derived category of the abelian category of both its perverse and constructible heart.
\end{thm}

From now on, we assume that \Cref{tstrConj} is true. We begin by two lemmas, in which we denote by $\mathrm{Var}_k$ the category of quasi-projective $k$-varieties.

\begin{lem}
	\label{lem1tstrconj}
	Let $X$ be quasi-projective scheme. Denote by $\mcal(X)$ the heart of the perverse t-structure on $\mathcal{DM}^{\mathrm{\'et}}_c(X)$.
	There is a faithful exact functor $F_X:\Mp(X)\to \mcal(X)$ that commutes with the $\ell$-adic realisation. Denote by $$\gamma_X : \dcal^b(\Mp(X))\to\mathcal{DM}^{\mathrm{\'et}}_c(X)$$ the canonical functor induced by the universal property of the bounded derived category applied to the functor $$\Mp(X)\xrightarrow{F_X}\mcal(X)\to\mathcal{DM}_c^\mathrm{\'et}(X).$$
	The composition $\mathrm{Nor}^*_X\circ \gamma_X$ is equivalent to the identity functor $\mathrm{Id}_{\dcal^b(\Mp(X))}$. Moreover the $\mathrm{ho}(\gamma_X)$ fit in a morphism $\gamma$ of functors $\mathrm{Var}^\op_k\to\mathrm{SymMono}_1$ with values in symmetric monoidal categories, using the $(-)^*$-functoriality.
\end{lem}
\begin{proof}
	By assumption the restriction to the heart of the $\ell$-adic realisation functor $$\rho_\ell\colon\mcal(X)\to\mathrm{Perv}(X,\Q_\ell)$$ is faithful exact and the composition $\rho_\ell\circ \HHp^0$ with the perverse $\HHp^0$ on $\mathcal{DM}^{\mathrm{\'et}}_c(X)$ is isomorphic to the perverse $\HH^0$ of the $\ell$-adic realisation. By the universal property of perverse Nori motives \cite[Proposition 6.10]{ivorraFourOperationsPerverse2022} this provides a faithful exact functor $F_X\colon \Mp(X)\to\mcal(X)$ such the composition $\rho_\ell\circ F_X$ is isomorphic to the $\ell$-adic realisation $R_\ell$ of perverse Nori motives and the composition $F_X\circ\HH_\mathrm{univ}$ with the universal functor $\HH_\mathrm{univ}\colon\mathcal{DM}^{\mathrm{\'et}}_c(X)\to \Mp(X)$ is the perverse $\HHp^0$ on $\mathcal{DM}^{\mathrm{\'et}}_c(X)$: we have a commutative (up to natural isomorphism) diagram
	\[\begin{tikzcd}
		{\mathcal{DM}^{\mathrm{\'et}}_c(X)} & {\Mp(X)} \\
		& {\mcal(X)} & {\mathrm{Perv}(X,\Q_\ell)}
		\arrow["{\HH_{\mathrm{univ}}}", from=1-1, to=1-2]
		\arrow["{\HHp^0}"', from=1-1, to=2-2]
		\arrow["{F_X}", from=1-2, to=2-2]
		\arrow["{\mathcal{R}_\ell}", from=1-2, to=2-3]
		\arrow["{\rho_\ell}"', from=2-2, to=2-3]
	\end{tikzcd}.\] 
	Now consider the endofunctor \[\Mp(X)\xrightarrow{F_X}\mcal(X)\xrightarrow{\mathrm{Nor}^*} \Mp(X)\] of $\Mp(X)$. Because $\mathrm{Nor}^*$ is perverse t-exact (this can be checked after $\ell$-adic realisation), if we compose $F_X$ with the universal $\HH_\mathrm{univ}$ we obtain \[\mathrm{Nor}^*\circ F_X\circ \HH_\mathrm{univ}\simeq \mathrm{Nor}^*\circ \HHp^0 \simeq \HHp^0\circ\mathrm{Nor}^*\] as functors $\mathcal{DM}^{\mathrm{\'et}}_c(X)\to\Mp(X)$. But by \Cref{Huniv} the functor $\HHp^0\circ\mathrm{Nor}^*$ can canonically be identified with $\HH_\mathrm{univ}$. Thus we have a commutative diagram \[\begin{tikzcd}
		{\mathcal{DM}^{\mathrm{\'et}}_c(X)} & {\Mp(X)} \\
		&& {\mcal(X)} & {\mathrm{Perv}(X,\Q_\ell)} \\
		& {\Mp(X)}
		\arrow["{\HH_{\mathrm{univ}}}", from=1-1, to=1-2]
		\arrow["{\HH_{\mathrm{univ}}}"', from=1-1, to=3-2]
		\arrow["{F_X}"{description}, from=1-2, to=2-3]
		\arrow["{R_\ell}", from=1-2, to=2-4]
		\arrow["{\mathrm{Nor}^*\circ F_X}"{description}, from=1-2, to=3-2]
		\arrow["{\rho_\ell}"{description}, from=2-3, to=2-4]
		\arrow["{\mathrm{Nor}^*}"{description}, from=2-3, to=3-2]
		\arrow["{R_\ell}"', from=3-2, to=2-4]
	\end{tikzcd}.\]
	By the universal property of $\Mp(X)$, the only vertical faithful exact functor fitting in the above diagram and making the two outer triangles commute is (up to natural isomorphism) the identity functor: this implies that there exists a natural equivalence $\mathrm{Nor}^*\circ F_X\simeq\mathrm{Id}_{\Mp(X)}$.
	
	Now by the universal property of the bounded derived category we have a functor $\dcal^b(\mcal(X))\to\mathcal{DM}^{\mathrm{\'et}}_c(X)$, thus by composition with the (trivially) derived functor of $F_X$ there is a functor \[\gamma_X\colon\dcal^b(\Mp(X))\to\mathcal{DM}^{\mathrm{\'et}}_c(X).\] 
	By definition, the restriction to $\Mp(X)$ of $\gamma_X$ is exactly the functor $F_X$ composed with the inclusion $\mcal(X)\to\mathcal{DM}^{\mathrm{\'et}}_c(X)$. In particular, the composition $\mathrm{Nor}^*_X\circ\gamma_X$ induces the functor $\mathrm{Nor}^*_X\circ F_X$ between the hearts. By the universal property of $\dcal^b(\Mp(X))$ and the above paragraph, this implies that there exists a natural equivalence $\mathrm{Nor}^*_X\circ\gamma_X\simeq \mathrm{Id}_{\dcal(\Mp(X))}$.

	Using \cite[Section 2.2]{ivorraFourOperationsPerverse2022} and the t-exactness of the $\ell$-adic realisation of étale motives we see that there exist natural isomorphisms $i_*\circ \gamma\simeq \gamma\circ i_*$ for $i$ a closed immersion, $f^*[d]\circ\gamma\simeq \gamma\circ f^*[d]$ for $f$ a smooth morphism of relative dimension $d$ and $\gamma(-)\boxtimes\gamma(-)\simeq \gamma(-\boxtimes -)$ with $\boxtimes$ the external tensor product. Indeed for the pullback by  a smooth map $f\colon Y\to X$ we have a commutative diagram \[\begin{tikzcd}
		{\mathcal{DM}^{\mathrm{\'et}}_c(X)} & {\Mp(X)} & {\mcal(X)} \\
		{\mathcal{DM}^{\mathrm{\'et}}_c(Y)} & {\Mp(Y)} & {\mcal(Y)}
		\arrow["{\HH_\mathrm{univ}}", from=1-1, to=1-2]
		\arrow["{f^*[d]}"', from=1-1, to=2-1]
		\arrow["{F_X}", from=1-2, to=1-3]
		\arrow["{f^*[d]}"', from=1-2, to=2-2]
		\arrow["{f^*[d]}"', from=1-3, to=2-3]
		\arrow["{\HH_\mathrm{univ}}",from=2-1, to=2-2]
		\arrow["{F_Y}"', from=2-2, to=2-3]
	\end{tikzcd}\] obtained by the universal property of $\Mp(X)$, which induces a commutative square \[\begin{tikzcd}
		{\Mp(X)} & {\mathcal{DM}^{\mathrm{\'et}}_c(X)} \\
		{\Mp(Y)} & {\mathcal{DM}^{\mathrm{\'et}}_c(Y)}
		\arrow[from=1-1, to=1-2]
		\arrow["{f^*[d]}"', from=1-1, to=2-1]
		\arrow["{f^*[d]}", from=1-2, to=2-2]
		\arrow[from=2-1, to=2-2]
	\end{tikzcd}\] that gives the wanted natural isomorphism by applying the universal property of the bounded derived category. The cases of $i_*$ and $\boxtimes$ are similar. By adjunction there is a natural transformation $i^*\circ \gamma\Rightarrow \gamma\circ i^*$, which is an isomorphism because $\mathrm{Nor}^*$ is conservative, commutes with $i^*$ and is a retraction of $\gamma$. Using the fact that any morphism of quasi-projective varieties can be written as a composition of a closed immersion and a smooth map, together with the formula for the tensor product using the external tensor product and pullbacks, we see that $\mathrm{ho}(\gamma)$ commutes with pullbacks and tensor products, giving the last claim of the lemma. 
\end{proof}

By the lemma above we know that we have a natural transformation $$\gamma\colon \Dd^b(\Mp(-))\to \mathrm{ho}(\mathcal{DM}^\mathrm{\'et}_c(-))$$ of contravariant functors on quasi-projective with values in symmetric monoidal $1$-categories. We need an $\infty$-categorical enhancement of this functor. This is the next lemma. Note that as for perverse Nori motives, the existence of a perverse t-structure on étale motives $\mathcal{DM}^\mathrm{\'et}_c(X)$ implies that there is also a constructible t-structure, whose heart we denote by $\mathcal{CM}(X)$. The functors $\gamma_X$ and $\mathrm{Nor}^*_X$ are t-exact for the constructible t-structure, as this can be checked after $\ell$-adic realisation.

\begin{lem}
	\label{lem2stdconj}
Let $X$ be a quasi-projective scheme. The functor $\dcal^b(\Mp(X))\to\mathcal{DM}^{\mathrm{\'et}}_c(X)$ induces a  symmetric monoidal functor $\Mc(X)\to \mathcal{CM}(X)$.
This gives a natural transformation \[\delta\colon\colon \dcal^b(\Mc(-))\Rightarrow \mathcal{DM}^{\mathrm{\'et}}_c(-)\] in $\mathrm{Fun}(\mathrm{Var}_k^\op,\mathrm{CAlg}(\catinfty))$, where we use the $(-)^*$-functoriality. Moreover $\delta$ is a morphism of coefficient systems and satisfies $\mathrm{Nor}^*\circ\delta\simeq\mathrm{Id}_{\dcal^b(\Mc(-))}$.  In fact, we have $\mathrm{ho}(\delta)\simeq \gamma$.
\end{lem}
\begin{proof}
	Because $\dcal^b(\Mp(X))\xrightarrow{\gamma_X}\mathcal{DM}^{\mathrm{\'et}}_c(X)$ and $\mathcal{DM}^{\mathrm{\'et}}_c(X)\xrightarrow{\mathrm{Nor}^*}\dcal^b(\Mp(X))$ are t-exact for the constructible t-structure, we have faithful exact functors 
	$\Mc(X)\to\mathcal{CM}(X)$ and $\mathcal{CM}(X)\to \Mc(X)$ whose composition is the identity of $\Mc(X)$ by \Cref{lem1tstrconj}.

	Now by \Cref{SymMonoReal} there is a natural transformation
	\[\dcal^b(\mathcal{CM}(-))\to\mathcal{DM}_c^\mathrm{\'et}\] of functors $\Sch_k^\op\to\mathrm{CAlg}(\catinfty)$. By the last part of \Cref{lem1tstrconj} we have a functor $\Delta^1\times\mathrm{Var}_k^\op\to\mathrm{SymMono}_1$ sending, for $f\colon Y\to X$ a map of quasi-projective varieties, the map $(0<1,f)$ to $\Mc(Y)\xrightarrow{f^*}\Mc(X)\xrightarrow{\gamma_X}\mathcal{CM}(X)$, which is also isomorphic to $\Mc(Y)\xrightarrow{\gamma_Y}\mathcal{CM}(Y)\xrightarrow{f^*}\mathcal{CM}(X)$. Thus by \Cref{SymMonoReal}
	there is a natural transformation \[\dcal^b(\Mc(-))\to\dcal^b(\mathcal{CM}(-))\] of functors 
	$\mathrm{Var}_k^\op\to \mathrm{CAlg}(\catinfty)$. Composing the two we obtain a natural transformation
	$$\delta\colon\dcal^b(\Mc(-))\Rightarrow \mathcal{DM}^{\mathrm{\'et}}_c$$
	of functors with values in symmetric monoidal stable $\infty$-categories. The restriction to $\Mc(X)$ of $\delta_X$ is, by definition, the restriction to $\Mc(X)$ of $\gamma_X$. Thus $\delta_X\simeq \gamma_X$ by \Cref{RestrToHeart}.

	To prove that $\delta$ is a morphism of coefficient systems, it suffices to prove that if $f\colon Y\to X$ is a smooth morphism of varieties, the natural transformation $f_\sharp\circ\delta_Y\to\delta_X\circ f_\sharp$ is an equivalence. For this, we may apply the conservative functor \[\rho_\ell\colon\mathcal{DM}^{\mathrm{\'et}}_c(X)\to\D^b_c(X_{\mathrm{\'et}},\Q_\ell)\] 
	so that the exchange morphism above becomes \[f_\sharp\circ \rho_\ell\circ \delta_Y\to \rho_\ell\circ\delta_X\circ f_\sharp.\]
	Now remark that for $S$ a variety, the composition $\mathrm{Nor}^*_S\circ\delta_S$ is the identity functor and the functor $\mathrm{Nor}^*$ commutes with the $\ell$-adic realisation thus we have an isomorphism \[\rho_\ell\circ\delta_S\simeq R_\ell\circ\mathrm{Nor}^*\circ\delta_S\simeq R_\ell.\]
	As the functor $R_\ell$ commutes with $f_\sharp$, the proof is finished.
\end{proof}

\begin{proof}(of \Cref{IfConjThenSame})
	Because both $\mathcal{DM}^\mathrm{\'et}_c$ and $\mathcal{D}^b(\Mp(-))$ are Zariski hypersheaves and $\mathrm{Nor}^*$ is a morphism of sheaves, we may assume that $X$ is quasi-projective.

	Let $\delta_X$ be the functor of \Cref{lem2stdconj}. We know that $\mathrm{Nor}^*_X\circ\delta_X\simeq\mathrm{Id}_{\dcal^b(\Mp(X))}.$ We can now show that $\delta_X\circ\mathrm{Nor}^*_X\simeq\mathrm{Id}_{\mathcal{DM}^{\mathrm{\'et}}_c(X)}$. By \Cref{lem2stdconj} and \Cref{realNori} the functor $\delta_X\circ\mathrm{Nor}^*_X$ is a morphism of coefficient systems. In particular, its composition with the canonical symmetric monoidal functor 
	\[M_X\colon \mathrm{Sm}_X\to\mathcal{DM}^{\mathrm{\'et}}_c(X)\] is isomorphic to the canonical symmetric monoidal functor $M_X$. Using the universal property of $\mathcal{DM}^{\mathrm{\'et}}_c(X)$ proved by Robalo in \cite[Corollary 2.29]{MR3281141} (or rather, its étale and $\Q$-linear counterpart which can be deduced from it) this implies:
$$		\delta_X\circ\mathrm{Nor}^*_X\simeq\mathrm{Id}_{\mathcal{DM}^{\mathrm{\'et}}_c(X)}.
$$	This finishes the proof as now we have that $\mathrm{Nor}^*_X$ and $\delta_X$ are inverses of each other.
	
\end{proof}

\appendix

\section{Coefficient systems.}
\label{CoeffSysAppe}
Recall the following definition of Drew and Gallauer in \cite{MR4560376} :
\begin{defi}
	\label{defCoSys}
A functor $C:\Sch^\op_k\to\mathrm{CAlg}(\catinfty^\mathrm{st})$ taking values in symmetric monoidal stable $\infty$-categories and exact symmetric monoidal functors is called a \emph{coefficient system} if it satisfies the following properties.

\begin{enumerate}
\item \textbf{(Pushforwards)} For every $f:Y\to X$ in $\Sch_k$, the
pullback functor $f^*$ admits a right adjoint $f_*:C(Y)\to C(X)$.
\item \textbf{(Internal homs)} For every $X\in\Sch_k$, the symmetric monoidal structure on $C(X)$ is closed.

\item For each smooth morphism $p:Y\to X\in\Sch_k$, the functor $p^*:C(X)\to C(Y)$ admits a left adjoint $p_\sharp$, and:
\begin{enumerate}
\item \textbf{(Smooth base change)} For each cartesian square
\[
\begin{tikzcd}
Y'
\ar[r, "p'" above]
\ar[d, "f'" left]
&
X'
\ar[d, "f" right]
\\
Y
\ar[r, "p" above]
&
X
\end{tikzcd}
\]
in $\Sch_k$, the exchange transformation $p'_\sharp (f')^*\to f^*p_\sharp$ is an equivalence.
\item \textbf{(Smooth projection formula)} The exchange transformation
\[
p_\sharp(p^*(-)\otimes -)\to -\otimes p_\sharp(-)
\]
is an equivalence of functors~$C(X)\times C(Y)\to C(Y)$.
\end{enumerate}
\item \textbf{(Localisation)} $C(\emptyset)\simeq 0$ and for each closed immersion $Z\to X$ in $\Sch_k$ with complementary open immersion $j:U\to X$, the square
\[
\begin{tikzcd}
C(Z)
\ar[r, "{i_*}"]
\ar[d]
&
C(X)
\ar[d, "j^*"]
\\
0
\ar[r]
&
C(U)
\end{tikzcd}
\]
is cartesian in $\catinfty^\mathrm{st}$.
\item  \textbf{($\A^1$-homotopy invariance)} For each $X\in\Sch_k$, if $\pi_{\A^1}:\A^1_X\to X$ denotes the canonical projection then
 the functor $\pi_{\A^1}^*:C(X)\to C(\A^1_X)$ is fully faithful.
\item \textbf{($\mathrm{T}$-stability)} The composite $\pi_{\A^1,\sharp} s_*:C(X)\to C(X)$ is an equivalence.

\end{enumerate}
A morphism of coefficient systems is a natural transformation $\phi:C\to C'$ such that for each smooth morphism $p:Y\to X$ in $\Sch_k$, the exchange transformation
\[
p_\sharp\phi_Y\to \phi_X p_\sharp
\]
is an equivalence.
This defines a sub $\infty$-category $\mathrm{CoSy}_k\subset\mathrm{Fun}(\Sch^\op_k,\mathrm{CAlg}(\catinfty^\mathrm{st}))$. We say that a coefficient system $C$ is \emph{cocomplete} if it takes values in symmetric monoidal cocomplete $\infty$-categories and colimit preserving functors. We denote by $\mathrm{CoSy}_k^c$ their $\infty$-category.

\end{defi}

Recall the main theorem of \cite{MR4560376}:
\begin{thm}[Drew-Gallauer]
	\label{univSh}
The stable $\A^1$-homotopy $\infty$-category defines an object $\mathcal{SH}\in \mathrm{CoSy}_k^c$. It is the initial object. 
\end{thm}

\begin{defi}
We let $\mathrm{QPCoSy}_k$ the $\infty$-category defined in the same terms as \Cref{defCoSys} but with functors $$\mathrm{Var}_k^\op \to \mathrm{CAlg}(\catinfty^\mathrm{st})$$ with source quasi-projective varieties instead of finite type $k$-schemes.
\end{defi}

\begin{prop}
	\label{extQProj}
	The restriction functor $$\mathrm{CoSy}_k\to \mathrm{QPCoSy}_k$$ is an equivalence of $\infty$-categories with quasi-inverse the right Kan extension. 
\end{prop}
\begin{proof}
	By \cite[Corollary 2.19]{GallauerIntroSixFF} or \cite[Proposition 7.13]{MR4560376}, any coefficient system (on all finite type $k$-schemes or varieties) has Nisnevich descent, thus Zariski descent. 
	Restriction and right Kan extension give an equivalence of categories between Zariski sheaves 
	$$\mathrm{Shv}_\mathrm{Zar}(\mathrm{Var}_k,\mathrm{CAlg}(\catinfty^\mathrm{st})) \simeq \mathrm{Shv}_\mathrm{Zar}(\Sch_k,\mathrm{CAlg}(\catinfty^\mathrm{st}))$$ because any finite type $k$-scheme admits a Zariski covering by affine schemes. Therefore it suffices to check that if $C\in \mathrm{QPCoSy}_k$, its right Kan extension to $\Sch^\op$ is an object of $\mathrm{CoSy}_k$, and that a morphism in $\mathrm{QPCoSy}_k$ gives a morphism in $\mathrm{CoSy}_k$ after Kan extension.

	We denote by $\dcal$ the right Kan extension of the functor 
	$$C : \mathrm{Var}_k^\op \to \mathrm{CAlg}(\catinfty)$$
	 along the inclusion $\mathrm{Var}_k^\op \to \Sch_k^\op$. Because $C$ is a Zariski sheaf, for $X\in\Sch_k$ we may choose a Zariski covering $p:U\to X$ with $U$ affine and as $p$ is quasi-projective, all $U_n := U^{\times_X {n+1}}$ are quasi-projective over $k$ and we have 
	 $$\dcal(X)\xrightarrow{\sim}\lim_\Delta C(U_n)$$
	 in $\mathrm{CAlg}(\catinfty)$. But then if $f: Y\to X$ is a smooth morphism, base change induces a map of simplicial schemes $f_\bullet:V_\bullet\to U_\bullet$ with $V_n = U_n\times_X Y$ and we also have $$\dcal(Y)\xrightarrow{\sim}\lim_\Delta C(V_n)$$
	 in $\mathrm{CAlg}(\catinfty)$. In fact, the functor 
	 $$f^*\colon \dcal(X)\to\dcal(Y)$$ is the limit of the functors 
	 $$f_n^*\colon C(U_n)\to C(V_n)$$
	 in the $\infty$-category $\mathrm{Fun}(\Delta^1,\catinfty)$. Because each $f_n^*$ has a left adjoint $(f_n)_\sharp$, and for every map $[n]\to [m]$ in $\Delta$ the square 
	 \begin{equation}\label{limfsharp}
		\begin{tikzcd}
			C(V_n)\ar[r,"(f_n)_\sharp"] \ar[d] & C(U_n)\ar[d]\\
			C(V_m)\ar[r,"(f_m)_\sharp"] & C(U_m)
		\end{tikzcd}
	\end{equation} commutes thanks to the canonical Beck-Chevalley morphism and the smooth base change property, the diagram $(f_n^*)_{[n]\in \Delta}$ in fact takes values in $\mathrm{Fun}^{\mathrm{LAd}}(\Delta^1,\catinfty)$ the $\infty$-category of left adjointable functors (see \cite[Definition 4.7.4.16]{lurieHigherAlgebra2022}). Thus by \cite[Corollary 4.7.4.18]{lurieHigherAlgebra2022} the limit $f^*$ is also a left adjointable functor: the left adjoint $f_\sharp$ exists and has a base change property with respect to all the functors $\dcal(X)\to C(U_n)$. Moreover, because $f^*$ is symmetric monoidal for $M\in \dcal(X)$ and $N\in\dcal(Y)$ there exists a canonical map $$f_\sharp(N\otimes f^*M )\to f_\sharp N\otimes M.$$ This map is an equivalence. Indeed if $M_n = M_{\mid U_n}$, we have a map of diagrams $(f_n^*)_n\to (f_n^*)_n$ given by the commutative squares
	$$
		\begin{tikzcd}
			\dcal(U_n) \arrow[r,"f_n^*"] \arrow[d,"-\otimes M_n",swap] 
			  & \dcal(V_n) \arrow[d,"-\otimes f_n^*M_n"] \\
			\dcal(U_n) \arrow[r,"f_n^*"]
			  & \dcal(V_n)
		\end{tikzcd}
	$$
	which, because of the projection formula for each $f_n$ and the base change for $(f_m)_\sharp$, is a map of diagrams with values in $\mathrm{Fun}^{\mathrm{LAd}}(\Delta^1,\catinfty)$. Thus at the limit the map $f^*\to f^*$ in $\mathrm{Fun}(\Delta^1,\catinfty)$ given by 
	$$f^*(-)\otimes M \Rightarrow f^*(-\otimes M)$$ is left adjointable, giving the projection formula.

	 Moreover, if $g\colon X'\to X$ is a map in $\Sch_k$ we can form the pullback
	 \[ 
		\begin{tikzcd}
			Y' \arrow[r,"f'"] \arrow[d,"g'",swap]
				\arrow[dr, phantom, very near start, "{ \lrcorner }"]
			  & X' \arrow[d,"g"] \\
			Y \arrow[r,"f"]
			  & X
		\end{tikzcd}
	 \]
	 and doing a base change to $U_\bullet$ we obtain a pullback square of simplicial objects in $\mathrm{Var}_k$ of the form
	 \[ 
		\begin{tikzcd}
			V'_\bullet \arrow[r,"b'"] \arrow[d,"a'",swap]
				\arrow[dr, phantom, very near start, "{ \lrcorner }"]
			  & U'_\bullet \arrow[d,"a"] \\
			V_\bullet \arrow[r,"b"]
			  & U_\bullet
		\end{tikzcd}.
	 \] Then as above, the commutative square 
	 $$
		\begin{tikzcd}
			\dcal(X) \arrow[r,"f^*"] \arrow[d,"g^*",swap] 
			  & \dcal(Y) \arrow[d,"(g')^*"] \\
			\dcal(X') \arrow[r,"(f')^*"]
			  & \dcal(Y')
		\end{tikzcd}
	 $$
	 can be seen as a map $f^*\to (f')^*$ in $\mathrm{Fun}(\Delta^1,\catinfty)$, which itself is a limit of maps $b'_n\to b_n$ in $\mathrm{Fun}(\Delta^1,\catinfty)$. As this last simplicial diagram of maps takes values in $\mathrm{Fun}^\mathrm{LAd}(\Delta^1,\catinfty)$ because of the smooth base change holds for maps in $\mathrm{Var}_k$, the limit is left adjointable, giving the smooth base for the pair $(f,g)$.

	 The existence of pushforwards and internal $\sHom$ is proven in a similar 
	 way to the existence of the $f_\sharp$ (using right adjointable squares instead of left adjointable). The only thing to check in order 
	 to obtain those functors as functors to a limit of $\infty$-categories is 
	 that over quasi-projective schemes these functors commute with pullbacks 
	 by Zariski covering so that they define a compatible system as in \Cref{limfsharp}. This is the case because these pullbacks $\pi^*$ are right 
	 adjoints to $\pi_\sharp$ and the smooth base change and the projection 
	 formula ensures the commutativity of $\pi_\sharp$ with the left adjoints 
	 of internal $\sHom$ and pushforwards. Localisation follows from the fact 
	 that the category $\mathrm{Sq}^\mathrm{cart}(\catinfty)$ of cartesian 
	 squares in $\catinfty$ is closed under limits in the category of 
	 commutative squares of $\infty$-categories. The properties of 
	 $\A^1$-invariance and $\mathrm{T}$-stability can be checked Zariski 
	 locally thanks to the smooth projection formula and the smooth base 
	 change, hence they hold. The same is true for extensions of morphisms of coefficient 
	 systems.
\end{proof}

\begin{rem}
The above proposition on extension of coefficient system from quasi-projective to finite type objects works \emph{verbatim} if one replaces $\Spec k$ with any noetherian finite dimensional scheme $B$. Moreover the proposition also holds with quasi-projective varieties replaced by separated reduced finite type $k$-schemes.
\end{rem}

Finally recall the following theorem:
\begin{thm}[Ayoub-Cisinski-Déglise-Röndings-Voevodsky]
	Any coefficient system has the six functors.	
\end{thm}

 \bibliographystyle{alpha}
 \bibliography{BibFinal}

\begin{thebibliography}{CGAdS17}

\bibitem[AGV22]{MR4466640}
Joseph Ayoub, Martin Gallauer, and Alberto Vezzani.
\newblock The six-functor formalism for rigid analytic motives.
\newblock {\em Forum Math. Sigma}, 10:Paper No. e61, 182, 2022.

\bibitem[And04]{MR2115000}
Yves Andr\'{e}.
\newblock {\em Une introduction aux motifs (motifs purs, motifs mixtes, p\'{e}riodes)}, volume~17 of {\em Panoramas et Synth\`eses [Panoramas and Syntheses]}.
\newblock Soci\'{e}t\'{e} Math\'{e}matique de France, Paris, 2004.

\bibitem[Aok23]{MR4549105}
Ko~Aoki.
\newblock Tensor triangular geometry of filtered objects and sheaves.
\newblock {\em Math. Z.}, 303(3):Paper No. 62, 27, 2023.

\bibitem[Ara13]{MR2995668}
Donu Arapura.
\newblock An abelian category of motivic sheaves.
\newblock {\em Adv. Math.}, 233:135--195, 2013.

\bibitem[Ara23]{MR4568787}
Donu Arapura.
\newblock Motivic sheaves revisited.
\newblock {\em J. Pure Appl. Algebra}, 227(8):Paper No. 107125, 22, 2023.

\bibitem[Ayo07]{MR2423375}
Joseph Ayoub.
\newblock Les six op\'{e}rations de {G}rothendieck et le formalisme des cycles \'{e}vanescents dans le monde motivique. {I}.
\newblock {\em Ast\'{e}risque}, (314):x+466, 2007.

\bibitem[Ayo10]{MR2602027}
Joseph Ayoub.
\newblock Note sur les op\'{e}rations de {G}rothendieck et la r\'{e}alisation de {B}etti.
\newblock {\em J. Inst. Math. Jussieu}, 9(2):225--263, 2010.

\bibitem[Ayo14a]{MR3205601}
Joseph Ayoub.
\newblock La r\'{e}alisation \'{e}tale et les op\'{e}rations de {G}rothendieck.
\newblock {\em Ann. Sci. \'{E}c. Norm. Sup\'{e}r. (4)}, 47(1):1--145, 2014.

\bibitem[Ayo14b]{MR3259031}
Joseph Ayoub.
\newblock L'alg\`ebre de {H}opf et le groupe de {G}alois motiviques d'un corps de caract\'{e}ristique nulle, {I}.
\newblock {\em J. Reine Angew. Math.}, 693:1--149, 2014.

\bibitem[Ayo22]{ayoubAnabelianPresentationMotivic2022}
Joseph Ayoub.
\newblock Anabelian presentation of the motivic {{Galois}} group in characteristic zero.
\newblock Preprint, available at \url{https://user.math.uzh.ch/ayoub/PDF-Files/Anabel.pdf}, 2022.

\bibitem[Ayo24]{AyoubWeil}
Joseph Ayoub.
\newblock Weil cohomology theories and their motivic hopf algebroids.
\newblock Preprint, available at \url{https://user.math.uzh.ch/ayoub/PDF-Files/motivic-connectivity.pdf}, 2024.

\bibitem[Bar21]{MR4325954}
Owen Barrett.
\newblock The derived category of the abelian category of constructible sheaves.
\newblock {\em Manuscripta Math.}, 166(3-4):419--425, 2021.

\bibitem[BBD82]{MR0751966}
A.~A. Beilinson, J.~Bernstein, and P.~Deligne.
\newblock Faisceaux pervers.
\newblock In {\em Analysis and topology on singular spaces, {I} ({L}uminy, 1981)}, volume 100 of {\em Ast\'{e}risque}, pages 5--171. Soc. Math. France, Paris, 1982.

\bibitem[BCKW24]{bunkeControlledObjectsLeftexact2019}
Ulrich Bunke, Denis-Charles Cisinski, Daniel Kasprowski, and Christoph Winges.
\newblock Controlled objects in left-exact \$\textbackslash infty\$-categories and the {{Novikov}} conjecture.
\newblock Preprint, available at \url{https://cisinski.app.uni-regensburg.de/unik.pdf}, 2024.

\bibitem[Bei87]{MR0923133}
A.~A. Beilinson.
\newblock On the derived category of perverse sheaves.
\newblock In {\em {$K$}-theory, arithmetic and geometry ({M}oscow, 1984--1986)}, volume 1289 of {\em Lecture Notes in Math.}, pages 27--41. Springer, Berlin, 1987.

\bibitem[Bei12]{MR2953406}
A.~Beilinson.
\newblock Remarks on {G}rothendieck's standard conjectures.
\newblock In {\em Regulators}, volume 571 of {\em Contemp. Math.}, pages 25--32. Amer. Math. Soc., Providence, RI, 2012.

\bibitem[BM21]{MR4278670}
Bhargav Bhatt and Akhil Mathew.
\newblock The arc-topology.
\newblock {\em Duke Math. J.}, 170(9):1899--1988, 2021.

\bibitem[Bon11]{bondarkoWeightsTstructuresGeneral2011}
Mikhail~V. Bondarko.
\newblock Weights and t-structures: In general triangulated categories, for 1-motives, mixed motives, and for mixed {{Hodge}} complexes and modules.
\newblock Preprint, available at \url{https://arxiv.org/abs/1011.3507}, 2011.

\bibitem[Bon15]{MR3347995}
Mikhail~V. Bondarko.
\newblock Mixed motivic sheaves (and weights for them) exist if `ordinary' mixed motives do.
\newblock {\em Compos. Math.}, 151(5):917--956, 2015.

\bibitem[CD16]{MR3477640}
Denis-Charles Cisinski and Fr\'{e}d\'{e}ric D\'{e}glise.
\newblock \'{E}tale motives.
\newblock {\em Compos. Math.}, 152(3):556--666, 2016.

\bibitem[CD19]{MR3971240}
Denis-Charles Cisinski and Fr\'{e}d\'{e}ric D\'{e}glise.
\newblock {\em Triangulated categories of mixed motives}.
\newblock Springer Monographs in Mathematics. Springer, Cham, [2019] \copyright 2019.

\bibitem[CGAdS17]{MR3649230}
Utsav Choudhury and Martin Gallauer Alves~de Souza.
\newblock An isomorphism of motivic {G}alois groups.
\newblock {\em Adv. Math.}, 313:470--536, 2017.

\bibitem[Cis19]{MR3931682}
Denis-Charles Cisinski.
\newblock {\em Higher categories and homotopical algebra}, volume 180 of {\em Cambridge Studies in Advanced Mathematics}.
\newblock Cambridge University Press, Cambridge, 2019.

\bibitem[Del80]{MR0601520}
Pierre Deligne.
\newblock La conjecture de {W}eil. {II}.
\newblock {\em Inst. Hautes \'{E}tudes Sci. Publ. Math.}, (52):137--252, 1980.

\bibitem[DG22]{MR4560376}
Brad Drew and Martin Gallauer.
\newblock The universal six-functor formalism.
\newblock {\em Ann. K-Theory}, 7(4):599--649, 2022.

\bibitem[Dre18]{drewMotivicHodgeModules2018}
Brad Drew.
\newblock Motivic {Hodge} modules, 2018.
\newblock Preprint, available at \url{https://bdrew.gitlab.io/pdf/DH.pdf}.

\bibitem[EHIK21]{MR4319065}
Elden Elmanto, Marc Hoyois, Ryomei Iwasa, and Shane Kelly.
\newblock Milnor excision for motivic spectra.
\newblock {\em J. Reine Angew. Math.}, 779:223--235, 2021.

\bibitem[Fak00]{fakhruddinNotesNoriLectures2000}
Najmuddin Fakhruddin.
\newblock Notes of {{Nori}}'s lectures on {{Mixed Motives}}.
\newblock {TIFR, Mumbai,Preprint}, 2000.

\bibitem[Gal22]{GallauerIntroSixFF}
Martin Gallauer.
\newblock An introduction to six-functor formalisms, 2022.
\newblock Available at \url{https://arxiv.org/abs/2112.10456}.

\bibitem[Har16]{harrerComparisonCategoriesMotives2016}
Daniel Harrer.
\newblock {\em Comparison of the {{Categories}} of {{Motives}} Defined by {{Voevodsky}} and {{Nori}}}.
\newblock PhD thesis, The University of Freiburg, 2016.
\newblock available at \url{https://d-nb.info/1122743106/34}.

\bibitem[Hin16]{MR3460765}
Vladimir Hinich.
\newblock Dwyer-{K}an localization revisited.
\newblock {\em Homology Homotopy Appl.}, 18(1):27--48, 2016.

\bibitem[HMS17]{MR3618276}
Annette Huber and Stefan M\"{u}ller-Stach.
\newblock {\em Periods and {N}ori motives}, volume~65 of {\em Ergebnisse der Mathematik und ihrer Grenzgebiete. 3. Folge. A Series of Modern Surveys in Mathematics [Results in Mathematics and Related Areas. 3rd Series. A Series of Modern Surveys in Mathematics]}.
\newblock Springer, Cham, 2017.
\newblock With contributions by Benjamin Friedrich and Jonas von Wangenheim.

\bibitem[HRS23]{MR4609461}
Tamir Hemo, Timo Richarz, and Jakob Scholbach.
\newblock Constructible sheaves on schemes.
\newblock {\em Adv. Math.}, 429:Paper No. 109179, 46, 2023.

\bibitem[Hub00]{MR1775312}
Annette Huber.
\newblock Realization of {V}oevodsky's motives.
\newblock {\em J. Algebraic Geom.}, 9(4):755--799, 2000.

\bibitem[Hé11]{MR2834728}
David Hébert.
\newblock Structure de poids \`a la {B}ondarko sur les motifs de {B}eilinson.
\newblock {\em Compos. Math.}, 147(5):1447--1462, 2011.

\bibitem[IM22]{ivorraFourOperationsPerverse2022}
Florian Ivorra and Sophie Morel.
\newblock The four operations on perverse motives.
\newblock Preprint, available at \url{http://perso.ens-lyon.fr/sophie.morel/PerverseMotives.pdf}, 2022.

\bibitem[Ivo16]{MR3518311}
Florian Ivorra.
\newblock Perverse, {H}odge and motivic realizations of \'{e}tale motives.
\newblock {\em Compos. Math.}, 152(6):1237--1285, 2016.

\bibitem[Ivo17]{MR3723805}
Florian Ivorra.
\newblock Perverse {N}ori motives.
\newblock {\em Math. Res. Lett.}, 24(4):1097--1131, 2017.

\bibitem[Iwa18]{MR3789830}
Isamu Iwanari.
\newblock Tannaka duality and stable infinity-categories.
\newblock {\em J. Topol.}, 11(2):469--526, 2018.

\bibitem[Lev05]{MR2181828}
Marc Levine.
\newblock Mixed motives.
\newblock In {\em Handbook of {$K$}-theory. {V}ol. 1, 2}, pages 429--521. Springer, Berlin, 2005.

\bibitem[Lura]{lurieKerodon}
Jacob Lurie.
\newblock Kerodon.
\newblock \url{https://kerodon.net/}.

\bibitem[Lurb]{lurieSpectralAlgebraicGeometry}
Jacob Lurie.
\newblock Spectral {{Algebraic Geometry}}.
\newblock available at \url{https://www.math.ias.edu/~lurie/papers/SAG-rootfile.pdf}.

\bibitem[Lur11]{lurieDerivedAlgebraicGeometry2011}
Jacob Lurie.
\newblock Derived {{Algebraic Geometry VII}}: {{Spectral Schemes}}, 2011.
\newblock available at \url{https://people.math.harvard.edu/~lurie/papers/DAG-VII.pdf}.

\bibitem[Lur22]{lurieHigherAlgebra2022}
Jacob Lurie.
\newblock Higher {{Algebra}}, 2022.
\newblock available at \url{https://people.math.harvard.edu/~lurie/papers/HA.pdf}.

\bibitem[LZ15]{liuGluingRestrictedNerves2015}
Yifeng Liu and Weizhe Zheng.
\newblock Gluing restricted nerves of \$\textbackslash infty\$-categories.
\newblock Preprint, available at \url{https://arxiv.org/abs/1211.5294}, 2015.

\bibitem[ML98]{MR1712872}
Saunders Mac~Lane.
\newblock {\em Categories for the working mathematician}, volume~5 of {\em Graduate Texts in Mathematics}.
\newblock Springer-Verlag, New York, second edition, 1998.

\bibitem[Mor19]{morelMixedAdicComplexes}
Sophie Morel.
\newblock Mixed -adic complexes for schemes over number fields.
\newblock Preprint, available at \url{http://perso.ens-lyon.fr/sophie.morel/sur_Q.pdf}, 2019.

\bibitem[MV99]{MR1813224}
Fabien Morel and Vladimir Voevodsky.
\newblock {${\bf A}^1$}-homotopy theory of schemes.
\newblock {\em Inst. Hautes \'{E}tudes Sci. Publ. Math.}, (90):45--143, 1999.

\bibitem[MW22]{martini_presentable_2022}
Louis Martini and Sebastian Wolf.
\newblock Presentable categories internal to an \${\textbackslash}infty\$-topos, 2022.

\bibitem[Nor02]{MR1940678}
Madhav~V. Nori.
\newblock Constructible sheaves.
\newblock In {\em Algebra, arithmetic and geometry, {P}art {I}, {II} ({M}umbai, 2000)}, volume~16 of {\em Tata Inst. Fund. Res. Stud. Math.}, pages 471--491. Tata Inst. Fund. Res., Bombay, 2002.

\bibitem[NRS20]{MR4093970}
Hoang~Kim Nguyen, George Raptis, and Christoph Schrade.
\newblock Adjoint functor theorems for {$\infty$}-categories.
\newblock {\em J. Lond. Math. Soc. (2)}, 101(2):659--681, 2020.

\bibitem[Rob14]{robaloThese}
Marco Robalo.
\newblock {\em Théorie homotopique motivique des espaces non-commutatifs.}
\newblock PhD thesis, University of Montpellier, 2014.

\bibitem[Rob15]{MR3281141}
Marco Robalo.
\newblock {$K$}-theory and the bridge from motives to noncommutative motives.
\newblock {\em Adv. Math.}, 269:399--550, 2015.

\bibitem[RS20]{MR4061978}
Timo Richarz and Jakob Scholbach.
\newblock The intersection motive of the moduli stack of shtukas.
\newblock {\em Forum Math. Sigma}, 8:Paper No. e8, 99, 2020.

\bibitem[RT24]{integralNori}
Raphaël Ruimy and Swann Tubach.
\newblock Nori motives (and mixed hodge modules) with integral coefficients.
\newblock Preprint, available at \url{https://swann.tubach.fr/en/research/}, 2024.

\bibitem[Sai90]{MR1047415}
Morihiko Saito.
\newblock Mixed {H}odge modules.
\newblock {\em Publ. Res. Inst. Math. Sci.}, 26(2):221--333, 1990.

\bibitem[Sai06]{saitoFormalismeMixedSheaves2006}
Morihiko Saito.
\newblock On the formalisme of mixed sheaves.
\newblock Preprint, available at \url{https://arxiv.org/abs/math/0611597}, 2006.

\bibitem[Ter24]{terenziTensorStructurePerverse2024}
Luca Terenzi.
\newblock Tensor structure on perverse {Nori} motives, 2024.
\newblock Available at \url{http://arxiv.org/abs/2401.13547}.

\bibitem[Tub24]{SwannMHM}
Swann Tubach.
\newblock Mixed hodge modules on stacks.
\newblock Preprint, available at \url{https://swann.tubach.fr/en/research/}, 2024.

\bibitem[Voe02]{MR1883180}
Vladimir Voevodsky.
\newblock Motivic cohomology groups are isomorphic to higher {C}how groups in any characteristic.
\newblock {\em Int. Math. Res. Not.}, (7):351--355, 2002.

\bibitem[VSF00]{MR1764197}
Vladimir Voevodsky, Andrei Suslin, and Eric~M. Friedlander.
\newblock {\em Cycles, transfers, and motivic homology theories}, volume 143 of {\em Annals of Mathematics Studies}.
\newblock Princeton University Press, Princeton, NJ, 2000.

\end{thebibliography}

Swann Tubach \url{swann.tubach@ens-lyon.fr}

E.N.S Lyon, UMPA, 46 Allée d'Italie, 

69364 Lyon Cedex 07, France
\end{document}